\documentclass[10pt]{amsart}
\usepackage{amsmath}
\usepackage{amssymb}
\usepackage{amsthm}
\usepackage{amsrefs}
\usepackage{dsfont}
\usepackage{mathrsfs}
\usepackage{stmaryrd}
\usepackage[all]{xy}
\usepackage{verbatim}  

\DeclareMathOperator{\id}{id}

\DeclareMathOperator{\Fun}{Fun}

\DeclareMathOperator{\im}{im}

\DeclareMathOperator{\colim}{colim}

\DeclareMathOperator{\hocolim}{hocolim}

\DeclareMathOperator{\ho}{ho}
\DeclareMathOperator{\ob}{ob}
\DeclareMathOperator{\mor}{mor}
\DeclareMathOperator{\op}{op}

\newcommand{\bC}{\mathbf{C}}

\newcommand{\bL}{\mathbf{L}}
\newcommand{\bM}{\mathbf{M}}
\newcommand{\bN}{\mathbf{N}}

\newcommand{\bR}{\mathbf{R}}

\newcommand{\bY}{\mathbf{Y}}
\newcommand{\bZ}{\mathbf{Z}}

\newcommand{\ba}{\mathbf{a}}
\newcommand{\bb}{\mathbf{b}}
\newcommand{\bc}{\mathbf{c}}
\newcommand{\bd}{\mathbf{d}}

\newcommand{\bm}{\mathbf{m}}
\newcommand{\bn}{\mathbf{n}}

\newcommand{\bbC}{\mathbb{C}}

\newcommand{\bbI}{\mathbb{I}}
\newcommand{\bbJ}{\mathbb{J}}

\newcommand{\bbL}{\mathbb{L}}
\newcommand{\bbM}{\mathbb{M}}
\newcommand{\bbN}{\mathbb{N}}
\newcommand{\bbO}{\mathbb{O}}
\newcommand{\bbP}{\mathbb{P}}
\newcommand{\bbQ}{\mathbb{Q}}
\newcommand{\bbR}{\mathbb{R}}

\newcommand{\bbU}{\mathbb{U}}

\newcommand{\sA}{\mathscr{A}}
\newcommand{\sB}{\mathscr{B}}
\newcommand{\sC}{\mathscr{C}}
\newcommand{\sD}{\mathscr{D}}
\newcommand{\sE}{\mathscr{E}}
\newcommand{\sF}{\mathscr{F}}

\newcommand{\sH}{\mathscr{H}}
\newcommand{\sI}{\mathscr{I}}

\newcommand{\sL}{\mathscr{L}}
\newcommand{\sM}{\mathscr{M}}

\newcommand{\sO}{\mathscr{O}}
\newcommand{\sP}{\mathscr{P}}

\newcommand{\sS}{\mathscr{S}}
\newcommand{\sT}{\mathscr{T}}
\newcommand{\sU}{\mathscr{U}}

\newcommand{\sX}{\mathscr{X}}
\newcommand{\sY}{\mathscr{Y}}

\newcommand{\cI}{\mathcal{I}}
\newcommand{\cJ}{\mathcal{J}}

\newcommand{\cO}{\mathcal{O}}

\newcommand{\sus}{\Sigma^{\infty}}
\newcommand{\loops}{\Omega^{\infty}}

\renewcommand{\phi}{\varphi}

\renewcommand{\emptyset}{\O}

\providecommand{\x}{\times}

\providecommand{\sma}{\wedge}

\renewcommand{\_}[1]{\underline{ #1 }}

\providecommand{\arr}{\longrightarrow}

\providecommand{\Aut}{\mathrm{Aut}}

\providecommand{\mGL}{\mathrm{GL}}

\newtheorem{theorem}{Theorem}[section]
\newtheorem{proposition}[theorem]{Proposition}
\newtheorem{lemma}[theorem]{Lemma}
\newtheorem{corollary}[theorem]{Corollary}

\theoremstyle{definition}
\newtheorem{definition}[theorem]{Definition}
\newtheorem{remark}[theorem]{Remark}
\newtheorem{construction}[theorem]{Construction}
\newtheorem{input_data}[theorem]{Input Data}

\title[Diagram Spaces, Diagram Spectra, and Spectra of Units]{Diagram Spaces, Diagram Spectra, and Spectra of Units}
\author[John A. Lind]{John A. Lind}
\address{Department of Mathematics\\ Johns Hopkins University}
\email{jlind@math.jhu.edu}

\begin{document}

\maketitle

\begin{abstract}
This article compares the infinite loop spaces associated to symmetric spectra, orthogonal spectra, and EKMM $S$-modules.  Each of these categories of structured spectra has a corresponding category of structured spaces that receives the infinite loop space functor $\Omega^{\infty}$.  We prove that these models for spaces are Quillen equivalent and that the infinite loop space functors $\Omega^{\infty}$ agree.  This comparison is then used to show that two different constructions of the spectrum of units $gl_1 R$ of a commutative ring spectrum $R$ agree.
\end{abstract}

\tableofcontents

\section{Introduction}

\addtocounter{equation}{-1}

In recent years, algebraic topology has witnessed the development of models for the stable homotopy category that are symmetric monoidal under the smash product.  Indeed, there are many such categories of spectra.  The models that we will consider are symmetric spectra \citelist{\cite{HSS} \cite{MMSS}}, orthogonal spectra \cite{MMSS} and (EKMM) $S$-modules \cite{EKMM}.  These categories of spectra, as well as their various categories of rings and modules, are known to be Quillen equivalent \citelist{\cite{MMSS} \cite{schwede_smodules} \cite{MM}}.  However, a Quillen equivalence only gives so much information.  On their own, these Quillen equivalences give no comparison of the infinite loop spaces associated to equivalent models of the same spectrum.  The present paper makes this comparison of infinite loop spaces.  

To describe the prototype of the objects under comparison, let $\sS$ denote the category of spectra $E$ whose zeroth spaces \emph{are} infinite loop spaces: these are sequences of spaces $E_{n}$ with homeomorphisms $E_{n} \arr \Omega E_{n + 1}$.  Then the infinite loop space associated to such a spectrum $E$ is its zeroth space: $\Omega^{\infty} E = E_0 \cong \Omega^{n} E_n$ for all $n$.  This defines a functor $\Omega^{\infty}$ from $\sS$ to the category $\sT$ of based spaces.  The suspension spectrum functor $\Sigma^{\infty}$ is the left adjoint of $\Omega^{\infty}$.  Composing with the free/forgetful adjunction between the category of unbased spaces $\sU$ and the category of based spaces $\sT$, we have the composite adjunction:
\begin{equation}
\xymatrix{ \sU \ar@<.5ex>[r]^{(-)_{+}} & \sT \ar@<.5ex>[r]^{\Sigma^\infty} \ar@<.5ex>[l] & \sS \ar@<.5ex>[l]^{\Omega^\infty}. }
\end{equation}
The main point of this paper is to define the analog of this adjunction for symmetric spectra, orthogonal spectra and $S$-modules, and then to show that these three adjunctions agree after passing to the homotopy category of spaces and the homotopy category of spectra.  Each model for spectra has a corresponding model for spaces that is symmetric monoidal with commutative monoids modeling $E_{\infty}$ spaces.  These models are Quillen equivalent and this equivalence is the enhancement of the Quillen equivalence of spectra needed to compare infinite loop space information.

The ``structured spaces'' associated to structured spectra are of considerable interest in their own right.  Since their commutative monoids model $E_{\infty}$ spaces, applying $\Omega^{\infty}$ to a commutative ring spectrum yields a space with a ``multiplicative'' $E_{\infty}$ space structure.  Multiplicative $E_{\infty}$ spaces, particularly as packaged in May's notion of an $E_{\infty}$ ring space, were central to the early applications of the theory of structured ring spectra (see \cite{RANT1} for the role of $E_{\infty}$ ring spaces in stable topological algebra and \cite{RANT3} for a history of these applications).  More recently, Rognes has developed the logarithmic algebra of structured ring spectra \cite{rognes_log}, which explicitly uses both the symmetric spectrum and EKMM approach to structured ring spectra and their multiplicative infinite loop spaces.  Our results not only prove that his constructions in the two contexts are equivalent, but show how to transport information between them.

Another example of the use of multiplicative infinite loop spaces is the spectrum of units $gl_1 R$ of a commutative ring spectrum $R$.  Spectra of units are essential to multiplicative orientation theory in geometric topology \cite{RANT3} and have been used more recently in Rezk's logarithmic power operations \cite{rezk_log} and the String orientation of $tmf$ \citelist{\cite{tmf_orientation} \cite{ABGHR}}.  To construct $gl_1 R$, the multiplication on $R$ is converted into the addition on the spectrum $gl_1 R$.  This transfer of structure occurs on the infinite loop space associated to $R$, and cannot be performed purely in terms of spectra.  In fact, there are two constructions of spectra of units: the first is the original definition of May, Quinn and Ray for $E_{\infty}$ ring spectra \cite{E_infty_rings} and the second is the definition for commutative symmetric ring spectra given by Schlichtkrull in \cite{schlichtkrull_units}.  Using the comparison of infinite loop spaces, we also prove that these two constructions give equivalent spectra as output.  

Let $\Sigma \sS$ be the category of symmetric spectra, $\sI \sS$ the category of orthogonal spectra, and $\sM_{S}$ the category of $S$-modules.  Let $\bbI$ be the category of finite sets and injective functions and let $\bbI \sU$ be the category of functors from $\bbI$ to the category $\sU$ of unbased spaces.  We call an object of $\bbI \sU$ an $\bbI$-space.  $\bbI \sU$ is the appropriate model for spaces corresponding to symmetric spectra; it receives the infinite loop space functor $\Omega^{\bullet}$ from symmetric spectra and participates in an adjunction
\begin{equation}
\xymatrix{ \bbI \sU \ar@<.5ex>[r]^{\Sigma^{\bullet}_+} &
 \Sigma \sS \ar@<.5ex>[l]^{\Omega^\bullet} }.
\end{equation}
Here, the subscript $+$ denotes adding a disjoint basepoint to each space in the diagram before taking the suspension spectrum.  We will work throughout with unbased spaces for flexibility and to simplify our later construction of $gl_1$.

In direct analogy to $\bbI$, let $\cI$ denote the category of finite dimensional real inner-product spaces and linear isometries (not necessarily isomorphisms).  Let $\cI \sU$ denote the category of continuous functors from $\cI$ to $\sU$.  Objects in this category, which we refer to as $\cI$-spaces, are the model for spaces corresponding to orthogonal spectra and we have an adjunction
\begin{equation}
\xymatrix{ \cI \sU \ar@<.5ex>[r]^{\Sigma^\bullet_{+}} & \sI \sS \ar@<.5ex>[l]^{\Omega^\bullet}. }
\end{equation}
We will prove that $\bbI$-spaces and $\cI$-spaces are models for topological spaces in the following sense.
\begin{theorem}  There is a model category structure on the category of $\bbI$-spaces whose weak equivalences are detected by the homotopy colimit functor
\[
\hocolim_{\bbI} \colon \bbI \sU \arr \sU.
\]
Similarly, the category of $\cI$-spaces has a model category structure with weak equivalences detected by
\[
\hocolim_{\cI} \colon \cI \sU \arr \sU.
\]
Furthermore, both of these model structures are Quillen equivalent to Quillen's model structure on $\sU$.
\end{theorem}
\noindent  The construction of the model structure is carried out for a general diagram category satisfying a few axioms and may be of independent interest.  The model structures will be used throughout the paper but their construction is delayed to \S\ref{model_structure_section}.  The Quillen equivalence with spaces follows from Theorems \ref{quillen_equiv_FCP} and \ref{quillen_equiv_If_spaces_L-spaces} (see Remark \ref{equiv_with_spaces}).  In the case of $\bbI$-spaces, this model structure has also been constructed by Schlichtkrull and Sagave \cite{ss}.

Moving on from diagram spaces, let $\sM_{*}$ denote the category of $*$-modules \cite{BCS}.  This is a symmetric monoidal model category Quillen equivalent to $\sU$ whose commutative monoids are equivalent to algebras for the linear isometries operad $\sL$ and thus model $E_{\infty}$ spaces.  $\sM_{*}$ is the model for spaces corresponding to $S$-modules and we have an adjunction:
\begin{equation}
\xymatrix{ \sM_{*} \ar@<.5ex>[r]^{\Sigma^{\infty}_{S +}} & \sM_{S} \ar@<.5ex>[l]^{\Omega^\infty_{S}}. }
\end{equation}
Readers of \cite{EKMM} should be warned that $\Omega^{\infty}_{S}$ is \emph{not} the usual functor $\Omega^{\infty}$ on $S$-modules.  Instead, $\Omega^{\infty}_{S}$ is a version of $\Omega^{\infty}$ that passes through the mirror-image categories $\sM^{S}$ and $\sM^{*}$.  This is necessary to yield $*$-modules as output and to make the resulting adjunction work correctly.

When we take the total derived functors of the three adjunctions (1) -- (3), the homotopy categories on the left and right sides are equivalent to the homotopy category of spaces and the stable homotopy category, respectively.  We prove that these adjunctions descend to homotopy categories compatibly, in the following sense.

\begin{theorem}\label{main_comparison_theorem}  The total derived versions of the adjunctions \emph{(1) -- (3)}, after passing along the equivalences of homotopy categories, are all isomorphic to the adjunction
\[
\xymatrix{ \ho \sU \ar@<.5ex>[r]^{\Sigma^\infty_{+}} & \ho \sS \ar@<.5ex>[l]^{\Omega^\infty}. }
\]
induced by the prototype adjunction (0).
\end{theorem}

To prove this, it suffices to prove that the four versions of the infinite loop space functor $\Omega^{\infty}$ agree on $\ho \sS$.  Since we are using different but Quillen equivalent models for spectra, we must incorporate the comparison functors between these models.  The same is true for our different models for spaces.  To make the comparison, we will use two intermediaries between orthogonal spectra and $S$-modules.  These are the category $\sS[\bbL]$  of $\bbL$-spectra and $\sM^{S}$, the mirror image to the category of $S$-modules.  There are analogs at the space level: $\sU[\bbL]$, the category of $\bbL$-spaces, and $\sM^{*}$, the mirror image to the category of $*$-modules.  We shall construct the following diagram relating all of these categories:
\begin{equation*}\label{main_spectrum_to_space_diagram}
\xymatrix{
\Sigma \sS \ar@<.5ex>[rr]^-{\bbP} \ar[dd]_{\Omega^{\bullet}} & &
\sI \sS \ar@<.5ex>[rr]^-{\bbN} \ar[dd]_{\Omega^{\bullet}} \ar@<.5ex>[ll]^-{\bbU} & &
\sS[\bbL] \ar@<.5ex>[ll]^-{\bbN^{\#}} \ar[dd]_{\Omega^{\infty}_{\bbL}} \ar@<.5ex>[rr]^-{F_{\sL}(S, - )} & &
\sM^{S} \ar@<.5ex>[ll]^-{r} \ar@<.5ex>[rr]^-{S \sma_{\sL} - } \ar[dd]_{\Omega^{\infty}_{\bbL}} & &
\sM_{S} \ar[dd]^{\Omega^{\infty}_{S}} \ar@<.5ex>[ll]^-{F_{\sL}( S , - )}
\\ & & &  & &  & & & \\
\bbI \sU \ar@<.5ex>[rr]^-{\bbP} & &
\cI \sU \ar@<.5ex>[rr]^-{\bbQ} \ar@<.5ex>[ll]^-{\bbU} & & 
\sU[\bbL] \ar@<.5ex>[ll]^-{\bbQ^{\#}} \ar@<.5ex>[rr]^-{F_{\sL}( * , - )} & &
\sM^{*} \ar@<.5ex>[rr]^-{* \boxtimes_{\sL} - } \ar@<.5ex>[ll]^-{r} & &
\sM_{*} \ar@<.5ex>[ll]^-{F_{\sL}(*, - )} }
\end{equation*}
The top row consist of models for the stable homotopy category and the bottom row consists of models for the homotopy category of spaces.  All parallel arrows are Quillen equivalences with left adjoints on top.  All vertical arrows are Quillen right adjoints.  By definition, the functors $\Omega^{\infty}_{\bbL}$ agree with the definition of $\Omega^{\infty}$ in the adjunction (0), and we prove Theorem \ref{main_comparison_theorem} by showing that the associated diagram of derived functors on homotopy categories commutes up to natural isomorphism.  The two squares on the right commute by construction, so this is accomplished by showing that (the derived versions of) the two left diagrams commute (Propositions \ref{comparison_prop_loops_diagram_sp} and \ref{loops_orthogonal_to_Lspectra_prop}, respectively).  

Each category in the top row of the main diagram carries a symmetric monoidal product $\sma$ whose monoids and commutative monoids give equivalent models for ring spectra and commutative ring spectra.  Each category in the bottom row also carries a symmetric monoidal product $\boxtimes$ and all of the adjunctions in the main diagram are symmetric monoidal.  An $\bbI$-space monoid under $\boxtimes$ is called an $\bbI$-FCP (functor with cartesian product) and an $\cI$-space monoid under $\boxtimes$ in $\cI$-spaces is called an $\cI$-FCP.  Strictly speaking, the categories $\sS[\bbL], \sU[\bbL], \sM^{S}$ and $\sM^{*}$ are only weakly symmetric monoidal, meaning that in general the unit maps are only weak equivalences.  However, we can still define monoids, commutative monoids, and we will use the usual language of monoidal categories (lax/strong symmetric monoidal functors, monoidal transformations, etc.).  As discussed in \cite{BCS}*{\S 4.2}, the category of monoids in $\sU[\bbL]$ is isomorphic to the category of $A_{\infty}$-spaces structured by the non-$\Sigma$ linear isometries operad $\sL$ and the category of commutative monoids in $\sU[\bbL]$ is isomorphic to the category of $E_{\infty}$-spaces structured by the linear isometries operad with symmetric group actions.  The following theorem shows that associative and commutative FCPs provide models for $A_{\infty}$ and $E_{\infty}$-spaces.  Given a symmetric monoidal category $\sC$, we write $\bM \sC$ for the category of monoids in $\sC$ and $\bC \sC$ for the category of commutative monoids in $\sC$.

\begin{theorem} The bottom row of Quillen equivalences in the main diagram restricts to give a chain of Quillen equivalences between the category $\bM \bbI \sU$ of $\bbI$-FCPs, the category $\bM \cI \sU$ of $\cI$-FCPs, and the category $\bM \sU[\bbL]$ of non-$\Sigma$ $\sL$-spaces.  These equivalences restrict further to give a chain of Quillen equivalences between the category $\bC \bbI \sU$ of commutative $\bbI$-FCPs, the category $\bC \cI \sU$ of commutative $\cI$-FCPs, and the category $\bC \sU[\bbL]$ of $\sL$-spaces.  
\end{theorem}
\noindent This is proved in Theorems \ref{quillen_equiv_FCP_L_spaces}, \ref{quillen_equiv_monoids} and \ref{equivalence_of_commutative_I_and_If_spaces}.  The infinite loop space functors $\Omega^{\bullet}$ are lax symmetric monoidal, so the underlying infinite loop spaces of (commutative) diagram ring spectra are (commutative) FCPs.  There is a version of the main diagram for ring spectra and for commutative ring spectra, given by passage to monoids and commutative monoids in all of the categories present.

In \S\ref{space_units_construction_section}, we construct a group-like FCP $\mGL_1^{\bullet} R$ of units from a diagram ring spectrum $R$.  This defines functors $\mGL_1^{\bullet} \colon \bM \Sigma \sS \arr \bM \bbI \sU$ from symmetric ring spectra to $\bbI$-FCPs and $\mGL_1^{\bullet} \colon \bM \sI \sS \arr \bM \cI \sU$ from orthogonal ring spectra to $\cI$-FCPs.  When $R$ is commutative, the FCP of units $\mGL_1^{\bullet}R$ is also commutative.  In \S\ref{spectrum_units_construction_section}, we build spectra out of group-like commutative FCPs, defining functors $gl_1 \colon \bC \Sigma \sS \arr \sS$ and $gl_1 \colon \bC \sI \sS \arr \sS$ that give the spectrum of units of a commutative diagram ring spectrum.  This agrees with the construction in \cite{schlichtkrull_units} for commutative symmetric ring spectra.

The construction of the units of $A_{\infty}$ and $E_{\infty}$ ring spectra has long been known \cite{E_infty_rings}: starting with an $A_{\infty}$ ring spectrum $R$, there is a group-like $A_{\infty}$ space $\mGL_1 R$ of units.  When $R$ is $E_{\infty}$ the group-like $E_{\infty}$ space $\mGL_1 R$ may be delooped to give a spectrum of units $gl_1 R$.  The categories $\bM \sM_S$ and $\bC \sM_S$ of $S$-algebras and commutative $S$-algebras arise as subcategories of the categories $\bM \sS[\bbL]$ and $\bC \sS[\bbL] = \sS[\sL]$ of $A_{\infty}$ and $E_{\infty}$ ring spectra, so we may also apply the functors $\mGL_1$ and $gl_1$ to them as well.  

\begin{theorem}  After passage to homotopy categories, all four versions of the $A_{\infty}$ space $\mGL_1 R$ of units of a ring spectrum $R$ agree.  After passage to homotopy categories, all four versions of the spectrum of units $gl_1 R$ of a commutative ring spectrum $R$ agree.
\end{theorem}
\noindent 
The first statement is proved as Propositions \ref{diagram_GL_1_comparison_prop} and \ref{GL_1_orthogonal_Einfty_prop}, and the second is proved as Theorems \ref{diagram_gl_comparison_theorem} and \ref{gl_comparison_orthogonal_Einfty_theorem}.  The proof requires the comparison of delooping machines, which seems to be intrinsically non-model theoretic.

\medskip

\emph{Outline.}  In \S2--3 we define $\bbI$-spaces, $\cI$-spaces and the adjunctions (1) and (2) involving the infinite loop space functors $\Omega^{\bullet}$ for symmetric and orthogonal spectra.  In \S4, we show that the adjunctions are monoidal and that commutative monoids in $\bbI$-spaces and $\cI$-spaces give rise to $E_{\infty}$-spaces and hence to infinite loop spaces.  \S5 contains the basic categorical technique used for all of the comparisons in the paper.  In \S6 we compare $\Omega^{\bullet}$ for symmetric and orthogonal spectra.  In \S7 we shift to the EKMM approach, describing $*$-modules and defining the infinite loop space functor $\Omega^{\infty}_{S}$ for $S$-modules.  In \S8 we construct the adjunction between $\cI$-spaces and $*$-modules and in \S9 we prove that it is a Quillen equivalence.  In \S10, we introduce the cylinder construction on $\cI$-FCPs and use it to make a comparison of $\sL$-spaces required in \S14.  The comparison of infinite loop spaces is completed by \S11, which gives the comparison of $\Omega^{\bullet}$ on orthogonal spectra with $\Omega^{\infty}_S$ on $S$-modules.  

In \S12 we define an FCP $\mGL_1^{\bullet} R$ of units associated to a diagram ring spectrum $R$ and in \S13 we convert commutative FCPs to spectra, thus defining the spectrum of units $gl_1 R$.  We compare $gl_1$ of symmetric and orthogonal ring spectra in \S14 and then compare with $gl_1$ of $E_{\infty}$ ring spectra in \S15.  

The rest of the paper contains the model-theoretic results that underly the comparison results.  In \S16 we construct the model structure on the category of $\sD$-spaces in full generality and in \S17 we prove that the model structures on $\bbI$-spaces and $\cI$-spaces are Quillen equivalent.  In \S18 we construct the model structure on the categories of monoids and commutative monoids in $\sD$-spaces (FCPs).  The appendix \S\ref{Topological Categories and the Bar Construction} gives background material on bar constructions and homotopy colimits over topological categories, as well as the action of the linear isometries operad $\sL$ on the categories and functors used throughout the paper.

\medskip

\emph{Conventions.}  In this paper, a topological category means a category internal to topological spaces, not just enriched in topological spaces (see \S\ref{Topological Categories and the Bar Construction} for a review).  The symbol $\sD$ will always denote a topological category.  The main examples in this paper are $\bbI$ (with the discrete topology) and $\cI^{\dagger}$, a small category equivalent to $\cI$.  The description of the topology on $\cI^{\dagger}$ is in \S\ref{Topological Categories and the Bar Construction}.  By a $\sD$-space we mean a continuous functor $\sD \arr \sU$.  We will often write $X_d$ for the value $X(d)$ of $X$ at an object $d$ of $\sD$.  Notice that a $\sD$-space only depends on the enrichment of $\sD$: the topology on $\ob \sD$ is not part of the structure of a $\sD$-space.  We will use the language of compactly generated model categories \cite{parametrized}*{4.5.3}.  A model category $\sC$ is topological if it is enriched, tensored and cotensored in $\sU$ and the topological analog of SM7 holds: given a cofibration $i \colon A \arr X$ and fibration $p \colon E \arr B$, the induced map of spaces
\[
\sC(i, p) \colon \sC(X, E) \arr \sC(A, E) \times_{\sC(A, B)} \sC(X, B)
\]
is a Serre fibration which is a weak equivalence if either $i$ or $p$ is.  

Although the following conditions are easy to verify in practice, we will make repeated use of them in glueing arguments and prefer to make a single definition for easy reference.  A topologically cocomplete category $\sC$ has cylinder objects defined by the tensor $I \times X$ of an object $X$ with the unit interval $I = [0, 1]$.  We have the notion of an $h$-cofibration in $\sC$, given by the homotopy extension property with homotopies defined by these cylinders.  Since $h$-cofibrations are defined by a lifting property, it is immediate that $h$-cofibrations are preserved under coproducts, pushouts and sequential colimits.  We say that a class of weak equivalences in a topologically bicomplete category $\sC$ is well-grounded (compare \cite{parametrized}*{5.4.1}) if the following properties hold:
\begin{itemize}
\item[(i)]  A coproduct of weak equivalences is a weak equivalence.
\item[(ii)] If $i \colon A \arr X$ is an $h$-cofibration and a weak equivalence and $f \colon A \arr Y$ is any map, then the cobase change $Y \arr X \cup_{A} Y$ is a weak equivalence.
\item[(iii)] If $i$ and $i'$ are $h$-cofibrations and the vertical arrows are weak equivalences in the following diagram
\[
\xymatrix{ 
X \ar[d] & A \ar[l]_{i} \ar[r] \ar[d] & Y \ar[d] \\
X' & A' \ar[l]^{i'} \ar[r] & Y' }
\]
then the induced map of pushouts $X \cup_{A} Y \arr X' \cup_{A'} Y'$ is a weak equivalence.
\item[(iv)] If $X$ and $Y$ are each colimits of sequences of $h$-cofibrations $X_n \arr X_{n + 1}$, $Y_n \arr Y_{n+1}$, and $f_n \colon X_n \arr Y_n$ is a compatible family of maps, each of which is a weak equivalence, then the induced map $\colim f_n \colon X \arr Y$ is a weak equivalence.  In particular, if each map $X_n \arr X_{n+1}$ is a weak equivalence, then $X_0 \arr X$ is a weak equivalence.
\item[(v)]  If $f \colon X \arr Y$ is a map in $\sC$ and $i \colon A \arr B$ is a retract of a relative CW complex, then the pushout product
\[
f \boxempty i \colon (X \times B) \cup_{X \times A} (Y \times A) \arr Y \times B
\]
is a weak equivalence if either $f$ or $i$ is a weak equivalence.  
\end{itemize}

Finally, a diagram of categories, such as the main diagram of this introduction, will be said to commute when we really mean ``commute up to natural isomorphism''.

\medskip

\emph{Acknowledgments.}  I thank Peter May for his guidance and enthusiasm, as well as Andrew Blumberg, Mike Mandell and Mike Shulman for helpful conversations.  I am very much in debt to a careful referee who caught serious errors in previous versions of this paper.


\renewcommand{\theequation}{\thesection.\arabic{theorem}}

\section{Infinite loop space theory of symmetric spectra}

In this section we will summarize the basic theory of $\bbI$-spaces (see also \citelist{\cite{rognes_log}*{\S6} \cite{ss}}).  Let us first recall the foundations on symmetric spectra from \citelist{\cite{HSS} \cite{MMSS} }.  Let $\Sigma$ be the category of finite sets $\bn = \{1, \dotsc, n\}$ for $n \geq 0$ and bijections $\bn \arr \bn$.  $\Sigma$ is a symmetric monoidal category under disjoint union.   The category $\Sigma \sT$ of based $\Sigma$-spaces is symmetric monoidal under the internal smash product $\sma$.  The sphere $\Sigma$-space $S \colon \bn \longmapsto S^n$ is a commutative monoid under $\sma$. A symmetric spectrum $E$ is a module over $S$ in the symmetric monoidal category $\Sigma \sT$.  The category $\Sigma \sS$ of symmetric spectra is symmetric monoidal under the smash product $\sma_{S}$ with unit object the sphere spectrum $S$.

In order to capture the infinite loop space information contained in a symmetric spectrum, we describe the structure naturally occuring in the collection of loop spaces $\{ \Omega^{n} E_{n} \}$.  To this end, let $\bbI$ be the category of finite sets $\bn = \{1, \dotsc, n\}$ for $n \geq 0$ with morphisms the injective set maps.  Let $\bbJ$ be the subcategory of $\bbI$ with the same objects but with morphisms only the inclusions $\iota \colon \bm \arr \bn$ such that $\iota(k) = k$.  Note that every morphism $\phi$ of $\bbI$ can be factored (non-uniquely) as $\phi = \overline{\phi} \circ \iota$ for some permutation $\overline{\phi}$ in $\Sigma$.  

Let $E$ be a symmetric spectrum.  Define a based $\bbI$-space $\Omega^{\bullet} E$ by $(\Omega^\bullet E)(\bn) = \Omega^n E_n$.  Given a morphism $\phi \colon \bm \arr \bn$ of $\bbI$, define $(\Omega^{\bullet}E)(\phi)$ as follows.  Write $\phi = \overline{\phi} \circ \iota$ where $\iota \colon \bm \arr \bn$ is the natural inclusion and $\overline{\phi} \in \Sigma_n$.  The induced map $\iota_* \colon \Omega^m E_m \arr \Omega^n E_n$ is $\Omega^m \tilde{\sigma}$, where $\tilde{\sigma}$ is the adjoint of the spectrum structure map $\Sigma^{n - m} E_m \arr E_n$.  Define $\overline{\phi}_* \colon \Omega^n E_n \arr \Omega^n E_n$ by sending $\gamma \in \Omega^n E_n$ to the composite:
\[
S^n \xrightarrow{\overline{\phi}^{-1}} S^n \overset{\gamma}{\arr} E_n \xrightarrow{E(\overline{\phi})} E_n.
\]
The induced map $\phi_* \colon \Omega^m E_m \arr \Omega^n E_n$ is defined to be $\overline{\phi}_* \circ \iota_*$ and is independent of the choice of $\overline{\phi}$.

The functor $\Omega^\bullet \colon \Sigma \sS \arr \bbI \sT$ has a left adjoint $\Sigma^\bullet$.  Given a based $\bbI$-space $X$, the symmetric spectrum $\Sigma^\bullet X$ is given by $(\Sigma^\bullet X)_n = X(\bn) \sma S^n$, with permutations acting diagonally.  The adjunction
\[
\xymatrix{ \bbI \sT \ar@<.5ex>[r]^{\Sigma^\bullet} & \Sigma \sS \ar@<.5ex>[l]^{\Omega^\bullet} }
\]
should be thought of as the symmetric spectrum analog of the usual $(\Sigma^\infty, \Omega^\infty)$ adjunction between based spaces and spectra.  The $\bbI$-space $\Omega^\bullet E$ is the appropriate notion of the infinite loop space associated to the symmetric spectrum $E$.  

From now on, we will work with unbased $\bbI$-spaces.  The usual adjunction between unbased and based spaces passes to diagram spaces, and we have the composite adjunction:
\[
\xymatrix{ \bbI \sU \ar@<.5ex>[r]^{(-)_+} &
\bbI \sT \ar@<.5ex>[r]^{\Sigma^\bullet} \ar@<.5ex>[l]  & \Sigma \sS \ar@<.5ex>[l]^{\Omega^\bullet} }.
\]
Denote the top composite by $\Sigma_{+}^{\bullet}$.

To understand the homotopy type that $\Omega^{\bullet} E$ determines, we combine  the spaces $\Omega^n E_n$ into a single space using the (unbased) homotopy colimit, which we denote by $\Omega^{\infty} E = \hocolim_{\bbI} \Omega^{\bullet} E.$

\begin{remark}\label{semistable_comment}
In general, $\pi_k \Omega^{\infty} E$ and $\pi_k E$ do not agree.  However, for $k \geq 0$, the $\pi_k \Omega^{\infty} E$ are the ``true'' homotopy groups of the symmetric spectrum $E$.  This is because $\Omega^{\infty} E$ is the zeroth space of Shipley's detection functor $D$  \cite{shipley_detection_functor} applied to $E$.   When $E$ is semistable in the sense of \cite{HSS}, then $\pi_k \Omega^{\infty} E$ is isomorphic to $\pi_k E$ for $k \geq 0$.  For more on semistable symmetric spectra and the nuances of the homotopy groups of symmetric spectra, see \cite{schwede_homotopy_groups}.
\end{remark}

Thinking of the homotopy colimit of an $\bbI$-space as determining its underlying homotopy type leads us to make the
\begin{definition}  A weak homotopy equivalence of $\bbI$-spaces is a map $f \colon X \arr Y$ such that the induced map of homotopy colimits
\[
f_* \colon \hocolim_{\bbI} X \arr \hocolim_{\bbI} Y
\]
is a weak homotopy equivalence of spaces.  
\end{definition}

\begin{theorem}  There is a compactly generated topological monoidal model structure on the category of $\bbI$-spaces with weak equivalences the weak homotopy equivalences.  The fibrations are level fibrations $f \colon X \arr Y$ such that for every morphism $\phi \colon \bm \arr \bn$ of $\bbI$, the induced map 
\[
X(\phi) \times f(\bm) \colon X(\bm) \arr X(\bn) \times_{Y(\bn)} Y(\bm)
\]
is a weak homotopy equivalence of spaces.  In particular, the fibrant objects are the $\bbI$-spaces $X$ such that $X(\phi) \colon X(\bm) \arr X(\bn)$ is a weak homotopy equivalence for every morphism $\phi \colon \bm \arr \bn$ of $\bbI$.  Furthermore the weak equivalences of $\bbI$-spaces are well-grounded.
\end{theorem}
\noindent The compactly generated topological model structure exists by Theorem \ref{general_model_cat_theorem}, and the pushout-product axiom is proved in \S\ref{model_structure_monoids_section}.  The last claim follows since the homotopy colimit functor preserves tensors with spaces, pushouts and sequential colimits, and weak homotopy equivalences of spaces are well-grounded.

For the following result and throughout, we will use the stable model structure on symmetric spectra \cite{MMSS}.

\begin{proposition}\label{symmetric_space_spectra_quillen_adjunction}
$(\Sigma^{\bullet}_{+}, \Omega^{\bullet})$ is a Quillen adjunction between $\bbI$-spaces and symmetric spectra.
\end{proposition}

\noindent Before proving this, we will prove:

\begin{lemma}\label{diagramloops_preserves_fibrants}  $\Omega^{\bullet}$ preserves fibrant objects.
\end{lemma}
\begin{proof}
Suppose that $E$ is a fibrant symmetric spectrum.  Then the maps $\tilde{\sigma} \colon E_n \arr \Omega E_{n + 1}$ are weak homotopy equivalences.  Let $\phi \colon \bm \arr \bn$ be any morphism of $\bbI$.  By the description of fibrant objects, we need to show that $(\Omega^{\bullet} E)(\phi) \colon \Omega^m E_m \arr \Omega^n E_n$ is a weak homotopy equivalence.  Factor $\phi$ as $\phi = \overline{\phi} \circ \iota$, where $\iota$ is the natural inclusion $\bm \subset \bn$ and $\overline{\phi}$ is a permutation.  As $\overline{\phi}$ induces a homeomorphism of spaces, we need only show that $(\Omega^{\bullet}E)(\iota) \colon \Omega^m E_m \arr \Omega^n E_n$ is a weak homotopy equivalence.  But $(\Omega^{\bullet}E)(\iota) = \Omega^m \tilde{\sigma}^{n - m}$, where $\tilde{\sigma}$ is the adjoint to the structure maps of the spectrum $E$.  Since $\tilde{\sigma}$ is a weak equivalence, the lemma is proved.
\end{proof}

\begin{proof}[Proof of Proposition \ref{symmetric_space_spectra_quillen_adjunction}]
It suffices to show that $\Omega^{\bullet}$ preserves fibrations and acyclic fibrations.  Suppose that $p \colon E \arr B$ is a fibration of symmetric spectra.  Then $p$ is a level fibration of symmetric spectra \cite{MMSS}*{\S9}, so each component $p_n \colon E_n \arr B_n$ is a fibration of spaces.  Thus $(\Omega^{\bullet} p)(\bn) \colon \Omega^n E_n \arr \Omega^n B_n$ is a fibration, so $\Omega^{\bullet} p$ is a level fibration of $\bbI$-spaces.  Next form the fiber $F$ of $p$ as the following pullback:
\[
\xymatrix{
F \ar[d] \ar[r] & E \ar[d]^{p} \\
\ast \ar[r] & B }
\]
Since $p$ has the right lifting property with respect to acyclic cofibrations, $F \arr *$ does as well, so $F$ is a fibrant symmetric spectrum.  By Lemma \ref{diagramloops_preserves_fibrants}, $\Omega^{\bullet} F$ is a fibrant $\bbI$-space.  Let $\phi \colon \bm \arr \bn$ be a morphism of $\bbI$.  Then the induced map $(\Omega^{\bullet}F)(\phi) \colon \Omega^m F_m \arr \Omega^n F_n$ is a weak homotopy equivalence.  Now consider the following diagram, where each vertical column is a fiber sequence of spaces:
\[
\xymatrix{
\Omega^m F_m \ar[r] \ar[d] & \Omega^n F_n \ar[r]^{=} \ar[d] & \Omega^n F_n \ar[d] \\
\Omega^m E_m \ar[r] \ar[d] & \Omega^{n}E_n \times_{\Omega^n B_n} \Omega^m B_m \ar[r] \ar[d]^{\pi} & \Omega^n E_n \ar[d]^{\Omega^{\bullet} p } \\
\Omega^m B_m \ar@{=}[r] & \Omega^m B_m \ar[r]_{(\Omega^{\bullet}B)(\phi)} & \Omega^n B_n }
\]
The lower right square is a pullback, so we may identify the fiber of $\pi$ with $\Omega^n F_n$ as indicated.  The top composite is  $(\Omega^{\bullet}F)(\phi) \colon \Omega^m F_m \arr \Omega^n F_n$.  It follows that the map of total spaces $\Omega^m E_m \arr \Omega^{n}E_n \times_{\Omega^n B_n} \Omega^m B_m$ is a weak homotopy equivalence.  By the description of fibrations in the stable model structure, this means that $\Omega^{\bullet} p  \colon \Omega^{\bullet} E \arr \Omega^{\bullet} B$ is a fibration of $\bbI$-spaces, so $\Omega^{\bullet}$ preserves fibrations.

Next we will show that $\Omega^{\bullet}$ preserves acyclic fibrations.  Suppose that $p \colon E \arr~B$ is an acylic fibration of symmetric spectra.  Then $\Omega^{\bullet} p$ is a fibration and we need to show that it is a weak homotopy equivalence.  By \cite{MMSS}*{9.8}, $p$ is a level equivalence 
of symmetric spectra.  Therefore each map $(\Omega^{\bullet} p)(\bn) \colon \Omega^n E_n \arr \Omega^n B_n$ is a weak homotopy equivalence, so the map of homotopy colimits 
\[
\Omega^{\bullet} p  \colon \hocolim_{\bbI} \Omega^{\bullet} E \arr \hocolim_{\bbI} \Omega^{\bullet} B
\]
is a weak homotopy equivalence, as desired.
\end{proof}

\section{Infinite loop space theory of orthogonal spectra}\label{Infinite loop space theory of orthogonal spectra}

We now make the analogous constructions for orthogonal spectra, referring to \cite{MMSS} for background.  Let $\sI$ be the category of finite dimensional inner-product spaces $V$ and linear isometric isomorphisms $V \arr W$.  $\sI$ is symmetric monoidal under direct sum $\oplus$.  The category $\sI \sT$ of based $\sI$-spaces is symmetric monoidal under the internal smash product $\sma$.  The sphere $\sI$-space $S \colon V \longmapsto S^V$ is a commutative monoid under $\sma$, and an orthogonal spectrum is a module over $S$.  The category $\sI \sS$ of orthogonal spectra is symmetric monoidal under $\sma_{S}$ with unit object the sphere spectrum $S$.

Let $\cI$ be the category of finite dimensional real inner-product spaces $V$ and linear isometries $V \arr W$ (not necessarily isomorphisms).  The category $\cI$ is the analog of the category $\bbI$ of finite sets $\bn$ and injections.   A universe $U$ is a real inner-product space that admits an isomorphism $U \cong \bR^{\infty}$.  Later on, we will use spectra indexed on $U$, as in \citelist{\cite{LMS} \cite{EKMM}}.  We fix, once and for all, a universe $U$.  Let $\cJ$ be the category of finite dimensional subspaces $V \subset U$ with morphisms the inclusions $V \subset W$ within $U$.  Notice that there is at most one morphism between any two objects in $\cJ$.  The category $\cJ$ is the analog of the category $\bbJ$ of finite sets $\bn$ and ordered inclusions.  We have inclusions of categories $\bbJ \arr \cJ$ and $\bbI \arr \cI$ defined by $\bn \longmapsto \bR^n$, where we choose a countable orthonormal basis of $U$ and use it to identify $\bR^n$ with a canonical $n$-dimensional subspace of $U$.     

There is an adjunction
\[
\xymatrix{ \cI \sT \ar@<.5ex>[r]^{\Sigma^\bullet} & \sI \sS \ar@<.5ex>[l]^{\Omega^\bullet}, }
\]
defined the same way as for symmetric spectra.  Given a based $\cI$-space $X$, $\Sigma^\bullet X$ is the orthogonal spectrum given by $(\Sigma^\bullet X)_V = X(V) \sma S^V$.  Given an orthogonal spectrum $E$, the $\cI$-space $\Omega^\bullet E$ is defined on objects by $(\Omega^\bullet E)(V) = \Omega^V E_V$ and on morphisms so that an isometry $\phi \colon V \arr W$ induces the map $\phi_* \colon \Omega^V E_V \arr~\Omega^W E_W$ that sends $\gamma \in \Omega^V E_V$ to:
\[
S^W \xrightarrow{\id \sma \phi^{-1}} S^{W - \phi(V)} \sma S^V \xrightarrow{\id \sma \gamma} S^{W - \phi(V)} \sma E_V \xrightarrow{\id \sma E(\phi)} S^{W - \phi(V)} \sma E_{\phi(V)} \xrightarrow{\Sigma} E_W.
\]
Notice that to define $E(\phi)$, we must consider $\phi$ as a linear isometric isomorphism onto its image $\phi(V)$ so that it is a morphism in $\sI$.

From now on, we will work with unbased $\cI$-spaces, and we have the composite adjunction:
\[
\xymatrix{ \cI \sU \ar@<.5ex>[r]^{(-)_+} &
\cI \sT \ar@<.5ex>[r]^{\Sigma^\bullet} \ar@<.5ex>[l] & \sI \sS \ar@<.5ex>[l]^{\Omega^\bullet}, }
\]

We want to form the topological homotopy colimit $\hocolim_{\cI} X$ of an $\cI$-space $X$ over $\cI$.  This homotopy colimit should take into account the topology of the space of objects and the space of morphisms in $\cI$.  However, $\cI$ is not a topological category because its class of objects is not a set.  Instead, we restrict to the equivalent small category $\cI^{\dagger}$ consisting of finite dimensional real inner product spaces $V$ that are a subspace of some finite product $U^n$ of the universe $U$.  The categories $\cI^{\dagger}$ and $\cJ$ are topological categories, but $\cI$ is not.  See \S\ref{Topological Categories and the Bar Construction} for the definitions of the topologies and the formation of topological homotopy colimits.  By abuse of notation, we will write $\hocolim_{\cI} X$ for the topological homotopy colimit $\hocolim_{\cI^{\dagger}} X$ of an $\cI$-space $X$ restricted to the category $\cI^{\dagger}$.

With these conventions in place, let $\Omega^{\infty} E = \hocolim_{\cI} \Omega^{\bullet} E$.  The following proposition relies on results from later sections, but we include it here since it is the analog for $\cI$-spaces of Remark \ref{semistable_comment}.

\begin{proposition}
For $k \geq 0$, there is a canonical isomorphism of homotopy groups $\pi_k \Omega^{\infty} E \cong \pi_k E$.
\end{proposition}
\begin{proof}
We have isomorphisms:
\[
\pi_k E = \underset{n \in \bbJ}{\colim} \; \pi_{k + n} E_{n} \cong \underset{n \in \bbJ}{\colim} \; \pi_k \Omega^{n} E_{n} \cong \pi_k \; \underset{n \in \bbJ}{\hocolim} \; \Omega^{n} E_{n}.
\]
The result follows by the following composite of weak homotopy equivalences induced by the inclusions of categories $\bbJ \arr \cJ \arr \cI^{\dagger}$ (Lemma \ref{theorem_A_hocolim_version} and Proposition \ref{untwisting_prop}):
\[
\hocolim_{\bbJ} \Omega^{\bullet} E \overset{\simeq}{\arr} \hocolim_{\cJ} \Omega^{\bullet} E \overset{\simeq}{\arr} \hocolim_{\cI} \Omega^{\bullet} E.
\]
\end{proof}
\begin{remark}
Unlike the case of symmetric spectra, orthogonal spectra always have an isomorphism $\pi_k \Omega^{\infty} E \cong \pi_k E$ for $k \geq 0$.  Another way to say this is that for both symmetric and orthogonal spectra, $\Omega^{\infty} E$ is the zeroth space of a fibrant replacement, but only for orthogonal spectra does a fibrant replacement always have the same homotopy groups as the original spectrum.  One way to understand this difference is that $\hocolim_{\bbJ}$ and $\hocolim_{\bbI}$ are only equivalent under certain hypotheses, as specified by B\"okstedt's telescope lemma (\ref{telescope}), but $\hocolim_{\cJ}$ and $\hocolim_{\cI}$ are always equivalent (Proposition \ref{untwisting_prop}).
\end{remark}

\begin{definition}  A weak homotopy equivalence of $\cI$-spaces is a map $f \colon X \arr Y$ such that the induced map of homotopy colimits
\[
f_* \colon \hocolim_{\cI} X \arr \hocolim_{\cI} Y
\]
is a weak homotopy equivalence of spaces.  
\end{definition}

\begin{theorem}\label{model_structure_If_spaces}  There is a compactly generated topological monoidal model structure on the category $\cI$-spaces with weak equivalences the weak homotopy equivalences.  The fibrations are level fibrations $f \colon X \arr Y$ such that for every morphism $\phi \colon V \arr W$ of $\cI$, the induced map 
\[
X(\phi) \times f(V) \colon X(V) \arr X(W) \times_{Y(W)} Y(V)
\]
is a weak homotopy equivalence of spaces.  In particular, the fibrant objects are the $\cI$-spaces $X$ such that $X(\phi) \colon X(V) \arr X(W)$ is a weak homotopy equivalence for every morphism $\phi \colon V \arr W$ of $\cI$.
\end{theorem}
\noindent The compactly generated topological model structure is a special case of Theorem \ref{general_model_cat_theorem} and the pushout-product axiom is proved in \S\ref{model_structure_monoids_section}.  The weak equivalences are well-grounded because the homotopy colimit functor commutes with tensors with spaces, pushouts, and sequential colimits.

The functor $\Omega^{\bullet}$ participates in a Quillen adjunction with the category of orthogonal spectra, which is endowed with the stable model structure \cite{MMSS}:

\begin{proposition}\label{orthogonal_space_spectra_quillen_adjunction} $
(\Sigma^{\bullet}_{+}, \Omega^{\bullet})$ is a Quillen adjunction between $\cI$-spaces and orthogonal spectra.
\end{proposition}

\noindent The proof is essentially identical to the proof of Proposition \ref{symmetric_space_spectra_quillen_adjunction} and will not be repeated.

\section{Functors with cartesian product and diagram ring spectra}

We now consider the multiplicative properties of the functor $\Omega^{\bullet}$.  Starting in full generality, let $(\sD, \oplus, 0)$ be a symmetric monoidal category enriched in spaces.  The category of unbased $\sD$-spaces $\sD \sU$ is symmetric monoidal under the internal cartesian product $\boxtimes$.  Given $\sD$-spaces $X$ and $Y$, $X \boxtimes Y$ is defined as the left Kan extension of the $(\sD \times \sD)$-space $X \times Y$ along $\oplus \colon \sD \times \sD \arr \sD$ and its universal property is described by the adjunction:
\[
\sD \sU (X \boxtimes Y, Z) \cong (\sD \times \sD)\sU (X \times Y, Z \circ \oplus).
\]
The unit of $\boxtimes$ is the represented $\sD$-space $\sD[0] = \sD(0, -)$.  When $0$ is the initial object, this is the terminal $\sD$-space $*$.  We will call a monoid in $\sD$-spaces under $\boxtimes$ a $\sD$-functor with cartesian product, abbreviated to $\sD$-FCP.  By the above adjunction, an FCP $X$ can be described internally in terms of an associative and unital map of $\sD$-spaces $X \boxtimes X \arr X$ or externally in terms of an associative and unital natural transformation $X(m) \times X(n) \arr X(m \oplus n)$.  A commutative monoid under $\boxtimes$ is called a commutative FCP.

Specializing to $\sD = \bbI$ and  $\cI$, we have a symmetric monoidal product $\boxtimes$ on $\bbI$-spaces and $\cI$-spaces.

\begin{lemma}  For both symmetric and orthogonal spectra, the functor $\Omega^{\bullet}$ is lax symmetric monoidal.
\end{lemma}
\begin{proof}
The functor $\Sigma^{\bullet}_{+}$ is strong symmetric monoidal by inspection.  As with any right adjoint of a strong symmetric monoidal functor, $\Omega^{\bullet}$ is lax symmetric monoidal with structure maps
\[
\Omega^{\bullet} E \boxtimes \Omega^{\bullet} E' \arr \Omega^{\bullet} (E \sma E') \quad \text{and} \quad * \arr \Omega^{\bullet}(S)
\]
defined to be the adjuncts of:
\[
\Sigma^{\bullet}_{+} ( \Omega^{\bullet} E \boxtimes \Omega^{\bullet} E') \overset{\cong}{\arr} \Sigma^{\bullet}_{+} \Omega^{\bullet} E \sma \Sigma^{\bullet}_{+} \Omega^{\bullet} E' \xrightarrow{\epsilon \sma \epsilon} E \sma E'
\]
and
\[
\Sigma^{\bullet}_{+} (\ast) \overset{\cong}{\arr} S.
\]
Here $\epsilon$ is the counit of the $(\Sigma^{\bullet}_{+}, \Omega^{\bullet})$ adjunction and the isomorphisms are the strong monoidal structure maps for $\Sigma^{\bullet}_{+}$.
\end{proof}

Lax symmetric monoidal functors preserve monoids and commutative monoids, so we have:
\begin{proposition}  If $R$ is a symmetric ring spectrum, then $\Omega^{\bullet} R$ is an $\bbI$-FCP.  If $R$ is an orthogonal ring spectrum, then $\Omega^{\bullet} R$ is an $\cI$-FCP.  If $R$ is commutative, then $\Omega^{\bullet} R$ is commutative.
\end{proposition}

\begin{remark}\label{hocolim_IFCP_is_E_infty}  If $X$ is a commutative $\sD$-FCP, then $\hocolim_{\sD} X$ is an $E_\infty$ space ($\sD = \bbI$ or $\cI$).  In particular, for a commutative symmetric or orthogonal ring spectrum $R$, $\Omega^{\infty} R = \hocolim \Omega^{\bullet} R$ is an $E_{\infty}$-space.  The $E_{\infty}$-space structure is encoded by an action of the (topological) Barratt-Eccles operad $E\Sigma$ with $E\Sigma(j) = E \Sigma_{j}$, the usual total space of the univeral $\Sigma_j$-bundle.  Indeed, $\hocolim_{\sD} X$ arises as the classifying space of the translation category $\sD[X]$, which in our situation can be made into a permutative topological category, and so $B \sD[X]$ carries an action of $E \Sigma$ \cite{E_infty_permutative_cats}*{4.9}.  
\end{remark}

\section{Comparing left and right Quillen functors}\label{left_right_adjoint_section}

Having set up the infinite loop space theory of diagram spectra, we are ready to start proving the comparison results.  This brief section is an overview of the method of comparison that we will use.  It is given in full generality, and is an easy observation about conjugation of adjoint functors.  It can be thought of as a special case of a general method for dealing with composites of left and right derived functors \cite{shulman_doubles}.

Let $(f, g) \colon \sA \arr \sB$ and $(h, k) \colon \sC \arr \sD$ be two Quillen adjunctions of model categories.  Further suppose that there are functors $a \colon \sA \arr \sC$ and $b \colon \sB \arr \sD$ that preserve fibrant objects and weak equivalences of fibrant objects.  For example, $a$ and $b$ might be right Quillen functors.  This situation can be pictured in the following diagram:
\addtocounter{theorem}{1}
\begin{equation}\label{situation}
\xymatrix{ 
\sA \ar@<.5ex>[r]^-{f} \ar[d]_-{a} & 
\sB \ar@<.5ex>[l]^-{g} \ar[d]^-{b} \\
\sC \ar@<.5ex>[r]^-{h}  & 
\sD \ar@<.5ex>[l]^-{k} }
\end{equation}
In the applications to follow, it will be easy to find a natural equivalence of composites $ag \simeq kb$.  We wish to understand how to transfer this equivalence to an equivalence in the other direction: $ha \simeq bf$.  In general, this will not be possible.  However, when $(f, g)$ and $(h, k)$ are Quillen equivalences, we can make a comparison of derived functors.

As left Quillen adjoints, $f$ and $h$ have left derived functors $\bL f \colon \ho \sA \arr \ho \sB $ and $\bL h \colon \ho \sC \arr \ho \sD$.  Similarly, $g$ and $k$ have right derived functors $\bR g \colon \ho \sB \arr \ho \sA$ and $\bR k \colon \ho \sD \arr \ho \sC$.  Since $a$ and $b$ preserve fibrant objects and weak equivalences of fibrant objects, they also have right derived functors $\bR a$ and $\bR b$, represented by $a R$ and $b R$, respectively, where $R$ denotes fibrant approximation.

\begin{proposition}\label{appendix_prop}
Suppose that $(f, g)$ and $(h, k)$ are Quillen equivalences in the situation pictured in diagram \eqref{situation}.  Further suppose that there is a natural equivalence of functors $ag \simeq kb$.  Then there is an isomorphism of composites of derived functors: $\bL h \bR a \cong \bR b \bL f$.
\end{proposition}
\begin{proof}
Since right derived functors compose to give right derived functors, we have isomorphisms $\bR a \bR g \cong \bR(a g) \cong \bR(kb) \cong \bR k \bR b$.  The units and counits of the derived adjunctions $(\bL f, \bR g)$ and $(\bL h, \bR k)$ are isomorphisms, yielding the desired isomorphism:
\[
\bL h \bR a \xrightarrow{\bL h \bR a \, \eta} \bL h \bR a \bR g \bL f \cong \bL h \bR k \bR b \bL f \overset{\epsilon}{\arr} \bR b \bL f.
\]
\end{proof}

\section{Comparison of infinite loop spaces of diagram spectra}

To compare symmetric spectra and orthogonal spectra, one uses the embedding of diagram categories $\Sigma \arr \cI$ defined by $\bn \mapsto \bR^n$.  This induces a forgetful functor $\bbU \colon \sI \sS \arr \Sigma \sS$ with a left adjoint $\bbP \colon \Sigma \sS \arr \sI \sS$ defined by left Kan extension.  Similarly, the embedding $\bbI \arr \cI$ induces a forgetful functor $\bbU \colon \cI \sU \arr \bbI \sU$ with left adjoint $\bbP \colon \bbI \sU \arr \cI \sU$ defined by left Kan extension.  Such left adjoints $\bbP$ are called prolongation functors.  We recall the comparison results for diagram spectra, then state the analogous results for $\bbI$ and $\cI$ spaces.

\begin{theorem}\cite{MMSS}  The adjunction $(\bbP, \bbU) \colon \Sigma \sS \rightleftarrows \sI \sS$ is a Quillen equivalence between the categories of symmetric spectra and orthogonal spectra, both considered with the stable model structure.  The adjunction restricted to the subcategories of monoids, respectively commutative monoids, gives a Quillen equivalence between symmetric ring spectra and orthogonal ring spectra, respectively commutative symmetric ring spectra and commutative orthogonal ring spectra.  
\end{theorem}

\noindent In the case of commutative ring spectra, we use the positive stable model structure.  There is also a positive model structure on diagram spaces, which is used for commutative FCPs.  The following result will be proved as Theorems \ref{equivalence_of_I_and_If_spaces} and \ref{equivalence_of_commutative_I_and_If_spaces}:

\begin{theorem}\label{quillen_equiv_FCP}  The adjunction $(\bbP, \bbU) \colon \bbI\sU \rightleftarrows \cI \sU$ is a Quillen equivalence between the categories of $\bbI$-spaces and $\cI$-spaces.  The adjunction restricted to the subcategories of monoids, respectively commutative monoids, gives a Quillen equivalence between $\bbI$-FCPs and $\cI$-FCPs, respectively commutative $\bbI$-FCPs and commutative $\cI$-FCPs.  
\end{theorem}

The infinite loop space functors $\Omega^{\bullet}$ fit into the following diagram:
\[
\xymatrix{
\Sigma \sS \ar@<.5ex>[r]^{\bbP} \ar[d]_{\Omega^{\bullet}} & \sI \sS \ar@<.5ex>[l]^{\bbU} \ar[d]^{\Omega^{\bullet}} \\
\bbI \sU \ar@<.5ex>[r]^{\bbP} & \cI \sU \ar@<.5ex>[l]^{\bbU} }
\]
Our goal is to show that this diagram commutes in both directions, up to weak equivalence.  For the direction involving the prolongation functors, we must compose left and right Quillen functors.  This requires descending to homotopy categories and derived functors, as discussed in \S \ref{left_right_adjoint_section}.

\begin{proposition}\label{comparison_prop_loops_diagram_sp}  The following diagrams commute:
\[
\xymatrix{
\Sigma \sS \ar[d]_{\Omega^{\bullet}} & \sI \sS \ar[l]_{\bbU} \ar[d]^{\Omega^{\bullet}} & & \ho \Sigma \sS \ar[r]^-{\bL \bbP} \ar[d]_{\bR \Omega^{\bullet}} & \ho \sI \sS  \ar[d]^{\bR \Omega^{\bullet}} \\
\bbI \sU  & \cI \sU \ar[l]^{\bbU} & &  \ho \bbI \sU \ar[r]_-{\bL \bbP} & \ho \cI \sU }
\]
\end{proposition}
\begin{proof}
The commutativity of the first diagram is immediate from the definitions.  The  adjunctions are Quillen equivalences, and the functors $\Omega^{\bullet}$ are right Quillen by Propositions \ref{symmetric_space_spectra_quillen_adjunction} and \ref{orthogonal_space_spectra_quillen_adjunction}.  Hence we are in the situation described in \S\ref{left_right_adjoint_section}, so the second diagram commutes by Proposition \ref{appendix_prop}.
\end{proof}

We also make the comparison of $\Omega^{\bullet}$ for ring and commutative ring spectra, parenthesizing the statement about commutative monoids.  We first record the following observation:

\begin{proposition}\label{diagram_loops_quillen_adjunction_comm}  Restricted to the categories of symmetric and orthogonal (commutative) ring spectra and (commutative) $\bbI$ and $\cI$-FCPs, $(\Sigma^{\bullet}_{+}, \Omega^{\bullet})$ is a Quillen adjunction.
\end{proposition}
\begin{proof}  This goes just as for the additive case in
Proposition \ref{symmetric_space_spectra_quillen_adjunction}.  For commutative ring spectra, we use the positive stable model structure on ring spectra and FCPs.
\end{proof}

In particular, Proposition \ref{appendix_prop} applies and we have:

\begin{proposition}\label{comparison_prop_loops_diagram_ring_sp}  The following diagrams \emph{(}and their analogs with $\bM$ replaced by $\bC$\emph{)} commute:
\[
\xymatrix{
\bM \Sigma \sS \ar[d]_{\Omega^{\bullet}} & \bM \sI \sS \ar[l]_{\bbU} \ar[d]^{\Omega^{\bullet}} & & \ho \bM \Sigma \sS \ar[r]^-{\bL \bbP} \ar[d]_{\bR \Omega^{\bullet}} & \ho \bM \sI \sS  \ar[d]^{\bR \Omega^{\bullet}} \\
\bM \bbI \sU  & \bM \cI \sU \ar[l]^{\bbU} & &  \ho \bM \bbI \sU \ar[r]_-{\bL \bbP} & \ho \bM \cI \sU }
\]
\end{proposition}
\begin{proof}
The first diagram is immediate from the definitions and the second follows from Proposition~\ref{appendix_prop}.
\end{proof}

\section{Infinite loop space theory of $S$-modules}\label{Infinite loop space theory of $S$-modules_section}

We recall what we need of the theory of $\bbL$-spectra and $S$-modules, referring to \cite{EKMM} as our primary source (see also \citelist{\cite{LMS}\cite{E_infty_rings}\cite{RANT1}}).  By a spectrum we mean a (LMS) spectrum $E$ indexed on a universe $U'$ of real inner product spaces such that the structure maps $E_{V} \arr \Omega^{W - V} E_{W}$ are homeomorphisms, as in \cite{LMS}.  We denote the category of spectra indexed on $U'$ by $\sS^{U'}$.  Recall that we have fixed a universe $U$, and we will work with the category $\sS = \sS^{U}$ of spectra indexed on $U$.  There is a monad $\bbL$ on spectra defined using the twisted half-smash product:
\[
\bbL E = \sL(1) \ltimes E.
\]
$\sL$ denotes the linear isometries operad with $\sL(j) = \cI_c(U^j, U)$.  
Here $\cI_c$ is the category of finite and countably infinite dimensional real inner product spaces and their linear isometries.  An $\bbL$-spectrum is an algebra for the monad $\bbL$ and we let $\sS[\bbL]$ denote the category of $\bbL$-spectra.  There is a smash product $\sma_{\sL}$ of $\bbL$-spectra defined in terms of a certain coequalizer involving the actions of $\sL(1)$.  The functor $(-)\sma_{\sL} E$ is left adjoint to the function $\sL$-spectrum functor $F_{\sL}(E, -)$.  $S$-modules in the sense of \cite{EKMM} are the $\bbL$-spectra $E$ for which the sphere spectrum acts as unit, meaning that the canonical weak equivalence $S \sma_{\sL} E \arr E$ is an isomorphism.  Denote the category of $S$-modules by $\sM_{S}$.  The functor $S \sma_{\sL} - \colon \sS[\bbL] \arr \sM_{S}$ is right adjoint to the forgetful functor $l \colon \sM_{S} \arr \sS[\bbL]$ and is left adjoint to the function $\sL$-spectrum functor $F_{\sL}(S, -) \colon \sM_{S} \arr \sS[\bbL]$.  

An $E_{\infty}$ ring spectrum is an $\bbL$-spectrum $R$ with an action of the linear isometries operad $\sL$.  Equivalently, $R$ is an algebra for the monad $\bC$ on $\bbL$-spectra defined by:
\[
\bC E = \bigvee_{j \geq 0} E^{\sma j} / \Sigma_{j},
\]
where $E^{\sma j}$ is the $j$-fold smash product $\sma_{\sL}$ of $E$.  Similarly, a commutative $S$-algebra is an algebra for the monad $\bC$ restricted to $S$-modules.  To avoid confusion, be warned that the monad $\bC$ is denoted by $\bbP$ in \cite{EKMM} and $\bbC$ is used for a different monad there.  Similarly, an $A_{\infty}$ ring spectrum is an algebra for the monad $\bM$ on $\bbL$-spectra defined by:
\[
\bM E= \bigvee_{j \geq 0} E^{\sma j}.
\]

We will use the following model structures on the categories of spectra, $\bbL$-spectra, and $S$-modules from \citelist{\cite{EKMM}*{\S VII.4} \cite{ABGHR}*{\S3.1}}:
\begin{itemize}
\item  There is a cofibrantly generated topological model structure on the category of spectra, with weak equivalences the weak homotopy equivalences of spectra and fibrations the Serre fibrations, i.e. the maps that are level-wise Serre fibrations of based spaces.  All spectra are fibrant in this model structure.
\item  There is a cofibrantly generated topological model structure on the category of $\bbL$-spectra with weak equivalences and fibrations created by the forgetful functor to spectra.  
\item There is a cofibrantly generated topological monoidal model structure on $S$-modules with weak equivalences and fibrations created by the forgetful functor to spectra.  
\end{itemize}

The underlying infinite loop space of a spectrum $E$ is the zero$^{th}$ space $\Omega^{\infty} E = E(0)$.  The functor $\Omega^{\infty}$ is right adjoint to the suspension spectrum functor $\sus \colon \sT \arr \sS$ \cite{LMS}*{\S I.4}.  Our goal is to study the analog of this adjunction for $S$-modules.

In order to capture the underlying infinite loop space of an $S$-module in a structured way, we will use a symmetric monoidal category Quillen equivalent to the category of topological spaces whose commutative monoids are $\sL$-spaces.   This category is constructed in direct analogy with $S$-modules, and was first carried out in \cite{blumberg_thesis}.  We will outline the basic definitions, referring to \citelist{ \cite{ABGHR} \cite{BCS}} for full details.  

Let $\bbL$ be the monad on unbased spaces defined by $\bbL X = \sL(1) \times X$.  An algebra for $\bbL$ is called an $\bbL$-space and the category of $\bbL$-spaces is denoted by $\sU[\bbL]$.  There is a weak symmetric monoidal product $\boxtimes_{\sL}$ on $\sU[\bbL]$ defined as the coequalizer
\[
\xymatrix{ \sL(2) \times \sL(1) \times \sL(1) \times X \times Y \ar@<-.5ex>[r] \ar@<.5ex>[r] & \sL(2) \times X \times Y \ar[r] & X \boxtimes_{\sL} Y}
\]
of the action of $\sL(1)^2$ on $\sL(2)$ by precomposition and the action of $\sL(1)^2$ on $X \times Y$.  The functor $(-)\boxtimes_{\sL} X$ is left adjoint to the function $\sL$-space functor $F_{\sL}(X, -)$.  An $\bbL$-space $X$ is a $*$-module if the canonical weak equivalence $\lambda \colon \ast \boxtimes_{\sL} X \arr X$ is an isomorphism.  The full subcategory of $\sU[\bbL]$ consisting of $\ast$-modules is denoted by $\sM_{*}$.   The functor $* \boxtimes_{\sL} - \colon \sU[\bbL] \arr \sM_{*}$ is right adjoint to the forgetful functor $l \colon \sM_{*} \arr \sU[\bbL]$ and is left adjoint to the function $\sL$-space functor $F_{\sL}(*, -) \colon \sM_{*} \arr \sU[\bbL]$.

A commutative monoid in the category of $\bbL$-spaces is an algebra for the monad:
\[
\bC X = \coprod_{j \geq 0} X^{\boxtimes j} / \Sigma_{j}.
\]
Here we have written $\boxtimes$ for $\boxtimes_{\sL}$.  Commutative monoids in $\bbL$-spaces are the same thing as $\sL$-spaces, meaning algebras for the linear isometries operad in unbased spaces.  A monoid in the category of $\bbL$-spaces is an algebra for the monad:
\[
\bM X = \coprod_{j \geq 0} X^{\boxtimes j}.
\]
Monoids in $\sL$-spaces are the same thing as non-$\Sigma$ $\sL$-spaces.

We will use the following model structures on $\bbL$-spaces and $\ast$-modules from \cite{BCS}*{4.15, 4.16}.  They are defined using the cofibrantly generated model structure on the category of topological spaces with weak equivalences the weak homotopy equivalences and fibrations the Serre fibrations.
\begin{itemize}
\item  There is a cofibrantly generated topological monoidal model structure on the category of $\bbL$-spaces with weak equivalences and fibrations created by the forgetful functor to spaces.  In particular, all objects are fibrant.  Colimits and limits of $\bbL$-spaces are constructed in the underlying category of spaces.  
\item There is a cofibrantly generated topological monoidal model structure on $*$-modules with weak equivalences and fibrations created by the forgetful functor to spaces.  Colimits are created in the category of $\bbL$-spaces, and limits are created by applying $\ast \boxtimes ( - )$ to the limit computed in the underlying category of spaces.
\end{itemize}
\noindent Notice that since colimits and weak equivalences are preserved by the forgetful functor to spaces, the weak equivalences are well-grounded for both $\bbL$-spaces and $\ast$-modules.

\begin{remark}\label{equiv_with_spaces}
As observed in \cite{BCS}, the category of $\bbL$-spaces and the category of $\ast$-modules are each Quillen equivalent to the category of topological spaces via the forgetful functor.  This is immediate from the construction of the model structures.  Although neither Quillen adjunction is monoidal, there is a different comparison functor from $\bbL$-spaces to spaces that is monoidal, see \cite{BCS}*{\S 4.5}.
\end{remark}

Let $X$ be a based $\bbL$-space (with trivial action of $\sL(1)$ on the basepoint).  The untwisting isomorphism \cite{EKMM}*{2.1} allows us to define an $\bbL$-spectrum structure on $\sus X$ using the $\bbL$-space structure on $X$:
\[
\sL(1) \ltimes \Sigma^{\infty} X \cong \Sigma^{\infty}(\sL(1)_{+} \sma X) \arr \Sigma^{\infty} X.
\]
Since $\sus$ preserves the coequalizers defining $\boxtimes_{\sL}$ and $\sma_{\sL}$, the untwisting isomorphism also gives a natural isomorphism $\sus(X \boxtimes_{\sL} Y)_+ \cong \sus X_+ \sma_{\sL} \sus Y_+$ for unbased $\bbL$-spaces $X$ and $Y$.  In fact, we may lift the $(\Sigma^{\infty}, \Omega^{\infty})$ adjunction to a symmetric monoidal adjunction between $\bbL$-spectra and $\bbL$-spaces.  

\begin{proposition}\cite{ABGHR}*{5.17}  The composable pair of Quillen adjunctions
\[
\xymatrix{
\sU  \ar@<.5ex>[r]^-{( - )_+} & 
\sT  \ar@<.5ex>[r]^-{\sus}  \ar@<.5ex>[l] &
\sS  \ar@<.5ex>[l]^-{\loops} }
\]
induces a pair of Quillen adjunctions between $\bbL$-spaces and $\bbL$-spectra:
\[
\xymatrix{
\sU[\bbL]  \ar@<.5ex>[r]^-{( - )_+} & 
\sT[\bbL]  \ar@<.5ex>[r]^-{\sus_{\bbL}}  \ar@<.5ex>[l] &
\sS[\bbL]  \ar@<.5ex>[l]^-{\loops_{\bbL}} }
\]
Furthermore, the composite of left adjoints $\sus_{\bbL} \circ ( - )_+ = \sus_{\bbL +}$ is strong symmetric monoidal and the right adjoint $\Omega^{\infty}_{\bbL}$ is lax symmetric monoidal.
\end{proposition}

\begin{remark}  The suspension spectrum $\Sigma^{\infty} X$ of an arbitrary based space $X$ carries a different $\bbL$-spectrum structure given by collapsing $\sL(1)$ to a point:
\[
\sL(1) \ltimes \Sigma^{\infty} X \cong \Sigma^{\infty} (\sL(1)_+ \sma X) \arr \Sigma^{\infty} (S^0 \sma X) \cong \Sigma^{\infty} X.
\]
To avoid confusion with this trivial $\bbL$-spectrum structure, we use the notation $\sus_{\bbL} X$ for the $\bbL$-spectrum associated to an $\bbL$-space.  
\end{remark}

Since $S = \sus_{\bbL +} (*)$, we have a natural isomorphism $\sus_{\bbL +}(* \boxtimes_{\sL} X) \cong S \sma_{\sL} \sus_{\bbL +} X$.  Therefore, the functor $\sus_{\bbL +}$ restricts to a functor $\sus_{S +} \colon \sM_* \arr \sM_S$.  Its right adjoint $\Omega^{\infty}_{S}$ will be the appropriate notion of the underlying infinite loop space of an $S$-module.  Notice that the composite $\Omega_{\bbL}^{\infty} \circ l$ of $\Omega_{\bbL}^{\infty}$ with the forgetful functor $l \colon \sM_S \arr \sS[\bbL]$ does not land in $*$-modules, and in particular is not the right adjoint of $\sus_{S +}$.  

In order to define the infinite loop space $\Omega_{S}^{\infty}$ of an $S$-module, we must be more careful.  The problem is that $\Omega_{\bbL}^{\infty} \circ l$ is the composite of a left adjoint and a right adjoint.  To remedy this, we will pass through the mirror image to the category of $S$-modules \cite{EKMM}*{\S II.2}.  This is the category $\sM^{S}$ of $\bbL$-spectra for which $S$ is a strict counit, meaning that the canonical weak equivalence $E \arr F_{\sL}(S, E)$ is an isomorphism.  The forgetful functor $r \colon \sM^{S} \arr \sS[\bbL]$ has as left adjoint the functor $F_{\sL}(S, -)$.  Furthermore, the resulting adjunction $(S \sma_{\sL} -, F_{\sL}(S, -) ) \colon \sM^{S} \arr \sM_{S}$ is an equivalence of categories.  We summarize this series of adjunctions in the following diagram, with left adjoints on top:
\[
\xymatrix{
\sS[\bbL]  \ar@<.5ex>[rr]^-{F_{\sL}(S, - )} & &
\sM^{S} \ar@<.5ex>[ll]^-{r} \ar@<.5ex>[rr]^-{S \sma_{\sL} - }  & &
\sM_{S}  \ar@<.5ex>[ll]^-{F_{\sL}( S , - )} }
\]
Similarly, there is a mirror image category $\sM^{*}$ to the category of $*$-modules.  Its objects are the $\bbL$-spaces for which the natural map $X \arr F_{\sL}(*, X)$ is an isomorphism.  We have adjunctions and an equivalence of categories between $\sM^*$ and $\sM_{*}$ just as for $S$-modules:
\[
\xymatrix{
\sU[\bbL]  \ar@<.5ex>[rr]^-{F_{\sL}(*, - )} & &
\sM^{*} \ar@<.5ex>[ll]^-{r} \ar@<.5ex>[rr]^-{* \boxtimes_{\sL} - }  & &
\sM_{*}  \ar@<.5ex>[ll]^-{F_{\sL}( * , - )} }
\]
A series of adjunctions shows that the natural isomorphism $\sus_{\bbL +}(* \boxtimes_{\sL} X) \cong S \sma_{\sL} \sus_{\bbL +} X$ induces a natural isomorphism $\loops_{\bbL} F_{\sL}(S, E) \cong F_{\sL}(*, \loops_{\bbL} E)$.  Therefore, the infinite loop space functor $\loops_{\bbL} \colon \sS[\bbL] \arr \sU[\bbL]$ restricts to a functor of mirror image categories $\loops_{\bbL} \colon \sM^{S} \arr \sM^{*}$.  

\begin{definition}\label{def_infinite_loop_of_Smodule}  Define the infinite loop $*$-module $\Omega^{\infty}_{S} M$ of an $S$-module $M$ by:
\[
\Omega^{\infty}_{S} M = * \boxtimes_{\sL} \Omega^{\infty}_{\bbL} F_{\sL} (S, M).
\]
\end{definition}

A series of adjunctions proves:

\begin{lemma}  The left adjoint of $\Omega^{\infty}_{S} \colon \sM_{S} \arr \sM_{*}$ is $\sus_{S +} = \sus_{\bbL +} \circ l \colon \sM_* \arr \sM_S$.  The left adjoint of $\Omega^{\infty}_{\bbL} \circ r \colon \sM^{S} \arr \sM^{*}$ is the functor 
\[
X \longmapsto F_{\sL} (S, \Sigma_{\bbL}^{\infty} (* \boxtimes_{\sL} X) ).
\]
\end{lemma}

\section{The adjunction between $\cI$-spaces and $*$-modules}

A Quillen equivalence between orthogonal spectra and $S$-modules is constructed in \cite{MM}.  Over the next two sections, we will construct an analogous Quillen equivalence between $\cI$-spaces and $*$-modules.  Since $*$-modules are equivalent to $\bbL$-spaces, it suffices to prove that $\cI$-spaces are equivalent to $\bbL$-spaces, which is accomplished in Theorem \ref{quillen_equiv_If_spaces_L-spaces}.  The equivalence restricts to a Quillen equivalence of commutative $\cI$-FCPs and $\sL$-spaces, giving two equivalent approaches to modeling infinite loop spaces.  \S\ref{Comparison of infinite loop spaces of orthogonal spectra and $S$-modules} will use this result to show that the infinite loop space functors for orthogonal spectra and $S$-modules agree.  We first summarize the Quillen equivalence of orthogonal spectra and $S$-modules.  

\begin{theorem}\cite{MM}*{1.1, 1.5}  There is a Quillen equivalence $(\bbN, \bbN^{\#}) \colon  \sI \sS \rightleftarrows \sS[\bbL]$ between orthogonal spectra and $\bbL$-spectra.  The functor $\bbN$ is strong symmetric monoidal and $\bbN^{\#}$ is lax symmetric monoidal.  The restricted adjunction between (commutative) orthogonal ring spectra and $A_{\infty}$ ($E_{\infty}$) ring spectra is also a Quillen equivalence.  Similarly, there is a Quillen equivalence $(\bbN_S, \bbN^{\#}_S) \colon  \sI \sS \rightleftarrows \sM_S$ between orthogonal spectra and $S$-modules that restricts to a Quillen equivalence between (commutative) orthogonal ring spectra and (commutative) $S$-algebras.
\end{theorem}

The functors we have denoted by $\bbN_{S}$ and $\bbN^{\#}_{S}$ are called $\bbN$ and $\bbN^{\#}$ in \cite{MM}.  The cited source only contains the comparison of orthogonal spectra and $S$-modules, but the equivalence of orthogonal spectra and $\bbL$-spectra follows by the equivalence between the categories of $\bbL$-spectra and $S$-modules.

We now make the central definition underlying the comparison between $\cI$-spaces and $\bbL$-spaces.  
\begin{definition}  Let $W$ be a real inner-product space of countable dimension.  The space of linear isometries $\cI_c(W, U)$ is an $\bbL$-space with the action 
\[
\bbL \cI_c(W, U) = \cI_c(U, U) \times \cI_c(W, U) \arr \cI_c(W, U)
\]
given by composition.  Define $\bbQ^{*} \colon \cI^{\op} \arr \sU[\bbL]$ by:
\[
\bbQ^*(V) = \cI_c(V \otimes U, U).
\]  
\end{definition}

\noindent By Hopkins' Lemma \cite{EKMM}*{I.5.4}, the $\bbL$-space $\cI_c((V \oplus W) \otimes U, U)$ is the coequalizer of the diagram defining $\cI_c(V \otimes U, U) \boxtimes_{\sL} \cI_c(W \otimes U, U)$.  This identification is natural and symmetric in $V$ and $W$, so we have proved the following:

\begin{lemma}  The functor $\bbQ^*$ is strong symmetric monoidal:
\[
\bbQ^*(V \oplus W) \cong \bbQ^*(V) \boxtimes_{\sL} \bbQ^*(W) \quad \text{and} \quad \bbQ^*(0) \cong *.
\]
\end{lemma}
The general theory of \cite{MM}*{\S1.2} now applies to the functor $\bbQ^*$.  Define functors $\bbQ \colon \cI \sU \arr \sU[\bbL]$ and $\bbQ^{\#} \colon \sU[\bbL] \arr \cI \sU$ by:
\[
\bbQ X = \bbQ^* \otimes_{\cI} X \quad \text{and} \quad (\bbQ^{\#} Y) (V) = \sU[\bbL](\bbQ^*(V), Y).
\]
Here $- \otimes_{\sD} -$ denotes the tensor product of functors (an enriched coend) over a topological category $\sD$.  See \S\ref{Topological Categories and the Bar Construction} for the definition in this context.  In particular, Proposition \ref{tensor_product_ok_prop} shows that $- \otimes_{\sD} -$ does not depend on the topology of $\ob \sD$, only on the enrichment of $\sD$.   The paper \cite{vogt} is a readable account of the case when $\ob \sD$ is discrete and both functors land in topological spaces.  For a more abstract setting and the relation to bar constructions, see \cite{shulman_hocolim}.  We will freely use basic properties of this construction as contained in these sources.  The results of \cite{MM}*{\S1.2} give us:

\begin{proposition}  The functor $\bbQ$ is left adjoint to $\bbQ^{\#}$.  Furthermore, $\bbQ$ is strong symmetric monoidal and $\bbQ^{\#}$ is lax symmetric monoidal.  Thus $(\bbQ, \bbQ^{\#})$ restricts to an adjunction between the categories of (commutative) $\cI$-FCPs and (commutative) monoids in $\bbL$-spaces.
\end{proposition}

The following diagram shows how the adjunction $(\bbQ, \bbQ^{\#})$ fits into the infinite loop space theory of orthogonal spectra, $\bbL$-spectra and $S$-modules:
\addtocounter{theorem}{1}
\begin{equation}\label{loops_orthogonal_to_Lspectra_diagram}
\xymatrix{
\sI \sS \ar@<.5ex>[rr]^-{\bbN} \ar[dd]_{\Omega^{\bullet}} & &
\sS[\bbL] \ar@<.5ex>[ll]^-{\bbN^{\#}} \ar[dd]_{\Omega^{\infty}_{\bbL}} \ar@<.5ex>[rr]^-{F_{\sL}(S, - )} & &
\sM^{S} \ar@<.5ex>[ll]^-{r} \ar@<.5ex>[rr]^-{S \sma_{\sL} - } \ar[dd]_{\Omega^{\infty}_{\bbL}} & &
\sM_{S} \ar[dd]^{\Omega^{\infty}_{S}} \ar@<.5ex>[ll]^-{F_{\sL}( S , - )}
\\
& & & & & & \\
\cI \sU \ar@<.5ex>[rr]^-{\bbQ}  & &
\sU[\bbL] \ar@<.5ex>[ll]^-{\bbQ^{\#}} \ar@<.5ex>[rr]^-{F_{\sL}( * , - )} & &
\sM^{*} \ar@<.5ex>[rr]^-{* \boxtimes_{\sL} - } \ar@<.5ex>[ll]^-{r} & &
\sM_{*} \ar@<.5ex>[ll]^-{F_{\sL}(*, - )} } 
\end{equation}
Each pair of arrows is a Quillen adjunction with left adjoint on top.  Each vertical arrow has a left adjoint and these vertical adjunctions are also Quillen adjunctions.  The rightmost square commutes (in both directions) by the definition of $\Omega^{\infty}_{S}$ and the fact that the top and bottom adjunctions are equivalences of categories.  The middle square also commutes (in both directions), and we will  prove in Proposition \ref{loops_orthogonal_to_Lspectra_prop} that the left square commutes up to natural weak equivalence.  

There is a direct construction of a functor $\bbQ_{*} \colon \cI \sU \arr \sM_{*}$ using the same method as the construction of $\bbQ$: choosing $\bbQ_*^*(V) = \ast \boxtimes_{\sL} \cI_c(V \otimes U, U)$ instead of $\cI_c(V \otimes U, U)$ lands in $*$-modules instead of $\bbL$-spaces.  The functor $\bbQ_* = \bbQ_*^* \otimes_{\cI} (-)$ has a right adjoint $\bbQ^{\#}_{*}$.  In fact, this adjunction of $\cI$-spaces and $*$-modules coincides with the composite adjunction in diagram \eqref{loops_orthogonal_to_Lspectra_diagram}.  The proof is a long series of adjunctions and will be omitted, but works equally well to compare $\bbN_{S}$ to $\bbN$.

\begin{lemma}  In diagram \eqref{loops_orthogonal_to_Lspectra_diagram}, the top horizontal composite $S \sma_{\sL} F_{\sL}(S, \bbN - )$ is naturally isomorphic to $\bbN_{S} \colon \sI \sS \arr \sM_{S}$.  Hence its right adjoint $\bbN^{\#} F_{\sL} (S, -)$ is naturally isomorphic to $\bbN^{\#}_{S}$.  The bottom horizontal composite $* \boxtimes_{\sL} F_{\sL} (*, \bbQ -)$ is naturally isomorphic to $\bbQ_{*}$ and its right adjoint $\bbQ^{\#} F_{\sL}(*, -)$ is naturally isomorphic to $\bbQ^{\#}_{*}$.
\end{lemma}

\section{The equivalence of $\cI$-spaces and $*$-modules}\label{section_equivalence_cI_spaces_L_spaces}

We now turn to the homotopical analysis of the functors $\bbQ$ and $\bbQ^{\#}$.  The main result of this section is Theorem \ref{quillen_equiv_If_spaces_L-spaces}, which states that $(\bbQ, \bbQ^{\#})$ is a Quillen equivalence.  

We will make serious use of the two sided bar construction $B(Y, \sD, X)$ built out of a topological category $\sD$ and functors $X \colon \sD \arr \sU$ and $Y \colon \sD^{\op} \arr \sU$.  The homotopy colimit of $X$ over $\sD$ is the special case of $Y = *$.  The bar construction acts as a derived version of the tensor product of functors $Y \otimes_{\sD} X$.  See \S A for the definition and basic properties of the two-sided bar construction over topological categories.  Recall from \S\ref{Infinite loop space theory of orthogonal spectra} that we use the equivalent topological category $\cI^{\dagger}$ instead of $\cI$ when taking the homotopy colimit of an $\cI$-space.  We make the same abuse of notation for bar constructions as we do for homotopy colimits and write $B(Y, \cI, X)$ for the bar construction defined using $\cI^{\dagger}$.

We will need the following results on spaces of isometries.

\begin{lemma}\label{mapping_space_equiv_lemma}  \hspace{2in}
\begin{itemize}  
\item[(i)]  Given a real inner-product space $V$, the space $\cI_c(V, U)$ of isometries is  contractible.  Furthermore, if $V$ is finite dimensional, then it is a CW complex.
\item[(ii)]   If $U'$ is a countably infinite dimensional real inner-product space, then the space $\cI_c(U', U)$ of linear isometries is a cofibrant $\bbL$-space.
\item[(iii)] A linear isometry $U' \arr U''$ of infinite dimensional real inner-product spaces induces a weak equivalence of mapping spaces:
\[
\sU[\bbL](\cI_c(U', U) , Y) \overset{\simeq}{\arr} \sU[\bbL] (\cI_c(U'', U), Y).
\]
Similarly, a weak equivalence $X \arr Y$ of $\bbL$-spaces induces a weak equivalence of mapping spaces
\[
\sU[\bbL](\cI_c(U', U), X) \overset{\simeq}{\arr} \sU[\bbL] (\cI_c(U', U), Y)
\]
\item[(iv)]  Let $f \colon X \arr Y$ be a fibration of $\bbL$-spaces and let $\phi \colon V \arr W$ be a morphism in $\cI$.  The map
\[
f_* \colon \sU[\bbL] (\bbQ^*(V), X) \arr \sU[\bbL](\bbQ^*(V), U), Y)
\]
is a Serre fibration and the induced map $\sU[\bbL]((\phi \otimes \id)^*, f)$
\[
\sU[\bbL](\bbQ^*(V), X) \arr \sU[\bbL](\bbQ^*(W), X) \times_{\sU[\bbL](\bbQ^*(W), Y) } \sU[\bbL](\bbQ^*(V), Y)
\]
is a weak homotopy equivalence.
\end{itemize}
\end{lemma}
\begin{proof}
The first claim in (i) is \cite{E_infty_rings}*{I.1.3}.  When $V$ is finite dimensional, $\cI_c(V, U)$ is the union over the finite dimensional subspaces $W \subset U$ of the Stiefel manifolds $\cI(V, W)$, hence is triangulable as a CW complex.  A choice of isomorphism $U \arr U'$ induces an isomorphism of $\bbL$-spaces $\cI_c(U', U) \cong \cI_c(U, U) = \bbL(*)$, and the latter is a generating cell for the model structure on $\bbL$-spaces.  This proves (ii), and (iii) follows because the model structure on $\bbL$-spaces is topological.  For (iv), the induced map $f_*$ is a Serre fibration because $\bbQ^*(V) = \cI_c(V \otimes U, U)$ is a cofibrant $\bbL$-space.  The induced maps $\sU[\bbL](\phi^*, X)$ and $\sU[\bbL](\phi^*, Y)$ are weak homotopy equivalences by (iii).  Since weak homotopy equivalences are preserved by pullbacks along Serre fibrations, it follows that the map $\sU[\bbL]((\phi \otimes \id)^*, f)$ is also a weak homotopy equivalence.
\end{proof}

The following lemma translates between the model-theoretic approach to homotopy theory and that based on the bar construction.

\begin{lemma}\label{cofibrant_colim}
Suppose that $X$ is a cofibrant $\cI$-space.  Then the projection 
\[
\pi \colon \hocolim_{\cJ} X \arr \colim_{\cJ} X
\] 
is a weak homotopy equivalence.
\end{lemma}
\begin{proof}  
We may assume that $X$ is an $FI$-cell complex, where $FI$ is the set of generating cofibrations for the level model structure (\S \ref{model_structure_section}):
\[
FI = \{ F_{V} i \mid  V \in \ob \cI^{\dagger}, \;  i \colon S^{n - 1} \arr D^n, n \geq 0 \}.
\]
Here $F_{V} \colon \sU \arr \cI \sU$ is the left adjoint to evaluation at the object $V$ of $\cI$ and is calculated by $(F_{V}A)(W) = \cI(V, W) \times A$.  Note than an inclusion $W \subset W'$ of inner product spaces induces a closed inclusion $(F_{V}A)(W) \arr (F_{V}A)(W')$.  Since the pushouts and sequential colimits defining cell complexes are calculated levelwise, we also have a closed inclusion $X(W) \arr X(W')$.

Now restrict $X$ along the inclusions of categories $\bbJ \arr \cJ \arr \cI$.  The homotopy colimit over $\bbJ$ is homotopy equivalent to the usual mapping telescope, and we have just argued that each map $\bm \arr \mathbf{m + 1}$ of $\bbJ$ induces a closed inclusion $X(\bm) \arr X(\bm + \mathbf{1})$.  It follows that the top horizontal map in the following commutative diagram is a weak homotopy equivalence:
\[
\xymatrix{
\hocolim_{\bbJ} X \ar[r] \ar[d] & \colim_{\bbJ} X \ar@{=}[d] \\
\hocolim_{\cJ} X \ar[r] & \colim_{\cJ} X .}
\]
The vertical map of homotopy colimits is induced by the inclusion of categories $\bbJ \arr \cJ$, and is a weak homotopy equivalence by Lemma \ref{theorem_A_hocolim_version}.  Therefore the bottom horizontal arrow is a weak homotopy equivalence as desired.
\end{proof}

\begin{lemma}\label{contractible_isometries_bar}  The $\cI^{\op}$-space $B(*, \cJ, \cI)$ is level-wise contractible.
\end{lemma}
\begin{proof}  Evaluated at $V$, the bar construction is:
\[
B(*, \cJ, \cI(V, -)) = \underset{W \in \cJ}{\hocolim} \, \cI(V, W).
\]
The represented $\cI$-space $\cI(V, -)$ is the evaluation $F_{V}(*)$ of $F_V$ on a point.  Thus $\cI(V, -)$ is an $F I$-cell complex, and in particular is cofibrant.  By Lemma \ref{cofibrant_colim}, the homotopy colimit over $\cJ$ is weak homotopy equivalent to the colimit:
\[
\underset{W \in \cJ}{\hocolim} \, \cI(V, W)  \overset{\simeq}{\arr} \underset{W \in \cJ}{\colim} \, \cI(V, W) = \cI_c(V, U).
\]
The contractibility of the space of isometries $\cI_c(V, U)$ finishes the proof.
\end{proof}

\begin{proposition}\label{untwisting_prop}  Let $X$ be an $\cI$-space.  The inclusion of categories $\cJ \arr \cI$ induces a natural homotopy equivalence of spaces:
\[
\hocolim_{\cJ} X \overset{\simeq}{\arr} \hocolim_{\cI} X.
\]
\end{proposition}
\begin{proof}
The map in question is the bottom horizontal map in the following diagram of spaces:
\[
\xymatrix{
B(*, \cJ, B(\cI, \cI, X)) \ar[dr] \ar[rr]^{\cong} \ar[dd]_{B(\id, \id, \epsilon)} & & B(B(*, \cJ, \cI), \cI, X) \ar[d]  \\
& B(*, \cI, B(\cI, \cI, X)) \ar[r]^{\cong} \ar[dr]_{B(\id, \id, \epsilon)} & B(B(*, \cI, \cI), \cI, X) \ar[d]^{B(\epsilon, \id, \id)} & \\
B(*, \cJ, X) \ar[rr] & & B(*, \cI, X) & }
\]
The horizontal isomorphisms are the canonical interchange maps for iterated bar constructions.  The other unmarked arrows are induced by the inclusion $\cJ \arr \cI$.  The bottom right triangle commutes up to homotopy as an instance of Lemma \ref{appendix_homotopy_diagram} with $Y = *$.  The other regions commute by naturality.  The map of $\cI$-spaces $\epsilon \colon B(\cI, \cI, X) \arr X$ is a level-wise deformation retract, hence the left vertical map is a homotopy equivalence.  The right vertical composite is induced by the map of $\cI^{\op}$-spaces
\[
B(*, \cJ, \cI) \arr B(*, \cI, \cI) \overset{\epsilon}{\arr} *,
\]
which is equal to the level-wise homotopy equivalence $B(*, \cJ, \cI) \arr *$ of Lemma \ref{contractible_isometries_bar}.  Thus the right vertical composite is also a homotopy equivalence and the lemma follows.
\end{proof}

\begin{remark}
The preceding proposition is the first step in connecting the homotopy theory of $\cI$-spaces to the homotopy theory of $\bbL$-spaces.  It is an interesting fact that the analogs of \ref{contractible_isometries_bar} and \ref{untwisting_prop} do \emph{not} hold for $\bbI$-spaces.  That is, $B(*, \bbJ, \bbI)$ is not contractible.  This difference between the symmetric and orthogonal contexts is related to the fact that $\pi_*$-isomorphisms and weak homotopy equivalences do not coincide for symmetric spectra, and do coincide for orthogonal spectra.
\end{remark}

We will now define a variant of $\bbQ$ that gives rise to the colimit functor $\colim_{\cJ}$ and compare it to $\bbQ$.  The definitions are in direct analogy with the two different functors $\bbM$ and $\bbN$ from orthogonal spectra to $S$-modules \cite{MM}*{\S1.7}.  Define $\bbO^* \colon \cI^{\op} \arr \sU[\bbL]$ by $\bbO^*(V) = \cI_c(V, U)$.  Then the tensor product of functors $\bbO X = \bbO^* \otimes_{\cI} X$ defines a functor $\bbO \colon \cI \sU \arr \sU[\bbL]$.  The map 
\begin{align*}
\sL(2) \times_{\sL(1)^2} \cI_c(V, U) \times \cI_c(W, U) &\arr \cI_c(V \oplus W) \\
(\gamma, f, g) &\longmapsto \gamma \circ(f \oplus g)
\end{align*}
defines a natural map $\bbO^*(V) \boxtimes_{\sL} \bbO^*(W) \arr \bbO^*(V \oplus W)$ making $\bbO^*$, and hence $\bbO$, a lax symmetric monoidal functor.  In particular, $\bbO$ takes commutative FCPs to $\sL$-spaces.  We will now show that the functor $\bbO$ is the original construction of an $\sL$-space from a commutative $\cI$-FCP described in \cite{E_infty_rings}:

\begin{lemma}\label{O_is_colim}  Let $X$ be an $\cI$-space.  Then there is a natural isomorphism $\bbO X \cong \colim_{\cJ} X$.  If $X$ is a commutative $\cI$-FCP then this is an isomorphism of $\sL$-spaces.
\end{lemma}
\begin{proof}
We may write $\cI_c(V, U) = \colim_{W \in \cJ} \cI(V, W)$ as the coend $\ast \otimes_{\cJ} \cI(V, -)$.  The isomorphism follows:
\[
\bbO X = \bbO^* \otimes_{\cI} X = \ast \otimes_{\cJ} \cI \otimes_{\cI} X \cong \ast \otimes_{\cJ} X = \colim_{\cJ} X.
\]
Comparing the definition of the monoidal structure maps of $\bbO^*$ with the $\sL$-space structure maps on $\colim_{\cJ} X$ in \cite{E_infty_rings}*{I.1.6} verifies the second claim.
\end{proof}

A choice of one dimensional subspace of $U$ determines an inclusion $V \arr V \otimes U$, which induces a natural transformation
\[
\xi^* \colon \bbQ^*(V) = \cI_c(V \otimes U, U) \arr \cI_c(V, U) = \bbO^*(V).
\]
Write $\xi = \xi^* \otimes_{\cI} (-) \colon \bbQ \arr \bbO$ for the induced natural transformation.

\begin{lemma}\label{xi_prop}  The natural transformation $\xi \colon \bbQ \arr \bbO$ is symmetric monoidal.  If $X$ is a cofibrant $\cI$-space, then $\xi \colon \bbQ X \arr \bbO X \cong \colim_{\cJ} X$ is a weak homotopy equivalence of $\bbL$-spaces.
\end{lemma}
\begin{proof}  Checking the definitions of the monoidal structure maps for $\bbQ^*$ and $\bbO^*$ verifies that $\xi^*$ is a symmetric monoidal natural transformation.  We may assume that $X$ is a $FI$-cell complex and induct up the cellular filtration.  Hence it suffices to shows that $\xi$ is a weak homotopy equivalence on $\cI$-spaces of the form $F_{V} K = \cI(V, -) \times K$ with $K$ a CW complex.  The functor $\bbQ$ commutes with tensors with spaces and on represented $\cI$-spaces takes the form $\bbQ (\cI(V, -)) = \bbQ^*(V)$, so we have a natural isomorphism
\[
\bbQ F_{V} K \cong \cI_c(V \otimes U, U) \times K.
\]
On the other hand, we have the natural isomorphism  $\bbO F_{V} K \cong \cI_c(V, U) \times K$ from Lemma \ref{O_is_colim}.  Since both space of isometries are contractible, it follows that $\xi \colon \bbQ F_{V} K \arr \bbO F_{V} K$ is a weak homotopy equivalence.
\end{proof}

\begin{remark}  Note that the identification of $\bbO X$ with $\colim_{\cJ} X$ is canonical, but $\xi$ is not canonical: it requires a choice of one dimensional subspace of $U$.  The functor $\bbO$ has a right adjoint $\bbO^{\#}$, but this adjunction does not appear to be monoidal, and $(\bbO, \bbO^{\#})$ does not appear to be a Quillen adjunction.  The difficulty is in showing that the $\bbL$-space $\bbO^*(V) = \cI_{c}(V, U)$ is cofibrant.
\end{remark}

We are now ready to prove the main result of this section:

\begin{theorem}\label{quillen_equiv_If_spaces_L-spaces}  The adjunction $(\bbQ, \bbQ^{\#})$ is a Quillen equivalence between the categories of $\cI$-spaces and $\bbL$-spaces.  
\end{theorem}
\begin{proof}  Using the characterization of fibrations of $\cI$-spaces given in Theorem \ref{model_structure_If_spaces}, it is an immediate consequence of Lemma \ref{mapping_space_equiv_lemma}.(iii) and (iv) that the functor $\bbQ^{\#}$ preserves fibrations and takes weak equivalences to level-wise weak homotopy equivalences.  Hence $(\bbQ, \bbQ^{\#})$ is a Quillen adjunction.

Now we turn to the Quillen equivalence.  Suppose that $X$ is a cofibrant $\cI$-space.  Extend $X$ to a functor defined on $\cI_c$ by taking the left Kan extension of $X$ along the inclusion of categories $\cI \arr \cI_c$.  This means that the value of $X$ on an infinite dimensional inner product space $W$ is computed as $\colim_{V \subset W} X(V)$, where $V$ runs over the finite dimensional subspaces of $W$.  In particular, $X(U) = \colim_{\cJ} X$.  

Choose a one dimensional subspace of $U$.  This gives a linear isometry $i_V \colon V \arr V \otimes U$ for each object $V$ of $\cI$, and a natural transformation $\xi^* \colon \bbQ^* \arr \bbO^*$.  By Lemma \ref{xi_prop}, the induced map $\xi \colon \bbQ X \arr \bbO X \cong X(U)$ is a weak equivalence of $\bbL$-spaces.

Now consider an object $V$ of $\cJ$.  The inclusions $V \subset U$ and $V \otimes U \subset U \otimes U$, and the maps $i_V$ and $i_U = \colim_V i_V$ give rise to the following commutative diagram:
\[
\xymatrix{
\sU[\bbL]( \cI_c(V \otimes U, U), \bbQ X ) \ar[rr]^{\bbQ^{\#} \xi}_{\simeq} & & \sU[\bbL](\cI_{c}(V \otimes U, U), X(U) ) \ar[ddd]^{\simeq} \\
X(V) \ar[u]^{\eta_V} \ar[r]^{i_V} \ar[d] & X(V \otimes U) \ar[d] \ar[ur] & \\
X(U) \ar[r]_{i_U} \ar[d]_{\cong} & X(U \otimes U) \ar[dr] & \\
\sU[\bbL]( \cI_c(U, U), X(U) ) \ar[rr]^-{\simeq}_-{i_U} & & \sU[\bbL]( \cI_c(U \otimes U, U), X(U) )
}
\]
Here, $\eta$ is the unit for the adjunction $(\bbQ, \bbQ^{\#})$.  The diagonal arrows and the isomorphism $X(U) \arr \sU[\bbL]( \cI_c(U, U), X(U) )$ are all the adjuncts of evaluation maps for the functor $X$.  The lower and right trapezoids commute by the naturality of those adjunctions.  A diagram chase involving the definition of the unit $\eta$ verifies that the upper trapezoid commutes.  The maps of mapping spaces are weak equivalences as indicated by Lemma \ref{mapping_space_equiv_lemma}.(ii).  Now take the homotopy colimit of this diagram over $V \in \cJ$.  The induced map $\hocolim_{\cJ} X \arr \hocolim_{\cJ} X(U)$ is a weak equivalence by the following commutative diagram:
\[
\xymatrix{
\hocolim_{\cJ} X \ar[r] \ar[d]_{\simeq} & \hocolim_{\cJ} X(U) \ar[d] \ar[r]^-{\cong} & X(U) \times B \cJ \ar[d]^{\simeq}  \\
\colim_{\cJ} X \ar[r]^-{\cong} & \colim_{\cJ} X(U) \ar[r]^-{\cong} & X(U)}
\]
Here the vertical maps are the canonical projections to the colimit and are weak equivalences by Lemma \ref{cofibrant_colim} and the homotopy equivalence $B \cJ \simeq *$, respectively.  Returning to the main diagram, it follows that the homotopy colimit of the unit $\eta$ is a weak equivalence as well.  Since homotopy colimits over $\cJ$ and $\cI$ are homotopy equivalent (Proposition \ref{untwisting_prop}), we have shown that the unit $\eta \colon X \arr \bbQ^{\#}\bbQ X$ is a weak equivalence for cofibrant $\cI$-spaces $X$.  Since all $\bbL$-spaces are fibrant, this implies that $(\bbQ, \bbQ^{\#})$ is a Quillen equivalence \cite{hovey_model_categories}*{1.3.16}.
\end{proof}

The equivalence between $\cI$-spaces and $\bbL$-spaces induces an equivalence at the level of monoids and commutative monoids.

\begin{theorem}\label{quillen_equiv_FCP_L_spaces}
Restricting to monoids, respectively commutative monoids, $(\bbQ, \bbQ^{\#})$ induces a Quillen equivalence between the categories of $\cI$-FCPs and non-$\Sigma$ $\sL$-spaces, respectively commutative $\cI$-FCPs and $\sL$-spaces.
\end{theorem}
\begin{proof}
In the case of monoids, the theorem follows from the fact that cofibrant $\cI$-FCPs are cofibrant as $\cI$-spaces.  While this is not true of cofibrant commutative $\cI$-FCPs, we will prove in Proposition \ref{general_equiv_cofibrant_FCPs_via_induction}.(iii)   that  $\xi \colon \bbQ X \arr \bbO X$ is a weak homotopy equivalence for $X$ a cofibrant commutative $\cI$-FCP.  The proof of Theorem \ref{quillen_equiv_If_spaces_L-spaces} then goes through to show that the unit of the adjunction $(\bbQ, \bbQ^{\#})$ is a weak homotopy equivalence on cofibrant commutative FCPs.  This proves the result in the case of commutative monoids.
\end{proof}

\section{Comparison of infinite loop spaces of orthogonal spectra and $S$-modules}\label{Comparison of infinite loop spaces of orthogonal spectra and $S$-modules}

We will now use the Quillen equivalence of $\cI$-spaces and $\bbL$-spaces to prove that the infinite loop space functors of orthogonal spectra, $\bbL$-spectra and $S$-modules agree.  

Let $X$ be a based space.  Given a universe $U'$ and a subspace $V \subset U'$, the shift desuspension functor $\Sigma^{U'}_{V} (-)$ is left adjoint to the functor $(-)_{V} \colon \sS^{U'} \arr \sT$ from spectra indexed on $U'$ to based spaces given by evaluation at $V$.  When $V = 0$, $\Sigma^{U'}_{0} = \Sigma^{U'}$ is the suspension spectrum functor for spectra indexed on $U'$.  For our privileged universe $U$, $\Sigma^{U}$ is the suspension spectrum functor $\Sigma^{\infty}$.  The twisted half-smash product $\cI_c(U', U) \ltimes \Sigma^{U'} X$ is naturally an $\bbL$-spectrum via the action of $\sL(1) = \cI_c(U, U)$ on $\cI_c(U', U)$ by composition.  This action of $\sL(1)$ on $\cI_c(U', U)$ also gives $\cI_c(U', U)_+ \sma X$ the structure of an $\bbL$-space.  These $\bbL$ actions are related by the untwisting isomorphism \cite{EKMM}*{2.1}:
\begin{equation}\label{untwisting_iso}  
\cI_c(U', U) \ltimes \Sigma^{U'} X \cong \Sigma_{\bbL}^{U} (\cI_c(U', U)_+ \sma X).
\end{equation}

We may now make the comparison:

\begin{proposition}\label{loops_orthogonal_to_Lspectra_prop}  The following diagrams commute:
\[
\xymatrix{
\sI \sS \ar[d]_{\Omega^{\bullet}} & \sS[\bbL] \ar[l]_{\bbN^{\#}} \ar[d]^{\Omega^{\infty}_{\bbL}} & &\ho \sI \sS \ar[r]^-{\bL \bbN} \ar[d]_{\bR \Omega^{\bullet}} & \ho \sS[\bbL]  \ar[d]^{\bR \Omega^{\infty}_{\bbL}} \\
\cI \sU  & \sU[\bbL] \ar[l]^-{\bbQ^{\#}} & &  \ho \cI \sU \ar[r]_-{\bL \bbQ} & \ho \sU[\bbL] }
\]
\end{proposition}
\begin{proof}
In \cite{MM}, the definition of the functor $\bbN$ depends on a choice of a one-dimensional subspace $\bR \subset U$.  Different choices give rise to (non-canonically) isomorphic functors, so we are free to fix a choice $\bR \subset U$ for the remainder of the proof.  We may identify a finite dimensional inner product space $V$ with the subspace $V \otimes \bR$ of the universe $V \otimes U$.  

We will make the comparison of right adjoints first.  Unraveling definitions, we have:
\begin{align*}
\Omega^{\bullet} \bbN^{\#} E (V) &= \Omega^{V} \sS[\bbL] (\cI_c(V \otimes U, U) \ltimes \Sigma_{V}^{V \otimes U} S^0, E) \\
&\cong \sS[\bbL] (\cI_c(V \otimes U, U) \ltimes (S^V \sma \Sigma_{V}^{V \otimes U} S^0 ), E ) \\
&\cong \sS[\bbL] (\Sigma_{\bbL_{+}}^{\infty} \cI_c(V \otimes U, U), E) \\
&\cong \sU[\bbL] (\cI_c(V \otimes U, U), \Omega^{\infty}_{\bbL} E) = \bbQ^{\#} \Omega_{\bbL}^{\infty} E(V).
\end{align*}
The first isomorphism follows since tensors with spaces are preserved by left adjoints, such as the twisted half-smash product.  For the second isomorphism, notice that $S^V \sma \Sigma_{V}^{V \otimes U} S^0 \cong \Sigma^{V \otimes U} S^0$, and then apply the untwisting isomorphism \eqref{untwisting_iso}.  This gives the diagrams on the left.  The diagrams on the right follow by the uniqueness of adjoints.
\end{proof}

The two rightmost squares of diagram \eqref{loops_orthogonal_to_Lspectra_diagram} commute strictly, so we may immediately deduce the comparison of infinite loop spaces between orthogonal spectra and $S$-modules from Proposition \ref{loops_orthogonal_to_Lspectra_prop}.

\begin{proposition}\label{loops_orthogonal_to_Smodules_prop}  The following diagrams commute:
\[
\xymatrix{
\sI \sS \ar[d]_{\Omega^{\bullet}} & \sM_{S} \ar[l]_{\bbN_{S}^{\#}} \ar[d]^{\Omega^{\infty}_{S}} & & \ho \sI \sS \ar[r]^-{\bL \bbN_{S}} \ar[d]_{\bR \Omega^{\bullet}} & \ho \sM_{S}  \ar[d]^{\bR \Omega^{\infty}_{S}} \\
\cI \sU  & \sM_{*} \ar[l]^-{\bbQ_{*}^{\#}} & &  \ho \cI \sU \ar[r]_-{\bL \bbQ_{*}} & \ho \sM_{*} }
\]
\end{proposition}

\noindent This concludes the comparison of infinite loop space functors and the proof of Theorem \ref{main_comparison_theorem}.

\section{The space of units of a diagram ring spectrum}\label{space_units_construction_section}

Let $R$ be a commutative diagram ring spectrum.  In the next two sections, we will define the spectrum of units $gl_1 R$.  The idea is that by taking the stably invertible components of the commutative FCP $\Omega^{\bullet} R$, we have a group-like commutative FCP $\mGL_1^{\bullet} R$.  This is accomplished in the current section.  We then convert commutative FCPs into spectra in \S\ref{spectrum_units_construction_section}.  In the case of symmetric spectra, the construction is due to Schlichtkrull \cite{schlichtkrull_units}.

\medskip

The forgetful functor from groups to monoids has a right adjoint $M \mapsto M^{\times}$, where $M^{\times}$ is the submonoid of invertible elements of $M$.  We will now make the analogous construction for $\sD$-FCPs, where $\sD = \bbI$ or $\cI$.  Let $X$ be a $\sD$-FCP.  The space $\hocolim_{\sD} X$ inherits a topological monoid structure from the FCP multiplication $\mu$ on $X$ and the permutative structure $\oplus$ of $\sD$:
\[
(\hocolim_{\sD} X)^2 = \hocolim_{\sD^2} (X \times X) \overset{\mu}{\arr} \hocolim_{\sD^2} (X \circ \oplus) \overset{\oplus_*}{\arr} \hocolim_{\sD} X.
\]
In fact, this monoid structure underlies the $E_{\infty}$-space structure on $\hocolim_{\sD} X$ (Remark \ref{hocolim_IFCP_is_E_infty}).
Taking components, there is an induced monoid structure on $\pi_0 X = \pi_0 \hocolim_{\bbI} X$.  We say that $X$ is grouplike if $\pi_0 X$ is a group.  Consider $\pi_0 X$ and $(\pi_0 X)^\times$ as constant $\sD$-spaces.  We define a discretization map of $\sD$-spaces $X \arr \pi_0 X$ by:
\[
X(d) \arr \pi_0 X(d) \arr \underset{d \in \sD}{\colim} \, \pi_0 X(d) \arr \pi_0 \hocolim_{\sD} X = \pi_0 X.
\]
The first map is the discretization map of spaces that takes a point to its connected component in $\pi_0$, and we give $\pi_0 X(d)$ the quotient topology.  Give $\pi_0 X$ the topology inherited from the colimit topology on $\colim \pi_0 X_n$, so that the the composite is continuous.  

\begin{definition}\label{pullback_defining_units}
Given a $\sD$-FCP $X$, define the $\sD$-FCP $X^{\times}$ to be the following pullback of $\sD$-spaces:
\begin{equation*}
\xymatrix{
X^{\times} \ar[d] \ar[r] & X \ar[d] \\
(\pi_0 X)^{\times} \ar[r] & \pi_0 X
}
\end{equation*}
\end{definition}
The space $(X^{\times})(d)$ is the disjoint union of the components of $X(d)$ whose elements are stably invertible, in the sense that they map to units in $\pi_0 X$.  It is immediate that $\pi_0 (X^{\times}) = (\pi_0 X)^{\times}$ and so $X^{\times}$ is grouplike.  Furthermore, $X \mapsto X^{\times}$ is right adjoint to the forgetful functor from grouplike $\sD$-FCPs to $\sD$-FCPs.

Let $\pi'_0 X$ be the $\sD$-FCP $(\pi'_0 X)(d) = \pi_0 (X(d))$.  Define a sub $\sD$-FCP $\pi'_0 X^{\x}$ by letting $(\pi'_0 X^{\x})(d)$ be the set of components $[x] \in \pi_0 X(d)$ for which there exists $[y] \in \pi_0 X(d')$ such that $[\mu(x, y)]$ and $[\mu(y, x)]$ are the components of the image of the FCP unit map $\eta \colon * \arr X$ in $\pi_0X(d \oplus d')$ and $\pi_0X(d' \oplus d)$.   It follows that $\colim_{\bbI} \pi'_0 X^{\x} \cong (\pi_0 X)^{\x}$, so $X^{\x}$ can also be described as the pullback
\addtocounter{theorem}{1}
\begin{equation}\label{alternative_units_pullback}
\xymatrix{
X^{\times} \ar[d] \ar[r] & X \ar[d] \\
\pi'_0 X^{\times} \ar[r] & \pi'_0 X
}
\end{equation}

Specializing to the case of the FCP $\Omega^{\bullet} R$ for diagram ring spectra $R$, we define FCPs of units:

\begin{definition}  If $R$ is a symmetric ring spectrum, we define the $\bbI$-FCP of units of $R$ to be $\mGL_1^{\bullet} R = (\Omega^{\bullet} R)^{\times}$.  If $R$ is an orthogonal ring spectrum, we define the $\cI$-FCP of units of $R$ to be $\mGL_1^{\bullet} R = (\Omega^{\bullet} R)^{\times}$.  
\end{definition}

\begin{remark}  For semistable symmetric spectra, $\pi_0 \Omega^{\infty} R \cong \pi_0 R$ (see Remark \ref{semistable_comment}).  In this case, $(\pi_0 R)^{\times}$ is the usual subgroup of multiplicatively invertible elements of the ring $\pi_0 R$ and $\mGL_{1}^{\bullet} R$ is defined by the following pullback:
\[
\xymatrix{
\mGL_{1}^{\bullet} R \ar[d] \ar[r] & \Omega^{\bullet} R \ar[d] \\
(\pi_0 R)^{\times} \ar[r] & \pi_0 R
}
\]
\noindent On the other hand, $\pi_0 \Omega^{\infty} R \cong \pi_0 R$ for all orthogonal ring spectra $R$.  Hence $\mGL_{1}^{\bullet} R$ is always described by the above pullback.
\end{remark}

\section{The spectrum of units of a commutative diagram ring spectrum}\label{spectrum_units_construction_section}

From a commutative diagram ring spectrum $R$ we have constructed a commutative FCP of units $\mGL_1^{\bullet} R$.  We will now convert commutative FCPs into $\Gamma$-spaces and then apply an infinite loop space machine to get spectra of units.   For $\bbI$-FCPs this construction is originally due to Schlichtkrull \cite{schlichtkrull_units} (see also \citelist{\cite{Sch_thomspectra} \cite{Sch_higherTHH} }) and for $\cI$-FCPs we will use a direct analog of his approach.

Let $\Gamma^{\op}$ be the skeletal category of finite base sets $\bn^{+} = \{0, 1, \dotsc, n\}$ with basepoint $0$ and based maps.  A $\Gamma$-space is a functor $Y \colon \Gamma^{\op} \arr \sU$ such that $Y(\mathbf{0}^+)$ is contractible and the projections $\delta_i \colon \bn^+ \arr \mathbf{1}^+$ sending $i$ to $1$ and all other elements to the basepoint induce a weak equivalence:
\begin{equation*}
\delta \colon Y(\bn^+) \overset{\simeq}{\arr} \prod_{i = 1}^n Y(\mathbf{1}^+).
\end{equation*}
We refer to this map as the Segal map.  Note that our $\Gamma$-spaces are often called special $\Gamma$-spaces elsewhere.  Write $Y_n$ for $Y(\bn^+)$.  The underlying space of a $\Gamma$-space is $Y_1$.  In general, $\pi_0 Y_1$ an abelian monoid.  We call a $\Gamma$-space grouplike if $\pi_0 Y_1$ is a group.

Starting with a commutative $\bbI$-FCP X, we will construct a $\Gamma$-space $H_{\bbI} X$.  Starting with a commutative $\cI$-FCP, we will construct two different $\Gamma$-spaces: $H_{\cI} X$ and $H_{\cI(U)} X$.  $H_{\cI} X$ will be used to define the spectrum of units of an orthogonal ring spectrum.  $H_{\cI(U)} X$ will be used to compare with the spectrum of units of an $E_{\infty}$ ring spectrum, and uses the full subcategory $\cI(U)$ of $\cI$ whose objects are finite dimensional sub-inner product spaces $V \subset U$.  Unlike $\cI$, the category $\cI(U)$ is an $\sL$-category (\ref{Topological Categories and the Bar Construction}.6), which will allow the comparison with the $\sL$-space $\colim_{\cJ} X$.  However, $\cI(U)$ is \emph{not} symmetric monoidal, so we will never consider $\cI(U)$-spaces or $\cI(U)$-FCPs.  In what follows, the homotopy colimits that define $H_{\cI} X$ are built by replacing $\cI$ with the equivalent small topological subcategory $\cI^{\dagger}$, as usual (see \S\ref{Infinite loop space theory of orthogonal spectra}).  We break from our usual convention, and explicitly write $\cI^{\dagger}$ where it appears in this construction.

\begin{construction}\label{FCP_to_fspace_construction}  Let $\sD$ denote $\bbI$, $\cI^{\dagger}$ or $\cI(U)$.  We construct a functor $X \longmapsto H_{\sD} X$ from commutative $\bbI$-FCPs (for $\sD = \bbI$) or $\cI$-FCPs (for $\sD = \cI^{\dagger}, \cI(U)$) to $\Gamma$-spaces as follows.  Let $\sP(\bn^{+})$ be the poset of subsets $A \subset \bn^{+} = \{0, 1, \dotsc, n\}$ that do \emph{not} contain $0$.  Let $\sD(\bn^{+})$ be the category of functors $\theta \colon \sP(\bn^{+}) \arr \sD$ such that for every pair of subsets $A, B \in \sP(\bn^{+})$ with $A \cap B = \emptyset$, the diagram induced by the inclusions $A \arr A \cup B$ and $B \arr A \cup B$
\[
\theta(A) \arr \theta(A \cup B) \longleftarrow \theta(B)
\]
is a coproduct diagram in the category of finite sets (for $\sD = \bbI$) or in the category of vector spaces (for $\sD = \cI^{\dagger}, \cI(U)$).  Morphisms in $\sD(\bn^{+})$ are natural transformations of functors.  Notice that $\theta(\emptyset) = 0$, the initial object (and unit for $\oplus$).  Thus an object $\theta$ of $\sD(\bn^{+})$ consists of:
\begin{itemize}
\item objects $\theta_i = \theta(\{i\})$ for $1 \leq i \leq n$,
\item morphisms $\theta_i \arr \theta(A)$ for each $A \subset \{1, \dotsc, n\}$ and $i \in A$ that assemble into a canonical isomorphism:
\addtocounter{theorem}{1}
\begin{align}\label{ordering_iso}
\bigoplus_{i \in A} \theta_i &\cong \theta(A).
\end{align}
\end{itemize}
Here $\oplus$ denotes the monoidal product of $\sD$.  For $\sD = \cI(U)$, the abstract direct sum is not a subspace of $U$, so the isomorphism is a morphism of $\cI$ but not $\cI(U)$.  

We will take topological homotopy colimits over $\sD(\bn^{+})$ and so we must give $\sD(\bn^{+})$ a topology.  Let $\sP(\bn^{+})$ have the discrete topology and topologize the category $\Fun(\sP(\bn^{+}), \sD)$ of functors and natural transformations as described in \ref{topological_functor_category}. Since $\sD(\bn^{+})$ is a full subcategory of $\Fun(\sP(\bn^{+}), \sD)$, we topologize $\ob \sD(\bn^{+})$ and $\mor \sD(\bn^{+})$ as subspaces of $\ob \Fun(\sP(\bn^{+}), \sD)$ and $\mor \Fun(\sP(\bn^{+}), \sD)$.  

We have a continuous forgetful functor $u \colon \sD(\bn^{+}) \arr \sD^n$ that sends $\theta$ to $(\theta_{1}, \dotsc, \theta_{n})$.    Let $X(\bn^{+})$ denote the composite functor
\begin{align*}
X(\bn^{+}) \colon \sD(\bn^{+}) &\overset{u}{\arr} \sD^n \overset{X^n}{\arr} \sU \\
&\theta \longmapsto \prod_{i = 1}^{n} X(\theta_{i}).
\end{align*}
Define a functor $H_{\sD}X$ from $\Gamma^{\op}$ to spaces using the homotopy colimit:
\[
H_{\sD} X (\bn^{+}) = \underset{\sD(\bn^{+})}{\hocolim} \; X(\bn^{+}).
\]
By definition, $H_{\sD} X(\mathbf{0}^{+}) = \ast$.  Note that the underlying space of $H_{\sD} X$ is $(H_{\sD} X)_1 = \hocolim_{\sD} X$.  To see the functoriality in $\Gamma^{\op}$, suppose that $\alpha \colon \bm^{+} \arr \bn^{+}$ is a map of based sets.  The inverse image functor $\alpha^{-1} \colon \sP(\bn^{+}) \arr \sP(\bm^{+})$ induces a functor $\alpha_* \colon \sD(\bm^{+}) \arr \sD(\bn^{+})$ by precomposition.  We then have a natural transformation $X(\alpha) \colon X(\bm^{+}) \arr X(\bn^{+}) \circ \alpha_*$ defined using the FCP multiplication $\mu$:
\begin{align*}
X(\alpha)_\theta \colon \prod_{i = 1}^{m} X(\theta_i) \arr \prod_{j = 1}^{m} \prod_{i \in \alpha^{-1}(j)} X(\theta_i) \overset{\mu}{\arr} &\prod_{j = 1}^{m} X \bigl( \bigoplus_{i \in \alpha^{-1}(j)} \theta_i \bigr) \\
&\cong \prod_{j = 1}^{m} X( \theta( \alpha^{-1}(j))).
\end{align*}
The first map is the projection away from the factors indexed by elements $i \in \bm^{+}$ that map to the basepoint under $\alpha$.  The last isomorphism is induced by the canonical isomorphism \eqref{ordering_iso}.  Since $X$ is a commutative FCP, $X(\alpha)$ does not depend on the choice of ordering used to define the product indexed by $\alpha^{-1}(j)$.  When $\sD = \cI(U)$ the isomorphism \eqref{ordering_iso} is not a morphism in $\cI(U)$.  However, in that case we start with an $\cI$-FCP, so we are free to use functoriality in $\cI$ when defining $H_{\cI(U)} X$.  The natural transformation $X(\alpha)$ induces the first map of homotopy colimits below:
\[
\underset{\sD(\bm^{+})}{\hocolim} \; X(\bm^{+}) \overset{X(\alpha)}{\arr} \underset{\sD(\bm^{+})}{\hocolim} \; X(\bn^{+}) \circ \alpha_* \overset{\alpha_*}{\arr} \underset{\sD(\bn^{+})}{\hocolim} \; X(\bn^{+}).
\]
The second map is induced by the functor $\alpha_* \colon \sD(\bm^{+}) \arr \sD(\bn^{+})$ and the composite gives functoriality of $H_{\sD} X$ in the morphism $\alpha$.
\end{construction}

\begin{proposition}\label{construction_gives_gamma_space}  The functor $H_{\sD} X$ is a $\Gamma$-space.
\end{proposition}
\begin{proof}  It remains to show that the Segal maps are weak equivalences.  Under the canonical isomorphism $(\hocolim_{\sD} X)^n \cong \hocolim_{\sD^n} X^n$, the Segal map 
\[
\delta \colon H_{\sD}(\bn^+) \arr \prod_{i = 1}^n H_{\sD}(\mathbf{1}^+)
\]
is identified with the map of homotopy colimits induced by the forgetful functor $u \colon \sD(\bn^{+}) \arr \sD^n$:
\[
\underset{\sD(\bn^+)}{\hocolim} \; X(\bn^+) \arr \underset{\sD^n}{\hocolim} \; X^n.
\]
To show that this is a weak homotopy equivalence, we will apply Lemma \ref{theorem_A_hocolim_version} to the forgetful functor $u \colon \sD(\bn^+) \arr \sD^n$.  For any object $\bd = (d_1, \dotsc, d_n)$ of $\sD^n$, the comma category $(\bd \downarrow u)$ has initial object $\id \colon \bd \arr u(\theta)$, where $\theta$ is the functor $\sP(\bn^+) \arr \sD$ given by:
\[
\theta (A) = \bigoplus_{i \in A} d_{i}.
\]
This works as written for $\sD = \bbI$ and $\cI^{\dagger}$.  For $\sD = \cI(U)$, the direct sum is not an object of $\cI(U)$.  Instead, we choose isometric isomorphisms $\phi_i \colon d_i \arr d'_i$ such that the $d'_i$ are pairwise orthogonal for $1 \leq i \leq n$.  Then define the functor $\theta$ using the \emph{internal} direct sum:
\[
\theta (A) = \sum_{i \in A} d'_{i} \subset U.
\]
The object $\prod_i \phi_i \colon \bd \arr u(\theta)$ is an initial object of $(\bd \downarrow u)$.  In all cases, the classifying space of a topological category with initial object is contractible, so Lemma \ref{theorem_A_hocolim_version} applies, yielding the desired weak equivalence.  
\end{proof}

From now on, denote the $\Gamma$-space $H_{\cI^{\dagger}} X$ by $H_{\cI} X$.  Notice that if $X$ is a commutative $\cI$-FCP, the inclusions of categories $\bbI \arr \cI^{\dagger}$ and $\cI(U) \arr \cI^{\dagger}$ induce maps of $\Gamma$-spaces $H_{\bbI} \bbU X \arr H_{\cI} X$ and $H_{\cI(U)} X \arr H_{\cI} X$.  

\begin{remark}  The preceding construction is equivalent to the construction of a $\Gamma$-space from the permutative category $\sD[X]$ given in \cite{May_unique_from_perm} (see Remark \ref{hocolim_IFCP_is_E_infty}).  In the case of $\sD = \cI(U)$, which is not permutative, one must use the structure of a partial permutative category given by the internal direct sum, as discussed in \cite{May_spectra_of_Imonoids}.  Our construction avoids the use of partial structures because we only need the internal direct sum in the verification of the $\Gamma$-space condition.
\end{remark}

The last step in constructing the spectrum of units is converting a  $\Gamma$-space into a spectrum.  Our model for the output will be weak $\Omega$-spectra.  A weak $\Omega$-spectrum $X$ is a sequence of spaces $X_n$ with spectrum structure maps $X_n \arr \Omega X_{n + 1}$ that are weak homotopy equivalences.  In the language of \cite{MMSS}, these are the fibrant objects in the category of coordinatized prespectra.  Denote the category of weak $\Omega$-spectra by $\sS^{\Omega}$.  We will use the term spectrum for this notion from now on, hoping not to cause confusion with the underlying (LMS) spectrum of an $\bbL$-spectrum.  Since the spectrum of units of a ring spectrum will never carry multiplicative structure, we need not model it in a symmetric monoidal category of spectra: ultimately, we only care about the object it determines in the stable homotopy category.

There are many different, but equivalent, constructions of a spectrum from a $\Gamma$-space.  Instead of choosing one, we follow the axiomatic approach of \cite{May_Thomason} ($\Gamma$-spaces are called $\sF$-spaces there).  The most general input is the notion of an $\sO$-space for a category of operators $\sO$.  $\Gamma$-spaces and $\sC$-spaces for an $E_{\infty}$ operad $\sC$ are both examples of $\sO$-spaces.  An $\sO$-space $Y$ has an underlying space $Y_1$.  An infinite loop space machine is a functor $E$ from $\sO$-spaces (for some category of operators $\sO$) to weak $\Omega$-spectra along with a natural group completion $\iota \colon Y_1 \arr E_0 Y = (E Y)_0$.  The main theorem of \cite{May_Thomason} shows that any two infinite loop space machines agree up to natural weak equivalence, so the following definition does not depend on the choice of $E$.  

\begin{definition}\label{def_of_units}  Let $R$ be a symmetric or orthogonal commutative ring spectrum.  The spectrum of units of $R$ is the output of any infinite loop space machine $E$ applied to the $\Gamma$-space given by Construction \ref{FCP_to_fspace_construction} for the commutative FCP $\mGL^{\bullet}_1 R$:
\[
gl_1 R = E H_{\sD} \mGL_{1}^{\bullet} R \qquad \begin{cases} \text{$\sD = \bbI$ for symmetric ring spectra,} \\ \text{$\sD = \cI$ for orthogonal ring spectra.} \end{cases}
\]
\end{definition}

\begin{remark}
We could have defined the spectra of units of diagram ring spectra operadically, by applying an infinite loop space machine to the $E_{\infty}$ space $\hocolim \mGL_1^{\bullet} R$ (Remark \ref{hocolim_IFCP_is_E_infty}).  In fact, this gives the same homotopy type as Definition \ref{def_of_units}.  One can prove this by applying the uniqueness result of \cite{May_unique_from_perm} to the permutative category whose geometric realization is $\hocolim \mGL_1^{\bullet} R$.  
\end{remark}

We will not use the following result until the comparison with the spectra of units of $E_{\infty}$ ring spectra, but we record it here since it is immediate from the definitions.   It shows how Construction \ref{FCP_to_fspace_construction} encodes infinite loop space structure in two different ways.  Given an operad $\sC$, a $\Gamma\sC$-space is a $\Gamma$-space $X$ such that each $X(A)$ is a $\sC$-space and the maps induced by morphisms in $\Gamma^{\op}$ are maps of $\sC$-spaces.  

\begin{proposition}\label{HX_is_F_Einfty}  Let $X$ be a commutative $\cI$-FCP. Then $H_{\cI(U)}X$ is a $\Gamma \sL$-space, where $\sL$ is the linear isometries operad.
\end{proposition}
\begin{proof}  By Proposition \ref{hocolim_of_FCP_is_Lspace_prop}, $\cI(U)$ is an $\sL$-category.  Let $u \colon \ob \cI(U)(\bn^{+}) \arr (\ob \cI(U))^n$ be the object map of the forgetful functor $\theta \mapsto (\theta_{1}, \dotsc, \theta_{n})$.  Define an $\sL$-space structure on $\ob \cI(U)(\bn^{+})$ by letting $\gamma \in \sL(j)$ act on $(\theta^1, \dotsc, \theta^j) \in \ob \cI(U)(\bn^{+})^j$ by:
\[
\gamma \cdot (\theta^1, \dotsc, \theta^j) \colon A \longmapsto \gamma( \theta^1(A) \oplus \dotsm \oplus \theta^j(A) ), \qquad \text{where $A \in \sP(\bn^{+})$.}
\]
Since this agrees with the product $\sL$-space structure on $(\ob \cI(U))^n$, the map $u \colon \ob \cI(U)(\bn^{+}) \arr (\ob \cI(U))^n$ is a map of $\sL$-spaces.

Let $t^n \colon \sX^n \arr (\ob \cI(U))^n$ be the $n$-fold cartesian product of the left $\cI(U)$-module associated to the $\cI$-space $X$.   By Proposition \ref{hocolim_of_FCP_is_Lspace_prop}.(iv), $t^n$ is a map of $\sL$-spaces.  The left $\cI(U)(\bn^{+})$-module associated to the $\cI(U)(\bn^{+})$-space $X(\bn^{+})$ is the pullback:
\[
\xymatrix{ \sX(\bn^{+}) \ar[r] \ar[d] & \sX^n \ar[d] \\
\ob \cI(U)(\bn^{+}) \ar[r]^{u} & (\ob \cI(U))^n .}
\]
Since this is the pullback of a diagram of $\sL$-spaces, $\sX(\bn^{+})$ is an $\sL$-space and the structure map $\sX(\bn^{+}) \arr \ob \cI(U)(\bn^{+})$ is a map of $\sL$-spaces.

The space of morphisms $\mor \cI(U)(\bn^{+})$ is the following pullback:
\[
\xymatrix{ \mor \cI(U)(\bn^{+}) \ar[r] \ar[d]_{s \times t} & (\mor \cI(U))^n \ar[d]^{s \times t} \\
\ob \cI(U)(\bn^{+}) \times \ob \cI(U)(\bn^{+}) \ar[r]^{u \times u} & (\ob \cI(U))^n \times (\ob \cI(U))^n .}
\]
Again this is a pullback diagram of $\sL$-spaces, so $\mor \cI(U)(A)$ is an $\sL$-space and the source and target maps $\mor \cI(U)(\bn^{+}) \arr \ob \cI(U)(\bn^{+})$ are maps of $\sL$-spaces.  The left module structure map $\mor \cI(U)(\bn^{+}) \times_{\ob \cI(U)(\bn^{+})} \sX(\bn^{+}) \arr \sX(\bn^{+})$ is defined using these pullback diagrams and so is a map of $\sL$-spaces as well.  It follows that $\sX(\bn^{+})$ is an $\sL \cI(U)(\bn^{+})$-algebra, in the sense of \S\ref{Topological Categories and the Bar Construction}.  By Lemma \ref{barconstruction_Lspace_lemma}, $\hocolim_{\cI(U)(\bn^{+})} X(\bn^{+})$ is an $\sL$-space.  The preceding argument is functorial in the variable $\bn^{+}$ and so the homotopy colimits $\hocolim_{\cI(U)(\bn^{+})} X(\bn^{+})$ assemble into a $\Gamma \sL$-space.
\end{proof}

\section{Comparison of units of diagram ring spectra}

We will compare the spectra of units of symmetric and orthogonal commutative ring spectra.  First we need some lemmas on homotopy colimits, which will be useful in later sections as well.  For the following, see Remark \ref{top_initial_object_remark} on initial objects in topological categories.

\begin{lemma}\label{fibrant_hocolim_lemma} Let $\sD$ be a topological category with an initial object $0$.  Let $X$ be a $\sD$-space such that for every morphism $\phi$ of $\sD$, the induced map $X(\phi)$ is a weak homotopy equivalence of spaces.  Then for any object $d$ of $\sD$, the inclusion
\[
X(d) \arr \hocolim_{\sD} X
\]
is a weak homotopy equivalence.
\end{lemma}
\noindent The proof in \cite{telescope_lemma}*{6.2} for ordinary categories also applies to topological categories.

When $X$ is a fibrant $\sD$-space in the positive model structure, $X(\phi) \colon X(d) \arr X(d')$ is only an equivalence for $d \neq 0$, so the previous lemma does not apply.  A serious technical result provides the same conclusion:

\begin{lemma}[B\"okstedt]\label{bokstedt_lemma}
Let $\sD = \bbI$ or $\cI$.  Let $X$ be a $\sD$-space and let $\sD_{>n}$ denote the full subcategory of $\sD$ consisting of objects $\bd > \bn$ for $\bbI$ and $V$ with $\dim(V) > n$ for $\cI$.  Suppose that every morphism $d \arr d'$ in $\sD_{> n}$ induces a $\lambda$-connected map $X(d) \arr X(d')$.  Then for $d \in \ob \sD_{>n}$, the inclusion $X(d) \arr \hocolim_{\sD} X$ is at least $(\lambda - 1)$-connected.
\end{lemma}

\noindent This goes back to B\"oksted's THH preprints.  For a published proof, see \cite{brun}*{2.5.1}.

\begin{lemma}\label{general_inclusion into_hocolim_equiv}  Let $\sD = \bbI$ or $\cI^{\dagger}$ and suppose that $X(\phi)$ is a weak equivalence for every morphism $\phi$ in $\sD_{> 0}$.  Then the inclusion 
\[
\hocolim_{\sD_{>0}} X \arr \hocolim_{\sD} X
\]
is a homotopy equivalence.
\end{lemma}
\noindent For a proof, see \cite{telescope_lemma}*{6.4}.  The proof there for $\bbI$ works for $\cI$ as well.

\begin{lemma}\cite{May_Thomason}*{2.3}\label{MT_fspace_equiv_lemma}  Let $\psi \colon X \arr Y$ be a map of $\Gamma$-spaces such that the map of underlying spaces $\psi_1 \colon X_1 \arr Y_1$ is a weak homotopy equivalence.  Then $E\psi$ is a weak equivalence of spectra.
\end{lemma}

\begin{lemma}\label{equiv_FCP_gives_equiv_spectra}
Let $X \arr Y$ be a weak homotopy equivalence of commutative $\bbI$-FCPs.  Then the induced map $EH_{\bbI} X \arr EH_{\bbI} Y$ is a weak equivalence of spectra.  Similarly, a weak homotopy equivalence of commutative $\cI$-FCPs induces a weak equivalence of spectra $EH_{\cI} X \arr EH_{\cI} Y$.
\end{lemma}
\begin{proof} Here is the proof for $\bbI$-FCPs.  The orthogonal case is essentially identical.  By Lemma \ref{MT_fspace_equiv_lemma}, it suffices to show the map of spaces $H_{\bbI}X_1 \arr H_{\bbI}Y_1$ is an equivalence.  By definition, $H_{\bbI}X_1 = \hocolim_{\bbI} X$.  Since $X \arr Y$ is a weak homotopy equivalence, the result follows.
\end{proof}

The following proposition compares the output of Construction \ref{FCP_to_fspace_construction} when fed $\cI$-FCPs and their underlying $\bbI$-FCPs.

\begin{proposition}\label{equivalence_Fspace_over_U}  Let $X$ be a fibrant commutative $\cI$-FCP.  Then there is a natural weak equivalence of spectra $EH_{\cI} X\simeq EH_{\bbI} (\bbU X)$.  
\end{proposition}
\begin{proof}
The inclusion of categories $\bbI \arr \cI$ induces a map of $\Gamma$-spaces $H_{\bbI} \bbU X \arr H_{\cI} X$.  We will show that this induces a weak equivalence of spectra.  Fix $d > 0$.  Because $X$ and $\bbU X$ are positive fibrant, Lemma \ref{bokstedt_lemma} implies that the horizontal maps in the following commutative diagram are weak homotopy equivalences:
\[
\xymatrix{
\bbU X (d) \ar[r]^-{\simeq} & \hocolim_{\bbI} \bbU X \ar[d] 
 \\
X(\bR^d) \ar[r]^-{\simeq} \ar@{=}[u] & \hocolim_{\cI} X }
\]
Thus the vertical arrow of homotopy colimits is a weak homotopy equivalence.  By Lemma \ref{MT_fspace_equiv_lemma}, the associated map of spectra is a weak equivalence.
\end{proof}

We are ready to prove the comparison theorem for the units of symmetric and orthogonal ring spectra.  Before stating the result, we need to know that $\mGL_1^{\bullet}$ descends to a functor on homotopy categories.

\begin{lemma}\label{GL_1_preserves_fibrants} For symmetric and orthogonal ring spectra and commutative ring spectra, the functor $\mGL_1^{\bullet}$ preserves fibrant objects and weak equivalences between fibrant objects.
\end{lemma}
\begin{proof}
Consider the case of a fibrant (commutative) symmetric ring spectrum $R$ (the orthogonal case is essentially identical).  Since $\Omega^{\bullet}$ is right Quillen, the (commutative) $\bbI$-FCP $\Omega^{\bullet} R$ is fibrant in the (positive)  model structure.  Restricting the equivalences $\Omega^m R_m \arr \Omega^n R_n$ to the stably unital components, we get equivalences $(\Omega^m R_m)^\times \arr (\Omega^n R_n)^\times$ (for $m \geq 1$ in the commutative case).  Hence $\mGL_1 R$ is fibrant in the (positive) stable model structure.  

Next, Suppose that $f \colon R \arr R'$ is a stable equivalence of fibrant (commutative) symmetric ring spectra.  By \cite{MMSS}*{8.11}, $f$ is a (positive) level equivalence of symmetric spectra.  Since $\pi_0(R) \arr \pi_0(R')$ is an isomorphism of monoids, the induced map of pullbacks $(\Omega^n R_n)^{\times} \arr (\Omega^n R'_n)^{\times}$ is a weak homotopy equivalence (for $n \geq 1$ in the commutative case).  Thus $\mGL^{\bullet}_{1} f$ is a (positive) level equivalence.  By Lemma \ref{general_inclusion into_hocolim_equiv}, a positive level equivalence is a weak homotopy equivalence, so we have proved that $\mGL_1^{\bullet}$ preserves weak equivalences between fibrant objects.
\end{proof}
\noindent It follows that in both the associative and the commutative settings, $\mGL_1^{\bullet}$ has a right derived functor $\bR \mGL_{1}^{\bullet}$.  By \ref{equiv_FCP_gives_equiv_spectra},  the functor $gl_1$ preserves weak equivalences between fibrant objects.  Hence $gl_1$ also has a right derived functor $\bR gl_1$ from the homotopy category of commutative ring spectra to the homotopy category of spectra.  Next, we make the comparison of the FCPs $\mGL_1^{\bullet}$:

\begin{proposition}\label{diagram_GL_1_comparison_prop}  The following diagrams commute:
\[
\xymatrix{
\bM \Sigma \sS \ar[d]_{\mGL_{1}^{\bullet}} & \bM \sI \sS \ar[l]_{\bbU} \ar[d]^{\mGL_{1}^{\bullet}} & & \ho \bM \Sigma \sS \ar[r]^-{\bL \bbP} \ar[d]_{\bR \mGL_{1}^{\bullet}} & \ho \bM \sI \sS  \ar[d]^{\bR \mGL_{1}^{\bullet}} \\
\bM \bbI \sU  & \bM \cI \sU \ar[l]^{\bbU} & &  \ho \bM \bbI \sU \ar[r]_-{\bL \bbP} & \ho \bM \cI \sU }
\]
The diagrams of commutative monoids \emph{(}with $\bC$ replacing $\bM$\emph{)} also commute.
\end{proposition}
\begin{proof}  
Let $R$ be an orthogonal ring spectrum.  Then $\bbU R$ is a semistable symmetric ring spectrum, so $\mGL_1^{\bullet} \bbU R$ may be identified with the following pullback of $\bbI$-FCPs:
\[
\xymatrix{
\mGL_1^{\bullet} \bbU R \ar[r] \ar[d] & \Omega^{\bullet} \bbU R \ar[d] \\
(\pi_0 \bbU R)^{\times} \ar[r] & \pi_0 \bbU R }
\]
Since $\Omega^{\bullet} \bbU R \cong \bbU \Omega^{\bullet} R$ and  $\pi_0 \bbU R \cong \pi_0 R$, $\bbU$ takes the pullback diagram of $\cI$-FCPs defining $\mGL_{1}^{\bullet} R$ to the displayed pullback.  $\bbU$ preserves pullbacks, so we have a natural isomorphism $\bbU \mGL_{1}^{\bullet} R \cong \mGL_{1}^{\bullet} \bbU R$.  This gives the diagram on the left.  The second diagram commutes by applying Proposition \ref{appendix_prop}.  

The same proof works in the commutative setting, but for the second diagram we use the Quillen equivalence of categories of commutative monoids instead of monoids when we apply Proposition \ref{appendix_prop}.
\end{proof}

\begin{theorem}\label{diagram_gl_comparison_theorem}  The following diagrams commute:
\[
\xymatrix{
\ho \bC \Sigma \sS \ar[dr]_{\bR gl_1} & & \ho \bC \sI \sS \ar[dl]^{\bR gl_1} \ar[0,-2]_{\bR \bbU} &  \ho \bC \Sigma \sS \ar[dr]_{\bR gl_1} \ar[0,2]^{\bL \bbP} & & \ho \bC \sI \sS \ar[dl]^{\bR gl_1}  \\
& \ho \sS^{\Omega} & &  & \ho \sS^{\Omega} & }
\]
\end{theorem}
\begin{proof}
The diagram on the left can be expanded into the diagram:
\[
\xymatrix{
\ho \bC \Sigma \sS \ar[d]_{\bR \mGL_1^{\bullet}} & & \ho \bC \sI \sS \ar[d]^{\bR \mGL_1^{\bullet}} \ar[0,-2]_{\bR \bbU} \\ 
\ho \bC \bbI \sU \ar[dr]_{E H_{\bbI}} & & \ho \bC \cI \sU \ar[0, -2]_{\bR \bbU} \ar[dl]^{EH_{\cI}} \\
& \ho \sS^{\Omega} & }
\]
The upper square commutes by Proposition \ref{diagram_GL_1_comparison_prop}.  
Since the right derived functors are calculated by fibrant approximation,  Proposition \ref{equivalence_Fspace_over_U} applies, proving that the lower triangle commutes.

The second diagram is:
\[
\xymatrix{
\ho \bC \Sigma \sS \ar[d]_{\bR \mGL_1^{\bullet}} \ar[0,2]^{\bL \bbP} & & \ho \bC \sI \sS \ar[d]^{\bR \mGL_1^{\bullet}}  \\ 
\ho \bC \bbI \sU \ar[dr]_{E H_{\bbI}} \ar[0, 2]^{\bL \bbP} & & \ho \bC \cI \sU  \ar[dl]^{EH_{\cI}} \\
& \ho \sS^{\Omega} & }
\]
The top square commutes by Proposition \ref{diagram_GL_1_comparison_prop}.  For the lower triangle, notice that the following natural transformation is an isomorphism:
\[
EH_{\bbI} \xrightarrow{EH_{\bbI} \eta} EH_{\bbI} \, \bR \bbU \, \bL \bbP \cong EH_{\cI} \, \bL \bbP.
\]
Here the displayed isomorphism is the commutativity of the triangle in the previous diagram and the unit $\eta$ of the adjunction $(\bL \bbP, \bR \bbU)$ is an isomorphism by the Quillen equivalence of Theorem \ref{quillen_equiv_FCP}.  \end{proof}

\section{Comparison of units of orthogonal and $E_{\infty}$ ring spectra}\label{section_units_comparison_ortho_einfty}

We now make the comparison between the spectra of units of commutative orthogonal ring spectra and $E_{\infty}$ ring spectra.  They must first be compared at the space level:

\begin{proposition}\label{GL_1_orthogonal_Einfty_prop}  The following diagrams commute:
\[
\xymatrix{
\bM \sI \sS \ar[d]_{\mGL_1^{\bullet}} & \bM\sS[\bbL] \ar[l]_-{\bbN^{\#}} \ar[d]^{ \mGL_1} & & \ho \bM \sI \sS \ar[r]^-{\bL \bbN} \ar[d]_{\bR \mGL_1^{\bullet}} & \ho \bM\sS[\bbL]  \ar[d]^{\bR \mGL_1} \\
\bM \cI \sU  & \bM\sU[\bbL] \ar[l]^-{\bbQ^{\#}} & &  \ho \bM \cI \sU \ar[r]_-{\bL \bbQ} & \ho \bM\sU[\bbL] }
\]
The diagrams of commutative monoids \emph{(}with $\bC$ replacing $\bM$\emph{)} also commute.
\end{proposition}
\begin{proof}
We consider the diagram on the left first.  The second diagram will follow by Proposition \ref{appendix_prop} as usual.  $\mGL_1^{\bullet}$ is the composite of $\Omega^{\bullet}$ and the functor of $\cI$-FCPs: $X \mapsto X^{\x}$.  Similarly, the functor $\mGL_1$ from $A_{\infty}$ ring spectra to $\bbL$-space monoids is the composite of $\Omega^{\infty}_{\bbL}$ and the functor taking an $\bbL$-space monoid $Y$ to the pullback
\[
\xymatrix{
Y^{\times} \ar[r] \ar[d] & Y \ar[d] \\
\pi_0 Y^{\x} \ar[r] & \pi_0 Y. }
\]
Using the diagram in Proposition \ref{loops_orthogonal_to_Lspectra_prop} that compares $\Omega^{\bullet}$ with $\bbN^{\#}$ and $\bbQ^{\#}$, it suffices to prove that $\bbQ^{\#} (Y^\x) \cong (\bbQ^{\#} Y)^\x$ for an $\bbL$-space monoid $Y$.

The components of the space $\bbQ^{\#}Y(V)$ are given by
\[
\pi_0 \bbQ^{\#} Y(V) = \pi_0 \sU[\bbL](\bbQ^*(V), Y) \overset{\cong}{\arr} \pi_0 Y,
\]
where the map sends a homotopy class $[f]$ to the component of its image in $Y$.  This is well-defined because $\bbQ^{*}(V) = \cI_c(V \otimes U, U)$ is contractible.  The FCP structure on $\bbQ^{\#}Y$ is defined in terms of the multiplication $Y \boxtimes_{\sL} Y \arr Y$ and the strong symmetric monoidal structure map $\bbQ^*(V) \boxtimes_{\sL} \bbQ^*(W) \cong \bbQ^*(V \oplus W)$.   It follows that the FCP multiplication $\pi_0 \bbQ^{\#} Y(V) \times \pi_0 \bbQ^{\#} Y(W) \arr \pi_0 \bbQ^{\#}Y(V \oplus W)$ agrees with the multiplication on $\pi_0 Y$ coming from the $\bbL$-space monoid structure.

Using the notation established right before diagram \eqref{alternative_units_pullback}, we have an isomorphism of FCPs $\pi'_0 \bbQ^{\#} Y \cong \pi_0Y$.  Therefore, $\bbQ^{\#}$ takes the pullback diagram defining $Y^{\x}$ to the pullback diagram \eqref{alternative_units_pullback} defining $(\bbQ^{\#} Y)^\x$.  Since $\bbQ^{\#}$ preserves pullbacks, we have a natural isomorphism $\bbQ^{\#} (Y^\x) \cong (\bbQ^{\#} Y)^\x$.
\end{proof}

To compare the associated spectra of units, we need to compare the output of different infinite loop space machines.  This is the main technical thrust of the uniqueness theorem for infinite loop space machines \cite{May_Thomason}, whose methods and notation we follow for the next proposition.  The Segal machine $S$ is an infinite loop space machine defined on $\Gamma$-spaces.  Let $E$ be any infinite loop space machine defined on $\sL$-spaces.  Let $\bY$ be a $\Gamma\sL$-space with $n$-th $\sL$-space $\bY_n = \bY(\bn^+)$ and let $Y$ denote the underlying $\Gamma$-space obtained by forgetting the $\sL$-space structures.    

\begin{proposition}\label{FL_space_equiv}  There is a natural weak equivalence of spectra $S Y \simeq E \bY_1$.
\end{proposition}
\begin{proof}
Applying the infinite loop space machine $E$ to each $\sL$-space $\bY_n$, we have a sequence of spectra $E \bY_n$ such that for $m$ fixed, $\bn^+ \mapsto E_m \bY_n$ defines a $\Gamma$-space $E_m \bY$.  Applying the Segal machine $S$ to each of these $\Gamma$-spaces, we have spectra $S E_m \bY$ for $m \geq 0$.  From here the proof is the same as the proof of \cite{May_Thomason}*{2.5}.  By properties of the Segal machine, there are weak equivalences
\[
SE_m \bY \overset{S \sigma}{\arr} S \Omega E_{m + 1} \bY \arr \Omega S E_{m + 1} \bY,
\]
and thus the spectra $S E_m \bY$ comprise a bispectrum.  The ``up-and-across'' theorem \cite{May_Thomason}*{3.9} then yields an equivalence of spectra $S E_0 \bY \simeq S_0 E \bY$.  Taking group completions into account leads to a zig-zag of weak equivalences:
\[
SY \overset{\simeq}{\arr} S E_0 \bY \simeq S_0 E \bY \overset{\simeq}{\longleftarrow} E \bY_1.
\]
\end{proof}

If $X$ is a commuative $\cI$-FCP, we may construct a spectrum by applying the Segal machine $S$ to the $\Gamma$-space $H_{\cI} X$ or by applying $E$ to the $\sL$-space $\hocolim_{\cJ} X$.  The two outputs are equivalent:

\begin{proposition}\label{F_to_L_equivofspectra_prop}  Let $X$ be a commutative $\cI$-FCP.  There is a natural chain of weak equivalences of spectra $S H_{\cI} X \simeq E \hocolim_{\cJ} X$.
\end{proposition}

\begin{proof}
The chain of weak equivalences is:
\[
S H_{\cI} X \overset{\simeq}{\longleftarrow} S H_{\cI(U)} X \simeq E (H_{\cI(U)} X)_1 \overset{\simeq}{\longleftarrow} E \; \underset{\cJ}{\hocolim} \; X.
\]
The inclusion of categories $\cI(U) \arr \cI$ induces a weak homotopy equivalence of spaces $\hocolim_{\cI(U)} X \arr \hocolim_{\cI} X$ by Lemma \ref{theorem_A_hocolim_version}.  The first equivalence in the chain follows by Lemma \ref{MT_fspace_equiv_lemma}.  For the middle equivalence, apply Proposition \ref{FL_space_equiv} to the $\Gamma \sL$-space $H_{\cI(U)} X$ (Proposition \ref{HX_is_F_Einfty}).  For the last equivalence, consider the inclusions of categories $\cJ \arr \cI(U) \arr \cI$.  The composite induces a weak equivalence of homotopy colimits by Proposition \ref{untwisting_prop}.  The second inclusion also induces a weak equivalence as just mentioned.  Thus the first inclusion induces a weak equivalence $\hocolim_{\cJ} X \arr \hocolim_{\cI(U)} X$.  Since this is a map of $\sL$-spaces (\ref{hocolim_of_FCP_is_Lspace_prop}.(v)), it induces the last weak equivalence of spectra after applying $E$.
\end{proof}

We may now compare the spectra of units of commutative orthogonal ring spectra and $E_{\infty}$ ring spectra:

\begin{theorem}\label{gl_comparison_orthogonal_Einfty_theorem}  The following diagrams commute:
\[
\xymatrix{
\ho \bC \sI \sS \ar[dr]_{\bR gl_1} & & \ho \bC\sS[\bbL] \ar[dl]^{\bR gl_1} \ar[0,-2]_{\bR \bN^{\#}} &  \ho \bC \sI \sS \ar[dr]_{\bR gl_1} \ar[0,2]^{\bL \bbN} & & \ho \bC\sS[\bbL] \ar[dl]^{\bR gl_1}  \\
& \ho \sS^{\Omega} & &  & \ho \sS^{\Omega} & }
\]
\end{theorem}
\begin{proof}
We will do the diagram on the right first.  We choose to use the Segal machine in the definition of $gl_1$ for commutative orthogonal ring spectra (Definition \ref{def_of_units}).  The diagram is:
\[
\xymatrix{
\ho \bC \sI \sS \ar[d]_{\bR \mGL_1^{\bullet}} \ar[0,2]^{\bL \bbN} & & \ho \bC\sS[\bbL] \ar[d]^{\bR \mGL_1}  \\ 
\ho \bC \cI \sU \ar[dr]_{S H_{\cI}}  \ar[0, 2]^{\bL \bbQ}  & & \ho \bC \sU[\bbL] \ar[dl]^{E} \\
& \ho \sS^{\Omega} & }
\]
The square commutes by Proposition \ref{GL_1_orthogonal_Einfty_prop}.  For the triangle, let $X$ be a cofibrant commutative $\cI$-FCP.  We will prove in Proposition \ref{general_equiv_cofibrant_FCPs_via_induction} that the map of $\sL$-spaces $\xi \colon \bbQ X \arr \bbO X$ from Lemma \ref{xi_prop} and the canonical projection
\[
\pi \colon \hocolim_{\cJ} X \arr \colim_{\cJ} X
\]
are both weak homotopy equivalences.  Observe that the $\sL$-space action on $\hocolim_{\cJ} X$ specified in Proposition \ref{hocolim_of_FCP_is_Lspace_prop} is defined so that $\pi$ is a map of $\sL$-spaces.   Combined with the isomorphism of $\sL$-spaces $\colim_{\cJ} X \cong \bbO X$ from Lemma \ref{O_is_colim}, we have exhibited a natural chain of weak homotopy equivalences of $\sL$-spaces $\bbQ X \simeq \hocolim_{\cJ} X$.  Along with Proposition \ref{F_to_L_equivofspectra_prop}, this gives a chain of weak equivalences
\[
SH_{\cI} X \simeq E \hocolim_{\cJ} X \simeq E \bbQ X.
\]
Thus the triangle commutes.

The diagram on the left commutes by Proposition \ref{GL_1_orthogonal_Einfty_prop}, the Quillen equivalence $(\bbQ, \bbQ^{\#})$ and the triangle in the other diagram, exactly as in the proof of Theorem \ref{diagram_gl_comparison_theorem}.
\end{proof}

\section{Construction of the model structure on diagram spaces}\label{model_structure_section}
In this section we construct the model structures on $\bbI$-spaces and $\cI$-spaces used throughout the paper.  The analogous model structures on diagram spectra are constructed in \cite{MMSS}, which is the source for many of the arguments in this section.

\medskip

The model structure on commutative monoids in $\sD$-spaces will require an underlying positive model structure on $\sD$-spaces, just as for diagram spectra.  We will work at a level of generality that subsumes all variants by fixing the following data:
\begin{input_data}\label{model_structure_setup}  Fix a pair $(\sD, \sD_{+})$ consisting of a symmetric monoidal topological category $(\sD, \oplus, 0)$ whose unit object $0$ is an initial object (see Remark \ref{top_initial_object_remark}) and a full subcategory $\sD_{+} \subset \sD$.  The categories $\sD$ and $\sD_{+}$ must satisfy the following conditions: 
\begin{itemize}
\item[(i)]  Given a $\sD$-space $X$, there exists an associated left $\sD$-module $\sX \arr \ob \sD$ (see \S\ref{Topological Categories and the Bar Construction} for this terminology) and this association is functorial.  Furthermore, if $X \arr Y$ is a levelwise fibration of $\sD$-spaces, then $\sX \arr \sY$ is a fibration of spaces.
\item[(ii)]  The inclusion of categories $\sD_{+} \arr \sD$ induces a natural weak equivalence of homotopy colimits:
\[
\hocolim_{\sD_{+}} X \overset{\simeq}{\arr} \hocolim_{\sD} X.
\]
\item[(iii)]  Suppose that for every morphism $\phi \colon d \arr d'$ of $\sD_{+}$, the induced map $X(\phi) \colon X(d) \arr X(d')$ is a weak equivalence.  Then for every $d \in \ob \sD_{+}$, the inclusion of $X(d)$ into the homotopy colimit is a weak equivalence:
\[
X(d) \overset{\simeq}{\arr} \hocolim_{\sD} X.
\]
\end{itemize}
\end{input_data}

$\sD_{+}$ is the subcategory determining the relative level model structure that we start with.  For example, the positive level model structure corresponds to the case where $\sD_{+}$ is the full subcategory of $\sD$ without $0$.  Condition (i) is required in order for homotopy colimits over $\sD$ to be defined and is automatically satisfied if $\ob \sD$ is discrete.

We now define the  model structure on $\sD$-spaces relative to $(\sD, \sD_{+})$.  First consider the $\sD_{+}$-relative level model structure on $\sD\sU$, (also known as the relative projective model structure).  This has weak equivalences the $\sD_{+}$ level  equivalences ($X(d) \arr Y(d)$ a weak equivalence of spaces for each $d \in \ob \sD_{+}$), fibrations the level fibrations and cofibrations defined by the left lifting property (LLP) with respect to acylic fibrations.  The generating cofibrations and acyclic cofibrations are given as follows.  Let $I$ be the set of inclusions of spaces $i \colon S^n \arr D^{n + 1}$ for $n \geq -1$ (where $S^{-1} = \emptyset$), and let $J$ be the set of inclusions $i_0 \colon D^n \arr D^n \times [0,1]$ for $n \geq 0$.  The sets $I$ and $J$ are the generating cofibrations and acyclic cofibrations for the underlying model structure on $\sU$.  Given an object $d$ of $\sD$, the functor $F_d \colon \sU \arr \sD\sU$ is left adjoint to evaluation on the object $d$.  Define:
\[
F_{+}I = \{F_d i \mid d \in \ob\sD_{+}, i \in I\}, \quad \text{and} \quad F_{+}J = \{F_d j \mid d \in \ob\sD_{+}, j \in J \}.
\]
Then $F_{+}I$ and $F_{+}J$ are the generating cofibrations and acylic cofibrations for the $\sD_{+}$-relative level model structure on $\sD\sU$.

We will refer to the fibrations and weak equivalences of the $\sD_{+}$-relative level model structure as level fibrations and level equivalences (leaving reference to $\sD_{+}$ implicit).  The cofibrations of the level model structure will simply be called cofibrations as they coincide with the cofibrations in the new model structure.

A map $f \colon X \arr Y$ of $\sD$-spaces is a weak homotopy equivalence if the induced map of homotopy colimits over $\sD$ is a weak homotopy equivalence of spaces:
\[
f_* \colon \hocolim_{\sD} X \overset{\simeq}{\arr} \hocolim_{\sD} Y.
\]
A map $p \colon E \arr B$ of $\sD$-spaces is a fibration if it has the right lifting property (RLP) with respect to the acylic cofibrations.  Notice that since the homotopy colimit functor preserves tensors with spaces, pushouts and sequential colimits, the weak equivalences are well-grounded.

The main result of this section is:
\begin{theorem}\label{general_model_cat_theorem}  Suppose that $(\sD, \sD_{+})$ satisfies the hypotheses of Input Data \ref{model_structure_setup}.
Then the category of $\sD$-spaces is a compactly generated topological model category with respect to the cofibrations, fibrations and weak homotopy equivalences.  The model structure is constructed as a left Bousfield localization of the $\sD_{+}$-relative level model structure on $\sD\sU$.  
\end{theorem}

The case $\sD_{+} = \sD$ is the most important, and we will refer to that case as the absolute  model structure when necessary for clarity.  The examples we will need are:
\begin{enumerate}
\item $\sD = \sD_{+} = \bbI$: the \emph{(absolute) model structure} on $\bbI$-spaces.
\item $\sD = \bbI$, $\sD_{+} = \bbI_{>0}$: the \emph{positive model structure} on $\bbI$-spaces.
\item $\sD = \sD_{+} = \cI^{\dagger}$: the \emph{(absolute) model structure} on $\cI^{\dagger}$-spaces.
\item $\sD = \cI^{\dagger}$, $\sD_{+} = \cI^{\dagger}_{> 0}$: the \emph{positive model structure} on $\cI^{\dagger}$-spaces.
\end{enumerate}

\begin{remark}  We define the absolute and positive  model structures on $\cI$-spaces by transferring structure across the equivalence of categories $\cI \sU \arr \cI^{\dagger} \sU$ that is induced by the equivalence $\cI^{\dagger} \arr \cI$.  It is straightforward to see that the resulting model structures on $\cI$-spaces are well-defined and compactly generated by the images of the generating sets under the prolongation $\cI^{\dagger} \sU \arr \cI \sU$.
\end{remark}

\begin{proposition}\label{verify_axioms}
Examples (1) - (4) all satisfy the hypotheses of Input Data~\ref{model_structure_setup}.
\end{proposition}
\begin{proof}  In all cases, condition (i) is given by Construction \ref{module_from_Dspace_construction} and Lemma \ref{map_of_hocolim_is_quasifib}.  For (1) and (3), condition (ii) is automatically satisfied and (iii) follows from Lemma \ref{fibrant_hocolim_lemma}.  For (2) and (4), condition (ii) is Lemma \ref{general_inclusion into_hocolim_equiv} and condition (iii) is Lemma \ref{bokstedt_lemma}.
\end{proof}

The rest of this section consists of the proof of Theorem \ref{general_model_cat_theorem}.  Given a map $\phi \colon c \arr d$ in $\sD_{+}$, there is an induced natural transformation $\phi^* \colon F_{d} \arr F_{c}$ of functors $\sU \arr \sD\sU$.  Let $\lambda_\phi = \phi^*(*) \colon F_{d}(*) \arr F_c(*)$ be the component of this natural transformation at $*$.  By the Yoneda lemma, $F_d(*)$ is the represented $\sD$-space $\sD[d] = \sD(d, -)$.  Factor $\lambda_\phi \colon \sD[d] \arr \sD[c]$ into a cofibration $k_\phi$ followed by a level acyclic fibration $r_\phi$ using the mapping cylinder $M \lambda_\phi$ of $\lambda_\phi$ (defined level-wise):
\[
\lambda_\phi \colon \sD[d] \overset{k_\phi}{\arr} M \lambda_\phi \overset{r_\phi}{\arr} \sD[c].
\]
Starting with $k_\phi$ and any $i \colon S^n \arr D^{n + 1}$ in the set $I$ of generating cofibrations, passage to pushouts yields the pushout product:
\[
k_\phi \boxempty i \colon (\sD[d] \times D^{n + 1}) \cup_{\sD[d] \times S^n} (M \lambda_\phi \times S^n ) \arr M \lambda_\phi \times D^{n + 1}.
\]
Let $k_\phi \boxempty I = \{ k_\phi \boxempty i \mid i \in I \}$.  Notice that $\lambda_\phi$ is a weak homotopy equivalence, since $\hocolim_{\sD} \sD[d]$ and $\hocolim_{\sD} \sD[c]$ are both contractible.  Hence each $k_\phi$ is a weak homotopy equivalence.  Define $K$ to be the union of $F_{+}J$ and the sets $k_\phi \boxempty I$ over all morphisms $\phi$ of $\sD_{+}$.  $K$ will be the set of generating acyclic cofibrations and $F_{+}I$ will be the set of generating cofibrations for the  model structure on $\sD\sU$.  

\begin{proposition}\label{characterize_rlp_K}  A map $p \colon E \arr B$ satisfies the RLP with respect to $K$ if and only if $p$ is a level fibration and the induced map $E(c) \arr E(d) \times_{B(d)} B(c)$ is a weak equivalence for all $\phi \colon c \arr d$ in $\sD_{+}$.
\end{proposition}
\begin{proof}
Since $F_{+}J \subset K$ and level fibrations are precisely the maps satisfying the RLP with respect to $F_{+}J$, we assume that $p$ is a level fibration and then show that for each morphism $\phi$ of $\sD_{+}$, $p$ satisfies the RLP with respect to $k_\phi \boxempty I$ if and only if $E(c) \arr E(d) \times_{B(d)} B(c)$ is a weak homotopy equivalence.

By \cite{MMSS}*{5.16}, $p$ has the RLP with respect to $k_\phi \boxempty I$ if and only if the following map of morphism spaces has the RLP with respect to $I$: 
\[
\sD\sU(k_\phi^*, p_*) \colon \sD\sU(M \lambda_\phi, E) \arr \sD\sU(\sD[d], E) \times_{\sD\sU(\sD[d], B)} \sD\sU(M \lambda_\phi, B).
\]
The latter condition means that $\sD\sU(k_\phi^*, p_*)$ is an acyclic Serre fibration.  Since $k_\phi$ is a cofibration and $p$ is a level fibration, we know that $\sD\sU(k_\phi^*, p_*)$ is a Serre fibration because the level model structure is topological.  Hence $p$ satisfies the RLP with respect to $k_\phi \boxempty I$ if and only if $\sD\sU(k_\phi^*, p_*)$ is a weak equivalence of spaces.

Since the canonical fibration $r_\phi \colon M \lambda_\phi \arr \sD[c]$ is a weak equivalence, $\sD\sU(k_\phi^*, p_*)$ is a weak equivalence if and only if \[
\sD\sU(\lambda_\phi^*, p_*) \colon \sD\sU(\sD[c], E) \arr \sD\sU(\sD[d], E) \times_{\sD\sU(\sD[d], B)} \sD\sU(\sD[c], B)
\]
is a weak equivalence.  This map is isomorphic to the induced map to the pullback
\[
E(c) \arr E(d) \times_{B(d)} B(c),
\]
so the proof is complete.
\end{proof}

\begin{corollary}\label{characterize_rlp_K_fortrivialmaps}  The trivial map $F \arr \ast$ satisfies the RLP with respect to $K$ if and only if $F(\phi) \colon F(c) \arr F(d)$ is a weak equivalence for every $\phi \colon c \arr d$ in $\sD_{+}$.
\end{corollary}

We will need to understand how the fiber of a levelwise fibration $E \arr B$ relates to the fiber of $\hocolim_{\sD} E \arr \hocolim_{\sD} B$.  For the following lemma, we fix a map $b \colon * \arr B$, giving compatible baspoints in each space $B(d)$.

\begin{lemma}\label{hocolim_is_homotopy_fiber}
Suppose that $p \colon E \arr B$ is a level fibration of $\sD$-spaces with fiber $F = E_{b}$  over $b \colon * \arr B$.  Choose a point $* \in B\sD$  and also write $*$ for its image under the induced map of homotopy colimits
\[
B\sD = \hocolim_{\sD} (*) \overset{b}{\arr} \hocolim_{\sD} B.
\]
Then in the following morphism of homotopy fiber sequences
\[
\xymatrix{
F_{*}(q) \ar[r] \ar[d] & F_{*}(p) \ar[d] \\
\hocolim_{\sD} F \ar[d]_{q} \ar[r] & \hocolim_{\sD} E \ar[d]^{p} \\
B \sD \ar[r]_-{i_b} & \hocolim_{\sD} B }
\]
\begin{itemize}
\item[(i)]  the map $F_{*}(q) \arr F_{*}(p)$ of homotopy fibers is a weak equivalence, and
\item[(ii)]  $\hocolim_{\sD} F$ is weak homotopy equivalent to the homotopy fiber $F_*(p)$.
\end{itemize}
\end{lemma}
\begin{proof}  

Each space $F(d)$ is the pullback of the fibration $p(d) \colon E(d) \arr B(d)$ over the point $b \colon \ast \arr B(d)$.  These pullbacks assemble into a pullback square of left $\sD$-modules:
\addtocounter{theorem}{1}
\begin{equation}\label{pullback_of_Dmodules}
\xymatrix{ \sF \ar[d]_{q} \ar[r] & \sE \ar[d]^{p} \\
\ob \sD \ar[r] & \sB }
\end{equation}
This pullback induces a pullback of topological homotopy colimits:
\[
\xymatrix{ \hocolim_{\sD} F \ar[d]_{q} \ar[r] & \hocolim_{\sD} E \ar[d]^{p} \\
B \sD  \ar[r] & \hocolim_{\sD} B }
\]
By condition (i) on $(\sD, \sD_{+})$, both of the vertical arrows in \eqref{pullback_of_Dmodules} are fibrations, so we may apply Proposition \ref{map_of_hocolim_is_quasifib}.  It follows that $q_*$ and $p_*$ are quasifibrations.  Since the square is a pullback, the fibers of $q$ and $p$ over $*$ are isomorphic.  Thus the map of homotopy fibers $F_{*}(q) \arr F_{*}(p)$ is a weak homotopy equivalence, proving (i).  To prove (ii), recall that the category $\sD$ has an initial object, and so the classifying space $B \sD $ is contractible.  Thus the inclusion $F_{*}(q) \arr \hocolim_{\sD} F$ of the homotopy fiber of $q$ is a weak equivalence.  Along with the weak equivalence of homotopy fibers, this gives a chain of weak equivalences between $\hocolim_{\sD} F$ and $F_{*}(p)$.
\end{proof}

We can now establish the crucial step in setting up the model structure:

\begin{proposition}\label{RLP_K_and_stable}  Suppose that $p \colon E \arr B$ is a weak homotopy equivalence satisfying the RLP with respect to $K$.  Then $p$ is a level acyclic fibration.
\end{proposition}
\begin{proof}
As $p$ has the RLP with respect to $F_{+}J$, it is a level fibration.  We must show that it is a level equivalence.  Choose a map $b \colon * \arr B$ and let $F = E_b$ be the pullback of $p \colon E \arr B$ over $b$.  Lemma \ref{hocolim_is_homotopy_fiber} implies that $\hocolim_{\sD} F$ is equivalent to the homotopy fiber of a weak equivalence:
\[
\hocolim_{\sD} F \arr \hocolim_{\sD} E \overset{\simeq}{\arr} \hocolim_{\sD} B.
\]
Hence $\pi_*\hocolim_{\sD} F = 0$.  By the pullback square, $F \arr \ast$ satisfies the RLP with respect to $K$.  Corollary \ref{characterize_rlp_K_fortrivialmaps} implies that the maps $F(\phi)$ are weak homotopy equivalences for all morphisms $\phi$ of $\sD_{+}$.  By assumption (iii) on $(\sD, \sD_{+})$, for every object $d$ of $\sD_{+}$, the inclusion $F(d) \arr \hocolim_{\sD} F$ is a weak homotopy equivalence.  Thus $\pi_* F(d) = 0$.  This means that the maps $p(d) \colon E(d) \arr B(d)$ are weak homotopy equivalences for all objects $d$ of $\sD_{+}$, proving that $p$ is a level equivalence.
\end{proof}

\begin{lemma}\label{stable_acylic_cofib_kcell}  A retract of a relative $K$-cell complex is an acyclic cofibration.
\end{lemma}
\begin{proof}
The maps in $K$ are all weak equivalences and $h$-cofibrations.  Since the weak equivalences are well grounded, this implies that every retract of a relative $K$-cell complex is a weak equivalence.  The maps in $K$ are also cofibrations, so by the closure properties of cofibrations in the level model structure, every retract of a relative $K$-cell complex is a cofibration.  
\end{proof}

The proof of Theorem \ref{general_model_cat_theorem} will be completed by invoking the following criterion for compactly generated model categories:

\begin{theorem}\label{model_cat_construction_quoted}\cite{parametrized}*{4.5.6}  Suppose that $\sC$ is a bicomplete category with a subcategory of weak equivalences satisfying the two out of three property.  Let $I$ and $J$ be compact sets of maps in $\sC$ satisfying the following two conditions:
\begin{itemize}
\item[(a)] Every relative $J$-cell complex is a weak equivalence.
\item[(b)] A map has the RLP with respect to $I$ if and only if it is a weak equivalence and has the RLP with respect to $J$.
\end{itemize}
Then $\sC$ is a compactly generated model category with generating cofibrations $I$ and generating acyclic cofibrations $J$.
\end{theorem}

Here a set of maps $I$ is said to be compact if for every domain object $X$ of a map in $I$ and every relative $I$-cell complex $Z_0 \arr Z$, the induced map $\colim \sC(X, Z_n) \arr \sC(X, Z)$ is an isomorphism.  

In our situation, the generating cofibrations are $F_{+}I$ and the generating acyclic cofibrations are $K$.  Using the adjunction between $F_n$ and evaluation at level $n$, compactness of $F_{+}I$ and $K$ follows from compactness of spheres and disks.  Condition (a) follows from Lemma \ref{stable_acylic_cofib_kcell}.  Proposition \ref{RLP_K_and_stable} shows that weak homotopy equivalences satisfying the RLP with respect to $K$ satisfy the RLP with respect to $F_{+}I$.  This is one direction of condition (b).  For the other, suppose that $f$ satisfies the RLP with respect to $F_{+}I$, i.e. $f$ is a level acylic fibration.  Since $f$ is a level equivalence, it is a weak homotopy equivalence, so we only need to show that $f$ satisfies the RLP with respect to $K$.  Using the level model structure, $f$ satisfies the RLP with respect to cofibrations.  It follows from Lemma \ref{stable_acylic_cofib_kcell} that $f$ satisfies the RLP with respect to relative $K$-cell complexes.  In particular, $f$ satisfies the RLP with respect to $K$.  Thus conditions (a) and (b) are satisfied, so Theorem \ref{model_cat_construction_quoted} gives the compactly generated model structure on $\sD$-spaces.  This concludes the proof of Theorem \ref{general_model_cat_theorem}.

\section{The equivalence of $\bbI$-spaces and $\cI$-spaces}

We will now prove that the prolongation and forgetful functors comprise a Quillen equivalence between $\bbI$-spaces and $\cI$-spaces.  We first record a standard consequence of B\"oksted's Lemma (\ref{bokstedt_lemma}), known as the ``telescope lemma''.

\begin{lemma}\label{telescope}  Suppose that $X$ is an $\bbI$-space such that $X(\bn)$ is $\lambda_n$-connected, where $\{\lambda_n\}$ is an unbounded nondecreasing sequence of integers.  Then the inclusion of categories $\bbJ \arr \bbI$ induces a weak homotopy equivalence of homotopy  colimits:
\[
\hocolim_{\bbJ} X \overset{\simeq}{\arr} \hocolim_{\bbI} X.
\]
\end{lemma}

\begin{lemma}\label{stiefel_manifold_lemma}  Let $\cI[n] = \cI(\bR^n, -)$ be the $\cI$-space represented by $\bR^n$.  Restricting $\cI[n]$ to an $\bbI$-space, the inclusion of categories $\bbJ \arr \bbI$ induces a weak homotopy equivalence $\hocolim_{\bbJ} \cI[n] \arr \hocolim_{\bbI} \cI[n]$.  Furthermore, $\hocolim_{\bbI} \cI[n]$ is contractible.
\end{lemma}
\begin{proof}  We need to show that the following map of homotopy colimits is a weak homotopy equivalence:
\[
\underset{k \in \bbJ}{\hocolim} \; \cI(\bR^n, \bR^k) \arr \underset{k \in \bbI}{\hocolim} \;  \cI(\bR^n, \bR^k).
\]
The space $\cI(\bR^n, \bR^k)$ is the Stiefel manifold of $n$-frames in $\bR^k$, which is $(k - n - 1)$-connected.  Thus, $\cI[n]$ satisfies the conditions for Lemma \ref{telescope}, which gives the weak homotopy equivalence of homotopy colimits.  

For the second claim, notice that the maps $\cI(\bR^n, \bR^k) \arr \cI(\bR^n, \bR^{k + 1})$ are closed inclusions of manifolds.  Therefore, there is a weak homotopy equivalence 
\[
\cI_c(\bR^n, \bR^{\infty}) = \colim_{\bbJ} \cI[n] \overset{\simeq}{\arr} \hocolim_{\bbJ} \cI[n].
\]
Since the space of isometries $\cI_c(\bR^n, \bR^{\infty})$ is contractible, the result follows.
\end{proof}

\begin{proposition}\label{eta_equiv}
Let $X$ be a cofibrant $\bbI$-space.  Both of the following maps are weak homotopy equivalences:
\begin{itemize}
\item[(i)] the map $\hocolim_{\bbJ} \bbP X \arr \hocolim_{\bbI} \bbP X$ induced by the inclusion of categories $\bbJ \arr \bbI$.
\item[(ii)]  the unit $\eta \colon X \arr \bbU \bbP X$ of the adjunction $(\bbP, \bbU)$.
\end{itemize}
\end{proposition}
\begin{proof}
The  $\bbP$, $\bbU$, and the homotopy colimit functors all commute with colimits, including tensors with spaces.  Since weak equivalences are well-grounded, we may assume that $X$ is an $FI$-cell complex and induct up the cell structure.  Hence it suffices to prove ether claim for a represented $\bbI$-space $X = F^{\bbI}_n(*) = \bbI[\bn]$.  Since their right adjoints are isomorphic by inspection, there is a natural isomorphism $\bbP F^{\bbI}_n \cong F^{\cI}_{\bR^n}$.  Hence (i) follows from Lemma \ref{stiefel_manifold_lemma}.  For (ii), we need to show that $\eta \colon \bbI[\bn] \arr \cI[n]$ is a weak homotopy equivalence.  We have the canonical level equivalence $\epsilon \colon B(*, \bbI, \bbI) \arr *$, so $\hocolim_{\bbI} \bbI[\bn]$ is contractible.  By Lemma \ref{stiefel_manifold_lemma}, the target $\hocolim_{\bbI} \cI[n]$ is also contractible.  Therefore $\eta$ is a weak homotopy equivalence.
\end{proof}

\begin{theorem}\label{equivalence_of_I_and_If_spaces}
The prolongation functor $\bbP \colon \bbI\sU \arr \cI\sU$ and the forgetful functor $\bbU \colon \cI \sU \arr \bbI \sU$ induce a Quillen equivalence of the stable model structures.
\end{theorem}
\begin{proof}
$\bbU$ preserves fibrations by the characterization of fibrations given in Proposition \ref{characterize_rlp_K}.   Acylic fibrations and level acyclic fibrations coincide and $\bbU$ preserves level equivalences, so $\bbU$ preserves acylic fibrations.  Thus $(\bbP, \bbU)$ is a Quillen adjunction.  

By \cite{hovey_model_categories}*{1.3.16}, $(\bbP, \bbU)$ is a Quillen equivalence if $\bbU$ detects weak equivalences between fibrant objects and for cofibrant $\bbI$-spaces $X$, the composite
\addtocounter{theorem}{1}
\begin{equation}\label{composite_for_quillen_equiv}
X \overset{\eta}{\arr} \bbU \bbP X \xrightarrow{\bbU r} \bbU R \bbP X
\end{equation}
of the unit of the adjunction with the map induced by fibrant approximation $r \colon \bbP X \arr R \bbP X$ is a weak equivalence.  Suppose that $f \colon X \arr Y$ is a map of fibrant $\cI$-spaces such that $\bbU f$ is a weak equivalence.  By Lemma \ref{fibrant_hocolim_lemma}, the vertical arrows in the following commutative diagram are weak homotopy equivalences:
\[
\xymatrix{
\hocolim_{\cI} X \ar[r]^{f} & \hocolim_{\cI} Y \\
X(0) \ar[u]^{\simeq} \ar[d]_{\simeq} \ar[r] & Y(0) \ar[u]_{\simeq} \ar[d]^{\simeq} \\
\hocolim_{\bbI} \bbU X \ar[r]^{\bbU f} & \hocolim_{\bbI} \bbU Y }
\]
Hence the top map is a weak homotopy equivalence, so $\bbU$ detects weak equivalences between fibrant objects.

By Proposition \ref{eta_equiv}.(ii), $\eta$ is a weak homotopy equivalence, so we just need to show that $\bbU r$ is a weak homotopy equivalence.  Consider the following commutative diagram:
\[
\xymatrix{
\hocolim_{\bbI} \bbP X \ar[r]^{\bbU r} & \hocolim_{\bbI} R\bbP X & \\
\hocolim_{\bbJ} \bbP X \ar[u]^{i_*} \ar[d] \ar[r] & \hocolim_{\bbJ} R\bbP \ar[u] \ar[d] X & (R\bbP X)(0) \ar[l]_-{\simeq} \ar[ul]_-{\simeq} \ar[dl]^-{\simeq} \ar[ddl]^-{\simeq} \\
\hocolim_{\cJ} \bbP X \ar[r] \ar[d] & \hocolim_{\cJ} R\bbP X \ar[d] & \\
\hocolim_{\cI} \bbP X \ar[r]^-{r} & \hocolim_{\cI} R\bbP X & }
\]
The four maps from $(R\bbP X)(0)$ are weak homotopy equivalences by Lemma \ref{fibrant_hocolim_lemma}.  The fibrant approximation map r induces a weak homotopy equivalence by definition.  The three vertical maps on the left are weak homotopy equivalences by Proposition \ref{eta_equiv}.(i), Lemma \ref{theorem_A_hocolim_version}, and Lemma \ref{untwisting_prop}.  It follows that $\bbU r$ is also a weak homotopy equivalence.
\end{proof}

\section{The model structure on FCPs}\label{model_structure_monoids_section}

In this section we will construct the model structure on the category of $\sD$-FCPs, where $\sD = \bbI$ or $\cI$.  The main technical point is the following lemma.

\begin{lemma}\label{times_cofibrants_preserves_equiv}
If $X$ is a cofibrant $\sD$-space, then the functor $X \boxtimes_{\sD} (-)$ preserves weak homotopy equivalences.  
\end{lemma}
\begin{proof}

We may assume that $X$ is an $FI$ cell complex.  Since applying $(-) \boxtimes Y$ to an $h$-cofibration is again an $h$-cofibration, and weak equivalences of $\sD$-spaces are well-grounded, we may induct over the cell structure of $X$.  It now suffices to prove the result when $X = F_d(*) = \sD[d]$.

Let $Y$ be a $\sD$-space.  By a comparison of right adjoints, we have a natural isomorphism $(\sD[d] \boxtimes Y)(n) \cong \sD(d \oplus - , n) \otimes_{\sD} Y$.  Write $\Aut(c) = \Aut_{\sD}(c)$ for the group of automorphisms of an object $c$ of $\sD$, and notice that $\Aut(c)$ is a compact Lie group in both of our examples.  By the coequalizer description of the coend, there is an isomorphism
\addtocounter{theorem}{1}
\begin{equation}\label{monoid_axiom_prep_identification}
\sD(d \oplus - , n) \otimes_{\sD} Y \cong \Aut(n) \times_{\Aut(c)} Y(c),
\end{equation}
where $c$ is an object of $\sD$ with a chosen isomorphism $d \oplus c \cong n$.  The group $\Aut(c)$ acts on $\Aut(n)$ via the group homomorphism $d \oplus - \colon \Aut(c) \arr \Aut(n)$, and the isomorphism is natural in the variable $n$.  Evaluating the level-wise homotopy equivalence $\epsilon \colon B(\sD, \sD, Y) \arr Y$ at $c$ induces a map 
\[
\id \times \epsilon(c) \colon \Aut(n) \times_{\Aut(c)} B(\sD(-, c), \sD, Y) \arr \Aut(n) \times_{\Aut(c)} Y(c)
\]
of fiber bundles over $\Aut(n)/\Aut(c)$.  Since it is a homotopy equivalence on each fiber, $\id \times\epsilon(c)$ is a homotopy equivalence.  The definition of $\epsilon$ and naturality give a commutative diagram
\[
\xymatrix{
\Aut(n) \times_{\Aut(c)} B(\sD(-, c), \sD, Y) \ar[rr]^-{\id \times \epsilon(c)} \ar[d]_{\cong} & & \Aut(n) \times_{\Aut(c)} Y(c) \ar[d]^{\cong} \\
B(\sD(d \oplus -, n), \sD, Y) \ar[rr]_{\pi} & & \sD(d \oplus -, n) \otimes_{\sD} Y }
\]
where the right vertical arrow is the identification \eqref{monoid_axiom_prep_identification} and the left vertical arrow passes $\Aut(n) \times_{\Aut(c)} (-)$ through the bar construction and then uses \eqref{monoid_axiom_prep_identification} (with $Y = \sD(d', -)$) level-wise.  It follows that $\pi$ is a homotopy equivalence as well.  

We will now consider the homotopy colimit of the map $\pi$ over $n$.  The canonical interchange isomorphism and the level-wise homotopy equivalence $B(*, \sD, \sD) \simeq *$ give a homotopy equivalence:
\[
B(*, \sD, B(\sD(d \oplus -, -), \sD, Y)) \cong B(B(*, \sD, \sD(d \oplus -, -)), \sD, Y) \arr B(*, \sD, Y).
\]
All together, we have constructed a natural chain of homotopy equivalences 
\[
\hocolim_{\sD} \sD[d] \boxtimes Y \simeq \hocolim_{\sD} Y.
\]
Therefore $\sD[d] \boxtimes (-)$ preserves weak equivalences and the proof is complete.
\end{proof}

We can now deduce the monoid axiom and the pushout-product axiom.

\begin{proposition}[Monoid Axiom]
For any acyclic cofibration $i \colon A \arr X$ and any $\sD$-space $Y$, the induced map $i \boxtimes \id_Y \colon A \boxtimes Y \arr X \boxtimes Y$ is a weak homotopy equivalence and an $h$-cofibration.  Furthermore, cobase changes and sequential colimits of such maps are also weak homotopy equivalences and $h$-cofibrations. 
\end{proposition}
\begin{proof}  We may assume that $i$ is a relative $K$-cell complex.  Since every cofibration is in particular an $h$-cofibration, by inducting over the cell structure of $i$, we are reduced to the case when $i$ is a generating acylic cofibration.  In particular, $i$ is an $h$-cofibration so $i \boxtimes \id_Y$ is as well.  Let $q \colon Y' \arr Y$ be a cofibrant approximation of $Y$.  Since the domain and codomain of the generating cell $i \colon A \arr X$ are cofibrant, Lemma \ref{times_cofibrants_preserves_equiv} proves that $\id_A \boxtimes q$, $\id_X \boxtimes q$, and $i \boxtimes \id_{Y'}$ are all weak equivalences.  It follows that $i \boxtimes \id_Y$ is a weak equivalence as well.  The second claim follows since weak homotopy equivalences of $\sD$-spaces are well-grounded.
\end{proof}

The pushout-product axiom follows from the monoid axiom, as in \cite{MMSS}*{12.6}.  This completes the proof that $\sD$-spaces is a monoidal model category.

\begin{proposition}[Pushout-Product Axiom]\label{pushout_product_axiom_prop}
Let $i \colon A \arr X$ and $j \colon Y \arr Z$ be cofibrations of $\sD$-spaces, and assume that $i$ is a weak homotopy equivalence.  Then the pushout-product
\[
i \boxempty j \colon (X \boxtimes Y) \cup_{A \boxtimes Y} (A \boxtimes Z) \arr X \boxtimes Z
\]
is a weak homotopy equivalence.
\end{proposition}

As in the proof of \cite{MMSS}*{12.1}, we can now deduce the following result using a version of \cite{schwede_shipley}*{4.1} for compactly generated topological model categories.

\begin{theorem}For $\sD = \bbI$ and $\cI$, the category of $\sD$-FCPs is a compactly generated topological model category with fibrations and weak equivalences created by the forgetful functor to $\sD$-spaces.  A cofibration of $\sD$-FCPs whose source is a cofibrant $\sD$-space is a cofibration of $\sD$-spaces.  Since the unit $\sD$-space $*$ is cofibrant, it follows that every cofibrant $\sD$-FCP is cofibrant as a $\sD$-space,
\end{theorem}

The following result is an immediate consequence of the Quillen equivalence between $\bbI$-spaces and $\cI$-spaces.

\begin{theorem}\label{quillen_equiv_monoids}  The adjunction $(\bbP, \bbU)$ restricts to give a Quillen equivalence between the categories of $\bbI$-FCPs and $\cI$-FCPs.
\end{theorem}

\section{The model structure on commutative FCPs}

We will now construct the positive model structure on commutative $\bbI$-FCPs and commutative $\cI$-FCPs, then show that they are Quillen equivalent.  The arguments are formally similar to those of \cite{MMSS}, and we will not go into full detail when unnecessary.  In order to state the main results in this section, write $F_{+} I$ and $K_{+}$ for the generating cofibrations and generating acyclic cofibrations for the positive model structure on $\sD$-spaces.  The sets $\bC F_{+} I$ and $\bC K_{+}$ result from applying the free commutative monoid functor $\bC$ to the elements of $F_{+} I$ and $K_{+}$.  In this section, we will use $\sD$ to denote either $\bbI$ or $\cI$.

\begin{theorem}\label{existence_commFCP_modelstructure_thm}
The category $\bC \sD \sU$ of commutative $\sD$-FCPs is a compactly generated topological model category with fibrations and weak equivalences created in the positive stable model structure on $\sD$-spaces.  The set of generating cofibrations if $\bC F_{+}I$ and the set of generating acyclic cofibrations is $\bC K_{+}$.
\end{theorem}

\begin{theorem}\label{equivalence_of_commutative_I_and_If_spaces}
The prolongation functor $\bbP$ and forgetful functor $\bbU$ induce a Quillen equivalence between the categories of commutative $\bbI$-FCPs and commutative $\cI$-FCPs.
\end{theorem}

We will employ an alternative description of the product $\boxtimes$ in the categories of $\bbI$-spaces and $\cI$-spaces.  Recall the diagram category $\Sigma$ with objects $\bn$ and morphisms the permutations.  Consider the category $\Sigma \sU$ of unbased $\Sigma$-spaces.  Using the cartesian monoidal structure of $\sU$, we have the symmetric monoidal product $\boxtimes_{\Sigma}$ on $\Sigma \sU$ defined by left Kan extension of the external cartesian product along $\oplus \colon \Sigma \times \Sigma \arr \Sigma$.  Let $\ast$ be the commutative monoid in $\Sigma \sU$ defined by $\ast(\mathbf{n}) = \ast$ for all $\bn$.  The product of $\ast$-modules $X \boxtimes_{*} Y$ is defined as the coequalizer of $\Sigma$-spaces:
\[
\xymatrix{
X \boxtimes_{\Sigma} \ast \boxtimes_{\Sigma} Y \ar@<-.5ex>[r] \ar@<.5ex>[r] & X \boxtimes_{\Sigma} Y \ar[r]& X \boxtimes_{*} Y.
}
\]
To avoid confusion, we will temporarily write the internal product of $\bbI$-spaces as $\boxtimes_{\bbI}$.  

\begin{proposition}  The category $\bbI \sU$ of unbased $\bbI$-spaces is isomorphic to the category of $\ast$-modules in $\Sigma \sU$.  Furthermore, this isomorphism is symmetric monoidal: for $\bbI$-spaces $X$ and $Y$, their product $X \boxtimes_{\bbI} Y$ as $\bbI$-spaces is naturally isomorphic to their product $X \boxtimes_{*} Y$ as $\ast$-modules.
\end{proposition}
\begin{proof}
A $\ast$-module $X$ consists of an underlying functor $\Sigma \arr \sU$ along with associative and unital natural transformations
\[
\ast \times X(\bm) \arr X(\bn).
\]
These give the maps $X(\iota) \colon X(\bm) \arr X(\bn)$ that define $X$ on the canonical inclusions $\iota \colon \bm \arr \bn$ of $\bbI$.  Since every morphism in $\bbI$ can be factored as a canonical inclusion followed by a permutation, this gives the extension of $X$ to an $\bbI$-space.  Conversely, for an $\bbI$-space $X$ the maps $X(\iota) \colon X(\bm) \arr X(\bn)$ define a $\ast$-module structure on the underlying $\Sigma$-space.  This correspondence of structures is functorial.

It remains to give a natural isomorphism $X \boxtimes_{\bbI} Y \cong X \boxtimes_{*} Y$ for $\bbI$-spaces $X$ and~$Y$.  Both sides of these bifunctors  have right adjoints defined by internal function objects and thus preserve colimits.  On the other hand, every $\bbI$-space is a colimit of represented $\bbI$-spaces $\bbI[\bm] = \bbI(\bm, - )$.  Consequently, it suffices to prove the result for represented $\bbI$-spaces:
\[
\bbI[\bm] \boxtimes_{\ast} \bbI[\bn] \cong \bbI[\bm] \boxtimes_{\bbI} \bbI[\bn].
\]
A long series of adjunctions shows that $\bbI[\bm] \boxtimes_{\bbI} \bbI[\bn] \cong \bbI[\bm \oplus \bn]$, and the analogous result is true for represented $\Sigma$-spaces $\Sigma[\bm] = \Sigma(\bm, -)$ (this is essentially \cite{MMSS}*{Lemma 1.8}).  As a $\ast$-module, $\bbI[\bm]$ is the free $\ast$-module $\ast \boxtimes_{\Sigma} \Sigma[\bm]$.  The desired isomorphism follows:
\begin{align*}
\bbI[\bm] \boxtimes_{\ast} \bbI[\bn] &\cong (\ast \boxtimes_{\Sigma} \Sigma[\bm] ) \boxtimes_{\ast} (\ast \boxtimes_{\Sigma} \Sigma[\bn] ) \\
&\cong \ast \boxtimes_{\Sigma} \Sigma[\bm \oplus \bn] \\
&\cong \bbI[\bm \oplus \bn] \\
&\cong \bbI[\bm] \boxtimes_{\bbI} \bbI[\bn].
\end{align*}

\end{proof}

The point of considering $\bbI$-spaces as $\ast$-modules is that it makes the computation of the internal cartesian product of $\bbI$-spaces much easier.  The product of $\Sigma$-spaces is given by the formula:
\[
(X \boxtimes_{\Sigma} Y)(\bm) = \coprod_{a + b = m} \Sigma_{m} \times_{\Sigma_a \times \Sigma_b} X(\ba) \times Y(\bb).
\]
We use the coequalizer definition of $\boxtimes_{*}$ to deduce the following:
\begin{lemma}\label{computation_of_bbIproduct}  Let $X$ and $Y$ be $\bbI$-spaces.  Then $(X \boxtimes_{\bbI} Y)(\bm)$ is the coequalizer of the following diagram:
\[
\xymatrix{
\underset{a + b = m}{\displaystyle\coprod} \Sigma_m \times_{\Sigma_a \times \Sigma_b \times \Sigma_c} X(\ba) \times Y(\bc)
 \ar@<.5ex>[r] \ar@<1.5ex>[r] & 
\underset{a + b = m}{\displaystyle\coprod} \Sigma_{m} \times_{\Sigma_a \times \Sigma_b} X(\ba) \times Y(\bb).
}
\]
\end{lemma}
Here the top map is defined using the inclusion $\Sigma_a \times \Sigma_b \arr \Sigma_{a + b}$ and the map 
\[
X(\id_\ba \oplus \iota) \colon X(\ba) = X(\ba \oplus \mathbf{0}) \arr X(\ba \oplus \bb),
\]
while the bottom map is defined using $\Sigma_b \times \Sigma_c \arr \Sigma_{b + c}$ and 
\[
Y(\iota \oplus \id_{\bc}) \colon Y(\bc) = Y(\mathbf{0} \oplus \bc) \arr Y(\bb \oplus \bc).
\]

\medskip

We have the corresponding results for $\cI$-spaces, proved in the same way.  Let $\cO$ be the category of finite dimensional inner product spaces and linear isometric isomorphisms.  

\begin{proposition}  The category $\cI \sU$ of $\cI$-spaces is isomorphic to the category of $\ast$-modules in $\cO \sU$.  Furthermore, this isomorphism is monoidal: for $\cI$-spaces $X$ and $Y$, their product $X \boxtimes_{\cI} Y$ as $\cI$-spaces is naturally isomorphic to their product $X \boxtimes_{*} Y$ as $\ast$-modules.
\end{proposition}

The full subcategory of $\cO$ consisting of the inner product spaces $\bbR^n$ is a skeletal subcategory, so the analog of the coequalizer in Lemma \ref{computation_of_bbIproduct} may be computed in the following way.  Make the abbreviations $O(n) = O(\bbR^n)$ and $X(n) = X(\bbR^n)$ for an $\cI$-space $X$.

\begin{lemma}\label{computation_of_I_f_product}  Let $X$ and $Y$ be $\cI$-spaces.  Then $(X \boxtimes_{\cI} Y)(m)$ is naturally isomorphic to the coequalizer of the following diagram:
\[
\xymatrix{
\underset{a + b + c = m}{\displaystyle\coprod} O(m) \times_{O(a) \times O(b) \times O(c)} X(a) \times Y(c)
 \ar@<.5ex>[d] \ar@<1.5ex>[d] \\
\underset{a + b = m}{\displaystyle\coprod} O(m) \times_{O(a) \times O(a)} X(a) \times Y(b).
}
\]
\end{lemma}

The following lemma is an analog of \cite{MMSS}*{15.5}

\begin{lemma}\label{prep_for_extended_power_prop}   Let $\sD = \bbI$ or $\cI$.  Let $n \geq 1$, and suppose that $K$ is a $\Sigma_n$-equivariant CW complex.  Let $d \neq 0$ be an object of $\sD$ and let $X$ be a $\sD$-space.  Then the quotient map
\[
E\Sigma_n \times_{\Sigma_n} (F_{d}(\ast)^{\boxtimes n} \times K)  \boxtimes X \arr (F_{d} (\ast)^{\boxtimes n} \times K) / \Sigma_n  \boxtimes X
\]
is a level-wise homotopy equivalence.  
\end{lemma}
\begin{proof} 
We will give the proof for $\cI$-spaces; the argument for $\bbI$-spaces is similar.  For ease of notation, we assume without loss of generality that the object $d$ is of the form $\bbR^d$.  We will continue to write $F_d$ for the left adjoint of evaluation at $\bbR^d$.  There is an isomorphism of $\cI$-spaces $F_d(\ast)^{\boxtimes n} \cong F_{nd} (\ast)$ \cite{MMSS}*{Lemma 1.8}.  Hence $(F_d(\ast)^{\boxtimes n} \times K)(a) \cong \cI( \bR^{nd}, \bR^a) \times K = O(a) \times_{O(a - nd)} K$.  By the description of $\boxtimes$ in Lemma \ref{computation_of_I_f_product}, $(F_d(\ast)^{\boxtimes n} \times K \boxtimes X )(m)$ is the coequalizer of:
\[
\xymatrix{
\underset{a + b + c = m}{\displaystyle\coprod} O(m) \times_{O(a - nd) \times O(b) \times O(c)} K \times X(c)
 \ar@<.5ex>[d] \ar@<1.5ex>[d] \\
\underset{a + b = m}{\displaystyle\coprod} O(m) \times_{O(a - nd) \times O(b)} K \times X(b).
}
\]
In the coequalizer, all summands are identified with the $(a, b) = (nd, m - nd)$ summand, which is left unchanged, so we have:
\[
((F_d(\ast)^{\boxtimes n} \times K) \boxtimes X )(m) \cong O(m) \times_{O(m - nd)} K \times X(m - nd).
\]
The group $\Sigma_n$ acts on $K$ and acts on $O(nd)$ by permuting the summands in $(\bbR^d)^n$, and thus acts on $O(m)$ via the inclusion $O(nd) \arr O(m)$.  Passing to orbits, we have:
\[
((F_d(\ast)^{\boxtimes n} \times K) / \Sigma_n  \boxtimes X)(m) \cong O(m) \times_{\Sigma_n \times O(m - nd)} K \times X(m - nd),
\]
and similarly:
\[
(E\Sigma_n \times_{\Sigma_n} (F_d(\ast)^{\boxtimes n} \times K)  \boxtimes X)(m) \cong (E\Sigma_n \times O(m)) \times_{\Sigma_n \times O(m - nd)} K \times X(m - nd ).
\]
The quotient map $E \Sigma_n \times O(m) \arr O(m)$ is a $(\Sigma_n \times O(m - nd))$-equivariant homotopy equivalence.  This proves the lemma.
\end{proof}

In order to make inductive arguments over cell attachments, we will use a certain filtration on the pushout $\bC B \cup_{\bC A} X$ of a commutative FCP along a free map of commutative FCPs $\bC f \colon \bC A \arr  \bC B$.  We first describe a filtration on the $n$-fold $\boxtimes$-product of a pushout of $\sD$-spaces.  This material is described in more generality in \cite{ss}*{\S A.6, A.15} and is also related to the filtration in \cite{EM}*{\S 12}.

Given two maps $f \colon A \rightarrow B$ and $g \colon A' \arr B'$, the pushout product $f \boxempty g$ is the induced map
\[
f \boxempty g \colon B \boxtimes A' \cup_{A \boxtimes A'} A \boxtimes B' \arr B \boxtimes B'.
\]
We write $f^{\boxempty n} \colon Q^{n}f \arr B^{\boxtimes n}$ for the $n$-fold iterated pushout product of $f$.  Now let $X \overset{g}{\longleftarrow} A \overset{f}{\arr} B$ be a diagram of $\sD$-spaces and write $P(g, f)$ for its pushout.  
\begin{lemma}\label{filtration_on_product_of_pushouts}
There is a sequence of $\sD$-spaces $P_{i}^{n}(g, f)$ and maps
\[
X^{\boxtimes n} = P_{0}^{n}(g, f) \arr P_{1}^{n}(g, f) \arr \dotsm \arr P_{n}^{n}(g, f) \cong P(g, f)^{\boxtimes n}
\]
whose composite is the canonical map $X^{\boxtimes n} \arr P(g, f)^{\boxtimes n}$.  The spaces $P_{i}^{n}(g, f)$ can be inductively described by $\Sigma_n$-equivariant pushout squares of the form
\[
\xymatrix{
\Sigma_{n} \times_{\Sigma_{n - i} \times \Sigma_{i}} X^{\boxtimes (n - i)} \boxtimes Q^{i} f \ar[rr]^{\id \boxtimes f^{\boxempty i}} \ar[d] & & \Sigma_{n} \times_{\Sigma_{n - i} \times \Sigma_{i}} X^{\boxtimes (n - i)} \boxtimes B^{\boxtimes i} \ar[d] \\
P_{i - 1}^{n}(g, f) \ar[rr] & & P_{i}^{n}(g, f)}
\]
Furthermore, if $f$ is a generating positive cofibration in $F_{+}I$, then the maps $\id \boxtimes f^{\boxempty i}$ and $P_{i - 1}^{n}(g, f) \arr P_{i}^{n}(g, f)$ are $h$-cofibrations of $\sD$-spaces.
\end{lemma}
\noindent See \cite{ss}*{A.8} for a construction of the filtration.  The claim about $h$-cofibrations follows from Lemma \ref{boxempty_is_hcofibration} below.

Let $f \colon A \arr B$ be a map of $\sD$-spaces, and let $X$ and $Y$ be commutative $\sD$-FCPs.  Consider the following pushout diagram of commutative $\sD$-FCPs
\[
\xymatrix{
\bC A \ar[d]_{\bC f} \ar[r] & X \ar[d]^{\overline{f}} \\
\bC B \ar[r] & Y}
\]
in which the map $\bC A \arr X$ is induced by a map of $\sD$-spaces $g \colon A \arr X$.  

\begin{lemma}\label{filtration_on_commutative_monoid_pushout}
There is a sequence of $\sD$-spaces
\[
X = P_{0} Y \arr P_{1} Y \arr \dotsm \arr P_{n} Y \arr \dotsm
\]
whose transfinite composition is the canonical map $\overline{f} \colon X \arr Y$.  The $\sD$-spaces $P_{n} Y$ can be inductively described by pushout squares of the form
\[
\xymatrix{
X \boxtimes Q^{n}f / \Sigma_{n} \ar[rr]^-{\id \boxtimes f^{\boxempty n} / \Sigma_n} \ar[d] & & X \boxtimes B^{\boxtimes n}/ \Sigma_{n} \ar[d] \\
P_{n - 1} Y \ar[rr] & & P_{n} Y 
}
\]
where $f^{\boxempty n} \colon Q^n f \arr B^{\boxtimes n}$ is the $n$-fold iterated pushout-product map.  Furthermore, if $f$ is a coproduct of generating positive cofibrations in $F_{+} I$, then the maps $\id \boxtimes f^{\boxempty}/\Sigma_n$ and $P_{n - 1}Y \arr P_{n}Y$ are $h$-cofibrations of $\sD$-spaces.
\end{lemma} 
\begin{proof}  For a proof in greater generality, see \cite{ss}*{A.16}.  Let $X \cup_{A} B$ and $\bC X \cup_{A} B$ be the pushouts of the diagrams
\[
X \overset{g}{\longleftarrow} A \overset{f}{\longrightarrow} B \qquad \text{and} \qquad \bC X \overset{\eta}{\longleftarrow} X \overset{f}{\longleftarrow} A \overset{g}{\arr} B
\]
in the category of $\sD$-spaces.  The FCP $Y$ is canonically isomorphic to the coequalizer of $\sD$-spaces
\[
\xymatrix{
\bC(\bC X \cup_{A} B)  \ar@<-.5ex>[r]_{\alpha_X} \ar@<.5ex>[r]^{\mu} & \bC(X \cup_{A} B) \ar[r] & Y,
}
\]
where $\mu$ is the composite
\[
\bC( \bC X \cup_{A} B) \arr \bC( \bC X \cup_{\bC A} \bC B) \arr \bC \bC(X \cup_{A} B) \overset{\mu}{\arr} \bC(X \cup_{A} B)
\]
and $\alpha_X$ is induced by the commutative FCP structure map $\alpha_X \colon \bC X \arr X$.  The filtration of the $n$-fold $\boxtimes$ products of the pushouts $X \cup_{A} B$ and $\bC X \cup_{A} B$ from Lemma \ref{filtration_on_product_of_pushouts} yields a filtration $P_i Y$ of $Y$ given by the coequalizer diagrams
\[
\xymatrix{
\underset{n \geq 0}{\displaystyle\coprod} P_{i}^{n}(\eta \circ g, f)/\Sigma_n \ar@<.5ex>[r]_{\alpha_X} \ar@<1.5ex>[r]^{\mu} & \underset{n \geq 0}{\displaystyle\coprod} P_{i}^{n} (g, f)/\Sigma_n \ar@<1ex>[r] & P_i Y.
}
\]
The pushout squares that inductively describe $P_i^n(\eta \circ g, f)$ and $P_i^n(g, f)$ give a pushout square of coequalizer diagrams that inductively describe the filtration $P_i Y$.  Using the coequalizer $\bC \bC X \rightrightarrows \bC X \arr X$, we see that the induced pushout diagram of coequalizers is of the form in the statement of the lemma.  The claim about $h$-cofibrations follows from Lemma \ref{boxempty_is_hcofibration} below.
\end{proof}

We will apply the filtrations $\{P_i^{n}(f, g)\}$ and $\{P_nY\}$ in situations where $f$ is a coproduct of maps in $F_{+}I$:
\[
f = \coprod_{\alpha} F_{d_{\alpha}} i_{\alpha} \colon \coprod_{\alpha} F_{d_{\alpha}} S^{q_{\alpha}} \arr \coprod_{\alpha} F_{d_{\alpha}} D^{q_{\alpha} + 1}.
\]
The iterated pushout product $f^{\boxempty n} \colon Q^n f \arr B^{\boxtimes n}$ then takes the form
\addtocounter{theorem}{1}
\begin{equation}\label{pushout_product_map_description}
f^{\boxempty n} = \coprod_{(\alpha_1, \dotsc, \alpha_n)} F_{d_{\alpha_1}\oplus \dotsm \oplus d_{\alpha_n}} (i_{\alpha_1} \boxempty \dotsm \boxempty i_{\alpha_n}) = \coprod_{\_{\alpha}} F_{d_{\_{\alpha}}} (i_{\_{\alpha}})
\end{equation}
where
\[
i_{\alpha_1} \boxempty \dotsm \boxempty i_{\alpha_n} \colon S^1 \sma S^{q_{\alpha_1}} \sma \dotsm \sma S^{q_{\alpha_n}} \arr D^{q_{\alpha_1} + 1} \times \dotsm \times D^{q_{\alpha_{n}} + 1}
\]
is the iterated pushout product of the maps $i_{\alpha} \colon S^{q_{\alpha}} \arr D^{q_{\alpha} + 1}$ in the category of spaces.  The coproduct runs over sequences $\_{\alpha} = (\alpha_1, \dotsc, \alpha_n)$ and the symmetric group $\Sigma_n$ acts on $f^{\boxempty n}$ by permuting the sequences.  

\begin{lemma}\label{boxempty_is_hcofibration}  If $f$ is a coproduct of generating positive cofibrations in $F_{+} I$, then the iterated pushout product $f^{\boxempty n}$ is a $\Sigma_n$-equivariant $h$-cofibration of $\sD$-spaces.
\end{lemma}
\begin{proof} (compare \cite{ss}*{7.1})
The orbit of the summand of $f^{\boxempty n}$ indexed on the sequence
\[
\_{\alpha} = (\underbrace{\alpha_1, \dotsc, \alpha_1}_{n_1}, \dotsc, \underbrace{\alpha_j, \dotsc, \alpha_j}_{n_j})
\]
takes the form:
\[
\Sigma_n \times_{\Sigma_{n_1} \times \dotsm \times \Sigma_{n_j}} F_{d_{\_{\alpha}}} (i_{\_{\alpha}}) = \Sigma_n \times_{\Sigma_{n_1} \times \dotsm \times \Sigma_{n_j}} F_{d_{\alpha_1}^{\oplus n_1} \oplus \dotsm \oplus d_{\alpha_j}^{\oplus n_j}} (i_{\alpha_1}^{\boxempty n_1} \boxempty \dotsm \boxempty i_{\alpha_j}^{\boxempty n_j}).
\]
The iterated pushout products of the map $i_{\alpha}$ are $h$-cofibrations of topological spaces, and the functor $F_d$ preserves $h$-cofibrations since it is a left adjoint.  Furthermore, we may choose homotopy extensions along $i_{\alpha}^{\boxempty k}$ to be $\Sigma_k$-equivariant, i.e. the map $i_{\alpha}^{\boxempty k}$ is a $\Sigma_k$-equivariant $h$-cofibration.  It follows that the map $F_{d_{\_{\alpha}}} (i_{\_{\alpha}})$ is a $(\Sigma_{n_1} \times \dotsm \times \Sigma_{n_j})$-equivariant $h$-cofibration, and after passage to orbits we see that $f^{\boxempty n}$ is a $\Sigma_n$-equivariant $h$-cofibration.  
\end{proof}

\begin{proposition}\label{symmetric_power_equiv}  Let $X$ be a positive cofibrant $\sD$-space and let $n \geq 1$.  Then $E \Sigma_{n} \times_{\Sigma_{n}}  X$ is also positive cofibrant and the quotient map
\[
q_n \colon E \Sigma_n \times_{\Sigma_n} X^{\boxtimes n} \arr X^{\boxtimes n} / \Sigma_n
\]
is a weak homotopy equivalence.
\end{proposition}
\begin{proof}  The space $E \Sigma_n$ is a $\Sigma_n$-equivariant CW complex constructed with free $\Sigma_n$-cells.  The induced filtration on the inclusion $X \arr E \Sigma_n \times_{\Sigma_n} X$ is by positive cofibrations.

For the second claim, we may assume that $X$ is an $F_{+}I$-cell complex and induct up the cellular filtration.  Suppose that for all $n \geq 0$ the natural transformation $q_n \colon E\Sigma_n \times_{\Sigma_n} (-) \arr (-)/\Sigma_n$ is a weak homotopy equivalence on $X^{\boxtimes n}$ and let $Y$ be the  pushout of the diagram $X \overset{g}{\longleftarrow} A \overset{f}{\arr} B$ where $f = \amalg_{\alpha} f_{\alpha}$ is a coproduct of maps in $F_{+}I$.  Apply the functor $(-)^{\boxtimes n}$ to the diagram, then consider the filtration $P_{i}^{n}(f, g)$ on $Y^{\boxtimes n}$ from Lemma \ref{filtration_on_product_of_pushouts}.  By analysis of the pushout diagram describing $P_{i}^{n}(f, g)$, it suffices to prove that $q_n$ is a weak homotopy equivalence on 
\[
\Sigma_n \times_{\Sigma_{n - i} \times \Sigma_i} X^{\boxtimes(n - i)} \boxtimes Q^{i} f \qquad \text{and} \qquad \Sigma_n \times_{\Sigma_{n - i} \times \Sigma_i} X^{\boxtimes (n - i)} \boxtimes B^{\boxtimes i}.
\]
Using the description of $Q^{i}f$ in \eqref{pushout_product_map_description}, we see that $Q^i f$ is positive cofibrant.  Hence $E\Sigma_i \times_{\Sigma_i} Q^{i} f$ is also positive cofibrant.  There is a $(\Sigma_{n - i} \times \Sigma_i)$-equivariant homotopy equivalence $E \Sigma_n \simeq E \Sigma_{n - i} \times E \Sigma_i$, so by Lemma \ref{times_cofibrants_preserves_equiv}, $q_{n - i}$ induces a weak homotopy equivalence 
\[
E \Sigma_n \times_{\Sigma_{n - i} \times \Sigma_i} X^{\boxtimes(n - i)} \boxtimes Q^{i} f \arr  X^{\boxtimes (n - i)} / \Sigma_{n - i} \boxtimes (E\Sigma_i \times_{\Sigma_i} Q^{i}f).
\]
By analyzing the $\Sigma_i$ orbits of $Q^{i} f$ as in the proof of Lemma \ref{boxempty_is_hcofibration} and making iterated use of Lemma \ref{prep_for_extended_power_prop}, we see that $q_{i}$ induces a level-wise homotopy equivalence on $X^{\boxtimes (n - i)} / \Sigma_{n - i} \boxtimes Q^{i} f$.  It follows that $q_n$ induces a weak homotopy equivalence on $\Sigma_n \times_{\Sigma_{n - i} \times \Sigma_i} X^{\boxtimes(n - i)} \boxtimes Q^{i} f$.  A similar argument shows that $q_n$ induces a weak homotopy equivalence on $\Sigma_n \times_{\Sigma_{n - i} \times \Sigma_i} X^{\boxtimes (n - i)} \boxtimes B^{\boxtimes i}$ as well.
\end{proof}

It is an immediate consequence of the proposition that the functor $\bC$ preserves weak homotopy equivalences between positive cofibrant $\sD$-spaces.  In particular, every map in $\bC K_{+}$ is a weak homotopy equivalence.  It is straightforward to prove that the functor $\bC$ preserves $h$-cofibrations of $\sD$-spaces, as in \cite{EKMM}*{XII.2.3}.  We can now use the same proof as in \cite{MMSS}*{15.9, 15.11} to prove the next lemma, which says that $\bC F_{+} I$ and $\bC K_{+}$ both satisfy the cofibration hypothesis \cite{MMSS}*{5.3}.

\begin{lemma}\label{cofibration_hypothesis_commFCPs}  Let $L$ denote either $\bC F_{+} I$ or $\bC K_{+}$.  
\begin{itemize}
\item[(i)] If $i \colon A \arr B$ is a coproduct of maps in $L$, then in any pushout diagram of commutative FCPs
\[
\xymatrix{ A \ar[d]_{i} \ar[r] & X \ar[d]^{j} \\
B \ar[r] & Y }
\]
the cobase change $j$ is an $h$-cofibration of $\sD$-spaces.
\item[(ii)] The colimit of a sequence of maps of commutative FCPs that are $h$-cofibrations in $\sD \sU$ is their colimit as a sequence of maps in $\sD \sU$.
\end{itemize}
\end{lemma}

Following the proof of \cite{MMSS}*{15.4}, we see that every relative $\bC K_{+}$ cell complex is a weak homotopy equivalence.  Combined with the cofibration hypothesis for $\bC F_{+}I$ and $\bC K_{+}$, this verifies the hypotheses for the model structure lifting result \cite{MMSS}*{5.13}.  This completes the construction of the model structure on commutative FCPs and the proof of Theorem \ref{existence_commFCP_modelstructure_thm}.

We now turn to proving that the adjunction $(\bbP, \bbU)$ between commutative $\bbI$-FCPs and commutative $\cI$-FCPs is a Quillen equivalence.  We will use an inductive argument that is general enough to be useful in a few different circumstances.

\begin{proposition}\label{general_equiv_cofibrant_FCPs_via_induction}  
For $X$ a cofibrant commutative $\bbI$-FCP, the following maps are weak homotopy equivalences:
\begin{itemize}
\item[(i)] the map $\hocolim_{\bbJ} \bbP X \arr \hocolim_{\bbI} \bbP X$ induced by the inclusion of categories $\bbJ \arr \bbI$.
\item[(ii)]  the unit $\eta \colon X \arr \bbU \bbP X$ of the adjunction $(\bbP, \bbU)$.
\end{itemize}
For $X$ a cofibrant commutative $\cI$-FCP, the following maps are weak homotopy equivalences:
\begin{itemize}
\item[(iii)] the map $\xi \colon \bbQ X \arr \bbO X$ induced by a choice of one-dimensional subspace of the universe $U$ (see \S \ref{section_equivalence_cI_spaces_L_spaces}, in partiular Lemma \ref{xi_prop}),
\item[(iv)] the canonical projection $\pi \colon \hocolim_{\cJ} X \arr \colim_{\cJ} X$ from the homotopy colimit to the colimit.
\end{itemize}
\end{proposition}
\begin{proof}  Write $\psi \colon F \arr G$ for any of the four maps.  We will repeatedly use the fact that all of the functors in (i) -- (iv) preserve colimits and tensors with spaces.  First consider the effect of $\psi$ on a $\bC F_{+} I$-cell complex constructed in a single stage of cell attachment.  Consider the pushout diagram of commutative FCPs
\[
\xymatrix{
\bC A \ar[d]_{\bC f} \ar[r] & \ast \ar[d]^{\overline{f}} \\
\bC B \ar[r] & X}
\]
where $f = \coprod f_{\alpha}$ is a coproduct of generating cofibrations $f_\alpha \colon F_{\bm_\alpha} S^{q_\alpha} \arr F_{\bm_\alpha} D^{q_\alpha + 1}$ in $F_{+}I$.  Consider the filtration $\{P_n X\}$ of $\overline{f}$ given by Lemma \ref{filtration_on_commutative_monoid_pushout}.  The diagram space $P_n X$ is the pushout of the diagram
\[
P_{n - 1} X \longleftarrow Q^{n}f / \Sigma_n  \xrightarrow{f^{\boxempty n}/\Sigma_{n}}  B^{\boxtimes n} / \Sigma_n 
\]
Assume inductively that $\psi$ is a weak equivalence on $P_{n - 1} X$.  Since $f^{\boxempty n}/\Sigma_n$ is an $h$-cofibration, it suffices to prove that $\psi$ is a weak homotopy equivalence on the middle and right entries.

Consider the following commutative diagram
\[
\xymatrix{
F(E \Sigma_n \times_{\Sigma_n} Q^n f) \ar[r]^{\psi} \ar[d]_{F q_n } & G(E \Sigma_n \times_{\Sigma_n} Q^n f) \ar[d]^{G q_n} \\
F(Q^n f / \Sigma_n) \ar[r]_{\psi} & G(Q^n f / \Sigma_n) }
\]
By the description of $Q^n f$ in \eqref{pushout_product_map_description}, we see that $Q^n f$ is a positive cofibrant diagram space.  Hence $E \Sigma_n \times_{\Sigma_n} Q^n f$ is also positive cofibrant by Proposition \ref{symmetric_power_equiv}.  Each of the four natural transformations under consideration is a weak homotopy equivalence on positive cofibrant diagram spaces (Proposition \ref{eta_equiv}, Lemma \ref{xi_prop}, and Lemma \ref{cofibrant_colim}).  Thus the top map in the diagram is a weak homotopy equivalence.  We will now prove that the vertical maps in the diagram are weak homotopy equivalences in each of the three cases.

(i)  The map $q_n$ is a weak homotopy equivalence by Proposition \ref{symmetric_power_equiv}.  Since $\bbP$ commutes with colimits and is strong symmetric monoidal, there is a natural isomorphism between $\bbU \bbP q_n$ and the map
\[
\bbU q_n \colon \bbU( E \Sigma_n \times_{\Sigma_n} Q^n (\bbP f)) \arr \bbU Q^n (\bbP f) / \Sigma_n.
\]
As observed in the proof of Proposition \ref{symmetric_power_equiv}, $q_n$ is in fact a level-wise homotopy equivalence on $Q^n \bbP f$.  Since $\bbU$ preserves level-wise homotopy equivalences, the claim follows.

(ii)  Since $q_n$ is a level-wise homotopy equivalence, both $\hocolim_{\bbJ} q_n$ and $\hocolim_{\bbI} q_n$ are weak homotopy equivalences.

(iii) Recall the description of $Q^n f$ in \eqref{pushout_product_map_description} and set $K = S^1 \sma S^{q_{\alpha_1}} \sma \dotsm \sma S^{q_{\alpha_n}}$.  From the definition of $\bbQ$ we see that:
\begin{align*}
\bbQ ( Q^{n} f / \Sigma_n ) \cong \ast \times_{\Sigma_n} \coprod_{(\alpha_1, \dotsc, \alpha_n)} \cI_c((V_{\alpha_1} \oplus \dotsm \oplus V_{\alpha_{n}}) \otimes U, U) \times K \\ \intertext{and} 
\bbQ ( E \Sigma_{n} \times_{\Sigma_{n}} Q^{n} f  ) \cong E \Sigma_{n} \times_{\Sigma_n} \coprod_{(\alpha_1, \dotsc, \alpha_n)} \cI_c((V_{\alpha_1} \oplus \dotsm \oplus V_{\alpha_{n}}) \otimes U, U) \times K.
\end{align*}
The map $\bbQ q_n$ induced by projecting $E \Sigma_{n}$ to a point is a homotopy equivalence because the symmetric group acts freely on the coproduct of spaces of isometries.  A similar argument using the definition of $\bbO$ shows that the analogous map $\bbO q_n$ is also a homotopy equivalence.  

(iv) There is a natural isomorphism $\colim_{\cJ} X \cong \bbO X$ (Lemma \ref{O_is_colim}), so we have already proved the claim for the colimit functor.  The case of the homotopy colimit functor is proved in the same way.

Returning to the general case, it now follows that the map $\psi$ is a weak homotopy equivalence on $Q^n f / \Sigma_n$.  A similar argument proves that $\psi$ is a weak homotopy equivalence on $B^{\boxtimes n} / \Sigma_n$.  Thus $\psi$ is a weak homotopy equivalence on $P_n Y$.  Passing to colimits, we have proved that $\psi$ is a weak homotopy equivalence on $\bC F_{+} I$-cell complexes constructed in a single stage of cell attachment.

Now we inductively assume that $\psi$ is a weak homotopy equivalence on $\bC F_{+} I$-cell complexes that can be constructed in $n$ stages, and consider the case of an $\bC F_{+} I$-cell complex $X$ that is constructed in $n + 1$ stages.  Write $X = X_{n} \boxtimes_{\bC A} \bC B$ where $X_n$ is a $\bC F_{+} I$-cell complex constructed in $n$ stages and $\bC f \colon \bC A \arr \bC B$ is induced by a coproduct $f = \coprod f_{\alpha}$ of generating cofibrations in $F_{+} I$.  Following the proof of \cite{MMSS}*{15.9}, we write $X$ as a two sided bar construction $X \cong B(X_{n}, \bC A, \bC T)$ where $T = \coprod_{\alpha} F_{d_{\alpha}} (\ast)$ is a coproduct of free diagram spaces on a point.  This bar construction is proper and all of the functors occuring in (i) -- (iii) preserve geometric realization of simplicial diagram spaces and $h$-cofibrations, so it suffices to prove that $\psi$ is a weak homotopy equivalence on the diagram space of $q$-simplices:
\[
X_{n} \boxtimes (\bC A)^{\boxtimes q} \boxtimes \bC T \cong X_{n} \boxtimes \bC (A \amalg \dotsm \amalg A \amalg T).
\]
This $\bC F_{+}I$-cell complex can be constructed in $n$ stages, so the result follows by the induction hypothesis.
\end{proof}

As a consequence of the weak equivalences (i) and (ii), the proof of Theorem \ref{equivalence_of_I_and_If_spaces} can be extended to prove Theorem \ref{equivalence_of_commutative_I_and_If_spaces}.  Notice that the comparison between $X(0)$ and $\hocolim X$ in the cited proof can be replaced by $X(1)$ and $\hocolim X$ in the case of positive fibrant $X$ by use of Lemma \ref{bokstedt_lemma}.

\appendix
\section{Topological categories, the bar construction and $\sL$-space structures}\label{Topological Categories and the Bar Construction}

We gather here the basic theory of bar constructions and homotopy colimits defined over topological categories.  Much of this material has appeared elsewhere (e.g. \cite{vogt} and references therein), but it will be useful to lay out exactly what we need.  In this paper a topological category does \emph{not} mean a category enriched in topological spaces, but rather a category internal to topological spaces.  Thus a topological category $\sD$ consists of a space of objects $\ob \sD$, a space of morphisms $\mor \sD$, and structure maps
\begin{align*}
s, t &\colon \mor \sD \arr \ob \sD, \\
i &\colon \ob\sD \arr \mor \sD, \\
\circ &\colon \mor \sD \times_{\ob\sD} \mor \sD \arr \mor \sD
\end{align*}
that are appropriately associative and unital.  Notice that $\mor \sD$ is a space over $(\ob \sD)^2$ via $s$ and $t$ and that $\ob \sD$ is a space over $(\ob \sD)^2$ via the diagonal map.  We will further require that $i$ is an $h$-cofibration of spaces over $(\ob \sD)^2$, as holds in all of the examples that we use.  We write $A \times_{\sD} B$ for the pullback $A \times_{\ob \sD} B$ of spaces over $\ob \sD$.  A left $\sD$-module $\sX$ consists of a space $\sX$ along with a map $t \colon \sX \arr \ob \sD$ and an action map
\[
\lambda \colon \mor \sD \times_{\sD} \sX \arr \sX
\]
that is associative and unital.  A right $\sD$-module $\sY$ is the same structure except that we label the structure map by $s \colon \sY \arr \ob \sD$ and $\sD$ acts on the right: 
\[
\rho \colon \sY \times_{\sD} \mor \sD \arr \sY.  
\]
Forgetting the topology on $\ob \sD$, a left $\sD$-module $\sX$ determines a continuous functor $X \colon \sD \arr \sU$ of categories enriched in spaces.  A right $\sD$-module $\sY$ determines a continuous functor $Y \colon \sD^{\op} \arr \sU$.  

\begin{definition}  Let $\sD$ be a topological category, $\sX$ a left $\sD$-module and $\sY$ a right $\sD$-module.  The bar construction $B(\sY, \sD, \sX)$ is the geometric realization of the simplicial space with $q$-simplices defined  by
\[
B_q(\sY, \sD, \sX) = \sY \times_{\sD} \mor \sD \times_{\sD} \dotsm \times_{\sD} \mor \sD \times_{\sD} \sX,
\]
where $\mor \sD$ appears $q$ times.  Insertion of identity arrows via $i \colon \ob \sD \arr \mor \sD$ provides the degeneracy maps and the composition in $\sD$ along with $\lambda$ and $\rho$ provide the face maps.  Our assumption that $i$ is an $h$-cofibration insures that $B_*(\sY, \sD, \sX)$ is proper.

We write $\ast$ for the $\sD$-module given by the identity map $\ob \sD \arr \ob \sD$.  Its underlying functor is constant at the one-point space $\ast$.  When $\sY = \ast$,  the bar construction defines the (topological) homotopy colimit of $\sX$ over $\sD$:
\[
\hocolim_{\sD} \sX = B(*, \sD, \sX).
\]
\end{definition}

First let us record a basic commutation relation.  The bivariance of $\sD(-, -)$ makes $\mor \sD$ the total space of a $\sD$-bimodule that we denote by $\sD$. Let $\sX$ be a left $\sD$-module.  Considering $\sX$ as a constant simplicial space, we have a simplicial map $\epsilon_* \colon B_*(\sD, \sD, \sX) \arr \sX$ defined on $q$-simplices as the $(q + 1)$-fold iteration $\lambda^{q + 1}$ of the left-module structure map.  Its geometric realization $\epsilon$ is a map of left $\sD$-modules.  Given a right $\sD$-module $\sY$, we define a map of right $\sD$-modules $\epsilon \colon  B(\sY, \sD, \sD) \arr \sY$ in a similar way.    The bimodule structure of $\sD$ allows iterated bar constructions $B(\sY, \sD, B(\sD, \sD, \sX))$ and $B(B(\sY, \sD, \sD), \sD, \sX)$ that are canonically isomorphic.

\begin{lemma}\label{appendix_homotopy_diagram}
The following diagram of spaces commutes up to homotopy:
\[
\xymatrix{
B(\sY, \sD, B(\sD, \sD, \sX)) \ar[rr]^{\cong} \ar[dr]_{B(\id, \id, \epsilon)} & &  \ar[dl]^{B(\epsilon, \id, \id)} B(B(\sY, \sD, \sD), \sD, \sX) \\
& B(\sY, \sD, \sX) & }
\]
\end{lemma}
\begin{proof}
The iterated bar constructions are the geometric realizations of the bisimplicial space $B_{*, *}(\sY, \sD, \sD, \sD, \sX)$ with $(p,q)$-simplices:
\[
\sY \times_{\sD} (\mor \sD)^{p} \times_{\sD} \mor \sD \times_{\sD} (\mor \sD)^{q} \times_{\sD} \sX.
\]
Of course the products $(\mor \sD)^p$ are really pullbacks over $\ob \sD$ so that the morphisms are composable.  The order in which the simplicial directions are realized determines the order of iteration of the bar construction.  Both $B(\sY, \sD, B(\sD, \sD, \sX))$ and $B(B(\sY, \sD, \sD), \sD, \sX)$ are canonically isomorphic to the geometric realization of the diagonal simplicial space $d_* B(\sY, \sD, \sD, \sD, \sX)$ with $q$-simplices:
\[
d_qB(\sY, \sD, \sD, \sD, \sX) = B_{q, q}(\sY, \sD, \sD, \sD, \sX).
\]
Under these identifications, the two routes in the diagram are the geometric realizations of the maps of simplicial spaces $f, g \colon d_* B(\sY, \sD, \sD, \sD, \sX) \arr B_*(\sY, \sD, \sX)$ given by $f = \lambda_{\sX}^{q + 1}$ and $g = \rho_{\sY}^{q + 1}$.  We will define a simplicial homotopy from $f$ to~$g$.  

Let $d'_i$, respectively $d^{''}_i$, denote the $i$-th face map of the bisimplicial space in the first ($p$), respectively second ($q$), direction.  We also write $d'_i$ and $d''_i$ for the effect of these maps on the diagonal simplicial space.  Define $h_i \colon d_qB(\sY, \sD, \sD, \sD, \sX) \arr B_{q + 1}(\sY, \sD, \sX)$ by:
\[
h_i = (d'_{0})^{i} d''_{i} \dotsm d''_{q}, \qquad \qquad  \text{$0 \leq i \leq q$.}
\]
Notice that $h_i$ applies the last face map in the second direction $q - i$ times, then applies the first face map in the first direction $i$ times.  In symbols (omitting the objects from the notation), 
\begin{align*}
h_i ( y; \phi'_q, \dotsc, \phi'_1; \phi; &\phi''_q, \dotsc, \phi''_1; x) \\
 &= (x; \phi'_q, \dotsc, \phi'_{i + 1}, (\phi'_{i} \dotsm \phi'_{1} \phi \phi''_{q} \dotsm \phi''_{i + 1}), \phi''_{i}, \dotsc, \phi''_{1}; y).
\end{align*}
It is straightforward to check that $h_i$ defines a simplicial homotopy from $f$ to $g$.
\end{proof}

We can think of the topological bar construction $B(\sY, \sD, \sX)$ as a derived or homotopy coherent version of the tensor product of functors $\sY \otimes_{\sD} \sX$.  The latter is a version of an enriched coend that takes the topology on $\ob \sD$ into account.  We define $\sY \otimes_{\sD} \sX$ as the coequalizer of the last two face maps in the simplicial space giving rise to the bar construction:
\[
\xymatrix{
\sY \times_{\sD} \mor \sD \times_{\sD} \sX \ar@<.5ex>[r]^-{\lambda} \ar@<-.5ex>[r]_-{\rho} & \sY \times_{\sD} \sX \ar[r] & \sY \otimes_{\sD} \sX
}
\]
Notice that there is a canonical quotient map $\pi \colon B(\sY, \sD, \sX) \arr \sY \otimes_{\sD} \sX$.  From this perspective, it is not clear that $\sY \otimes_{\sD} \sX$ agrees with the enriched coend $\int^{d \in \sD} Y(d) \times X(d)$ of the bifunctor $Y \times X$.  The latter is calculated as the coequalizer of the diagram
\[
\xymatrix{
\underset{c, d \in \ob \mathrm{sk} \sD}{\displaystyle\coprod} Y(d) \times \sD(c, d) \times X(c) \ar@<.5ex>[r] \ar@<-.5ex>[r] & \underset{d \in \ob \mathrm{sk} \sD}{\displaystyle\coprod} Y(d) \times X(d) 
}
\]
and does not depend on the topology of $\ob \sD$.  We will now give a general procedure for constructing $\sD$-modules for which the tensor product $ - \otimes_{\sD} -$ agrees with the coend.

\begin{construction}\label{module_from_Dspace_construction}
Let $X$ be a $\sD$-space and let $(\mor \sD)_{d}$ denote the space of morphisms with source $d$, i.e. the pullback of $s \colon \mor \sD \arr \ob \sD$ along the inclusion $d \colon * \arr \ob \sD$.  Define $\sX$ to be the enriched coend $\int^{d \in \sD} (\mor \sD)_{d} \times X(d)$.  
The target map $t \colon (\mor \sD)_d \arr \ob \sD$ induces a map $t \colon \sX \arr \ob \sD$ and the left action of $\mor \sD$ on $(\mor \sD)_d$ defined by the composition map $\circ$ in $\sD$ defines a left $\sD$-module structure on $\sX$.  Similarly, given a functor $Y \colon \sD^{\op} \arr \sU$, we have the right $\sD$-module $\sY = \int^{d \in \sD} Y(d) \times {}_{d}(\mor \sD)$, where ${}_{d}(\mor \sD)$ is the space of morphisms with target $d$.
\end{construction}

\begin{proposition}\label{tensor_product_ok_prop}  Suppose that $X$ is a $\sD$-space and $Y$ is a $\sD^{\op}$-space, and let $\sX$ and $\sY$ be the left and right $\sD$-modules defined as above.  Then $\sY \otimes_{\sD} \sX$ is canonically isomorphic to the enriched coend $\int^{d \in \sD} Y(d) \times X(d)$.  
\end{proposition}
\begin{proof}
Notice that ${}_{d}(\mor \sD) \otimes_{\sD} (\mor \sD)_c \cong \sD(c, d)$.  By taking the product with $Y(d)$ on the left and $X(c)$ on the right, then passing to coends over $c$ and $d$, we see that the coequalizer defining $\sY \otimes_{\sD} \sX$ is isomorphic to:
\[
\int^{d \in \sD} \int^{c \in \sD} Y(d) \times \sD(c, d) \times X(c) \cong \int^{d \in \sD} Y(d) \times X(d)
\]
\end{proof}
\noindent We write $Y \otimes_{\sD} X$ for the tensor product of functors $\sY \otimes_{\sD} \sX$, where $\sY$ and $\sX$ are defined as in Construction \ref{module_from_Dspace_construction}.

Here is a version of Bousfield-Kan's cofinality criterion for topological homotopy colimits:
\begin{lemma}\label{theorem_A_hocolim_version}
Let $F \colon \sC \arr \sD$ be a functor of topological categories and let $\sX$ be a left $\sD$-module.  Suppose that for every object $d \in \ob \sD$, the classifying space of the comma category $(d \downarrow F)$ (considered with the topology inherited from $\sC$ and $\sD$) is contractible.  Then the map of topological homotopy colimits
\[
\hocolim_{\sC} X \circ F \arr \hocolim_{\sD} X
\]
is a homotopy equivalence.
\end{lemma}
This is proved by Meyer's thoroughly general approach to bar constructions \cite{meyerII}.  In particular, it is a special case of \cite{meyerII}*{\S 4.3} with his $X$ our $X \circ F$ and his $X'$ our $X$ and the particular choice of admissible pair given by $\sH$ the homotopy equivalences and $\sM$ all commutative squares.  To verify the condition on $\sH$ in the source, note that $\hocolim_{(d \downarrow F)} \overline{X} = B(d \downarrow F) \times X(d)$ since in this case the functor $\overline{X}$ is the constant functor at $X(d)$.  Thus the map required to be in $\sH$ is the projection $B(d \downarrow F) \times X(d) \arr X(d)$, which is a homotopy equivalence by the assumption on $(d \downarrow F)$.

In the applications of Lemma \ref{theorem_A_hocolim_version}, we will check that $B(d \downarrow F)$ is contractible by showing that $(d \downarrow F)$ has an initial object. 

\begin{remark}\label{top_initial_object_remark} An object $0 \in \ob \sD$ of a topological category is initial if and only if it is initial in the underlying category internal to sets and the map $! \colon \ob \sD  \arr \mor \sD$ sending $d$ to the unique morphism $0 \arr d$ is continuous.  The usual proof that a category with initial object has a contractible classifying space then goes through.  For example, in the case of the continuous functor $F \colon \sD^{\delta} \arr \sD$ from $\sD$ with the discrete topology to $\sD$ with its given topology, the comma category $(d \downarrow F)$ has the initial object $(\id \colon d \arr F(d))$ as an ordinary category.  However, in general the map $!$ will not be continuous, so this is not an initial object of $(d \downarrow F)$ considered as a topological category.  
\end{remark}

The following lemma is \cite{meyerII}*{4.4.1}.  Note that the squares labelled (v) in the cited source are pullbacks and so belong to the class $\sM$. 

\begin{lemma}\label{map_of_hocolim_is_quasifib}  Suppose that $\sE \arr \sB$ is a morphism of left $\sD$-modules such that the map of underlying spaces is a fibration.  Then the induced map of topological homotopy colimits $\hocolim_{\sD} \sE \arr \hocolim_{\sD} \sB$ is a quasifibration.
\end{lemma}

We now record the definition of the topology on functor categories.

\begin{definition}\label{topological_functor_category}  Let $\sP$ and $\sD$ be topological categories.  Define a topology on the category $\Fun(\sP, \sD)$ of functors and natural transformations as follows.  The space of objects is given the subspace topology induced by the inclusion
\begin{align*}
\ob \Fun(\sP, \sD) \subset \sU(\mor \sP, \mor \sD) 
\end{align*}
that send a functor $\theta$ to its effect on morphisms: 
\[
(\phi \colon A \rightarrow B) \mapsto (\theta(\phi) \colon \theta(A) \rightarrow \theta(B)).  
\]
The space of morphisms is given the subspace topology induced by the inclusion
\[
\mor \Fun(\sP, \sD) \subset \sU(\ob \sP, \mor \sD) \times \sU(\mor \sP, \mor \sD )^2.
\]
that sends a natural transformation $\alpha \colon \theta \rightarrow \theta'$ to the triple $((\alpha_{A})_{A}, \theta, \theta')$ consisting of its components $\bigl( \alpha_{A} \colon \theta(A) \rightarrow \theta'(A) \bigr)_{A}$ and its source and target $(\theta(\phi), \theta'(\phi))_{\phi}$.  
\end{definition}

We will now describe the topologies on the categories of isometries used in this paper.  The space $\cI(V, W)$ of linear isometries from $V$ to $W$ is topologized as a subspace of $\sU(V, W)$.  If $V \subset W$, there is a canonical identification $\cI(V, W) \cong O(W)/O(W - V)$.  Infinite dimensional inner product spaces are topologized as the colimit of their finite dimensional subspaces and spaces of isometries $\cI_c(U', U)$ of possibly infinite dimensional inner product spaces are topologized using the compact open topology.  This means that
\[
\cI_c(U', U) = \underset{V' \subset U'}{\lim} \, \underset{V \subset U}{\colim} \, \cI(V', V)
\]
where $V' $ and $V$ run through finite dimensional subspaces.  From this definition it is straightforward to verify that composition in $\cI_c$ is continuous.

Throughout, we use a fixed universe $U$: this is an inner product space isomorphic to $\bR^{\infty}$.  Let $W$ be a real inner product space, finite-dimensional or infinite-dimensional, and consider the category $\cJ(W)$ of finite dimensional sub-inner product spaces $V \subset W$ with morphisms the inclusions $V \subset V'$.  The space of objects is the following disjoint union of Grassmanians:
\begin{align*}
\ob \cJ(W) &= \coprod_{n \geq 0} \cI_c(\bR^n, W) / O(n) 
\end{align*}
The space of morphisms is the space of flags of subspaces of $W$ of length two:
\begin{align*}
\mor \cJ(W) &= \coprod_{0 \leq m \leq n} \cI_c(\bR^{n}, W) / O(m) \times O(n - m).
\end{align*}
More generally, the space $N_q \cJ(W)$ of $q$-simplices in the nerve of $\cJ(W)$ is the space of flags of subspaces of $V$ of length $q + 1$:
\addtocounter{theorem}{1}
\begin{equation}\label{nerve_of_cJ}
\coprod_{0 \leq n_0 \leq \dotsm \leq n_q} \cI_c(\bR^{n_q}, W) / O(n_0) \times O(n_1 - n_0) \times \dotsm \times O(n_q - n_{q - 1}).
\end{equation}
In the case of $W = U$, we have described the topology on $\cJ = \cJ(U)$.  Now consider the category $\cI^{\dagger}$ of finite dimensional inner product spaces $V \subset U^a$ for $a \geq 0$, and linear isometries (not necessarily respecting the inclusion into $U^a$).  The space of objects of $\cI^{\dagger}$ is defined by:
\[
\ob \cI^{\dagger} = \{0\} \amalg \coprod_{a > 0} \coprod_{n > 0} \cI_c(\bR^n, U^a)/O(n).
\]
Notice that the only zero dimensional object is $0 \subset U^0$.  The inner product spaces $0 \subset U^a$ for $a > 0$ do not appear in this category.  The space of morphisms is:
\[
\mor \cI^{\dagger} = \ob \cI^{\dagger} \amalg \coprod_{a, b > 0} \coprod_{0 < m \leq n} \cI_c(\bR^n, U^b) \times_{O(n)} \cI(\bR^m, \bR^n) \times_{O(m)} \cI_c(\bR^m, U^a).
\]
Here the copy of $\ob \cI^{\dagger}$ represents the space of maps $0 \arr V$, and a point $[f, \phi, g]$ of the other summand corresponds to the morphism $f \circ \phi \circ g^{-1} \colon \mathrm{Im}(g) \arr \mathrm{Im}(f)$.  The source and target maps are defined by projecting to the third and first factors, respectively.  Composition is defined using the composition in the middle factor.  The direct sum taking $V \subset U^a$ and $W \subset U^b$ to $V \oplus W \subset U^{a + b}$ is well-defined and continuous on both $\ob \cI^{\dagger}$ and $\mor \cI^{\dagger}$.  It follows that $\cI^{\dagger}$ is a permutative topological category under direct sum, i.e. a symmetric monoidal topological category whose unit and associativity isomorphisms are identity maps.  There is a canonical inclusion of topological categories $\cJ \arr \cI^{\dagger}$ which is the identity on objects.  On morphisms, it sends an inclusion $V \subset W$ in $U$ to the point $[f_W, \iota, f_V]$ of the $a = b = 1$ summand, where $\iota$ is the canonical inclusion $\bR^m \subset \bR^n$ and $f_W \colon \bR^n \arr U$ and $f_V \colon \bR^m \arr U$ are representatives for $W$ and $V$ chosen such that $f_W \circ \iota = f_V$.  

Notice that the functor $\cI^{\dagger} \arr \cI$ sending $V \subset U^a$ to $V$ is an equivalence of categories, even though it is not injective on objects.  Throughout the paper, whenever we form a bar construction involving $\cI$, we implicitly use the category $\cI^{\dagger}$ in place of $\cI$ by precomposing functors with domain $\cI$ along the equivalence $\cI^{\dagger} \arr \cI$.

We have used one more category of isometries.  Let $\cI(W)$ be the full subcategory of $\cI^{\dagger}$ with objects the finite dimensional sub-inner product spaces $V \subset W$.  Thus $\ob \cI(W) = \ob \cJ(W)$ and $\mor \cI(W)$ is topologized as a subspace of $\mor \cI^{\dagger}$.  Notice that $\cI(U)$ is \emph{not} the same category as $\cI$, and that the inclusion $\cJ \arr \cI^{\dagger}$ factors as $\cJ \arr \cI(U) \arr \cI^{\dagger}$.

In order to take topological homotopy colimits of an $\cI$-space $X$ over each of the categories $\cJ(W), \cI(W),$ and $\cI^{\dagger}$, we use Construction \ref{module_from_Dspace_construction} to define three left modules associated to $X$, one over each of these categories.  We write $\sX(\sD)$ for the left $\sD$-module associated to $X$, where $\sD = \cJ(W), \cI(W),$ or $\cI^{\dagger}$.  By writing the coend that defines $\sX$ as a coequalizer, we find that:
\begin{align*}
\sX(\cI(W)) &= \sX(\cJ(W)) = \coprod_{n \geq 0} \cI_c(\bR^n, W) \times_{O(n)} X(\bR^n) \\
\sX(\cI^{\dagger}) &= X(0) \amalg \coprod_{a > 0} \coprod_{n > 0} \cI_c(\bR^n, U^a) \times_{O(n)} X(\bR^n).
\end{align*}
In each of the three cases the structure map $\sX(\sD) \arr \ob \sD$ collapses $X(\bR^n)$ to a point.  These are all $O(n)$-bundles and in each case we identify the fiber over an $n$-plane $V$ with the space $X(V)$.

We record here the following generalization of \eqref{nerve_of_cJ} that describes the $q$-simplices of the homotopy colimit of $\sX(\cJ(W))$:
\addtocounter{theorem}{1}
\begin{equation}\label{simplices_of_cJ_X}
B_q(*, \cJ(W), \sX) = \coprod_{0 \leq n_0 \leq \dotsm \leq  n_q} \cI_c(\bR^{n_q}, W) \times_{O(n_0) \times O(n_1 - n_0) \times \dotsm \times O(n_q - n_{q - 1})} X(\bR^{n_0}).
\end{equation}

We will now construct actions by the linear isometries operad $\sL$ \cite{E_infty_rings}.  The $j$-th space of $\sL$ is $\sL(j) = \cI_c(U^j, U)$, and an $\sL$-space is an algebra over $\sL$ in spaces.  An $\sL$-category is a topological category $\sD$ such that $\ob \sD$ and $\mor \sD$ are $\sL$-spaces and such that the category structure maps are maps of $\sL$-spaces.  A left $\sL \sD$-algebra $\sX$ is a left $\sD$-module $\sX$ over an $\sL$-category such that $\sX$ is an $\sL$-space and the structure maps $t \colon \sX \arr \ob \sD$ and $\lambda \colon \mor \sD \times_{\ob \sD} \sX \arr \sX$ are maps of $\sL$-spaces.  Right $\sL \sD$-algebras are defined similarly.  We also have variants of all these notions for the operad $\bbL$ with $\bbL(0) = *$, $\bbL(1) = \sL(1)$ and $\bbL(j) = \emptyset$ for $j > 1$.  An algebra over $\bbL$ in spaces is the same thing as an $\bbL$-space as defined in \S\ref{Infinite loop space theory of $S$-modules_section}.  We are interested in these structures because the bar construction of $\sL \sD$-algebras is an $\sL$-space.

\begin{lemma}\label{barconstruction_Lspace_lemma}
If $\sD$ is an $\sL$-category and $\sX$ and $\sY$ are left and right $\sL \sD$-algebras, then the bar construction $B(\sY, \sD, \sX)$ is an $\sL$-space.  The analogous statement holds with $\sL$ replaced by $\bbL$.
\end{lemma}
\begin{proof}
Since all spaces and maps involved are $\sL$-spaces and maps of $\sL$-spaces, the simplicial bar construction $B_*(\sY, \sD, \sX)$ is a simplicial $\sL$-space.  Thus its geometric realization is an $\sL$-space as well.
\end{proof}

\begin{proposition}\label{hocolim_of_FCP_is_Lspace_prop}  Let $X$ be an $\cI$-space.  Then:
\begin{itemize}
\item[(i)]  $\cJ$ and $\cI(U)$ are $\sL$-categories.
\item[(ii)]  $\sX(\cJ)$ and $\sX(\cI(U))$ are left $\bbL \sD$-algebras.
\item[(iii)]  $\hocolim_{\cJ} X$ and $\hocolim_{\cI(U)} X$ are $\bbL$-spaces and the map from the former to the latter induced by $\cJ \arr \cI(U)$ is a map of $\bbL$-spaces.
\end{itemize}
Suppose further that $X$ is a commutative $\cI$-FCP.  Then:
\begin{itemize}
\item[(iv)]  $\sX(\cJ)$ and $\sX(\cI(U))$ are left $\sL \sD$-algebras.
\item[(v)]  $\hocolim_{\cJ} X$ and $\hocolim_{\cI(U)} X$ are $\sL$-spaces and the map from the former to the latter induced by $\cJ \arr \cI(U)$ is a map of $\sL$-spaces.
\end{itemize}
\end{proposition}
\begin{proof}
(iii) and (v) follow from (ii) and (iv), respectively, using the preceding lemma.  We will start with (iv) and prove (i) along the way.

We define the action of $\sL$ on the space $\sX(\cJ) = \sX(\cI(U))$ as the composite
\addtocounter{theorem}{1}
\begin{gather}\label{main_Laction_diagram}
\xymatrix{
\cI_c(U^j, U) \times \displaystyle\prod_{i = 1}^{j} \cI_c(\bR^{n_i}, U) \times_{O(n_i)} X_{n_i}  \ar[d] \\
  \cI_c(U^j, U) \times \cI_c(\bR^n, U^j) \times_{O(n)} X_{n} \ar[d]^{\circ} \\
  \cI_c(\bR^n, U) \times_{O(n)} X_{n}
}
\end{gather}
of the map
\[
(\gamma ; [f_1, x_1], \dotsc, [f_j, x_j]) \longmapsto [\gamma, (f_1 \oplus \dotsm \oplus f_j), \mu(x_1, \dotsc, x_j)],
\]
followed by composition in $\cI_c$.  Here $n = n_1 + \dotsm + n_j$ and we omit from the notation the canonical isomorphism $\bR^{n_1} \oplus \dotsm \oplus \bR^{n_j} \cong \bR^{n}$.  The map $\mu \colon X_{n_1} \times \dotsm \times X_{n_j} \arr X_{n}$ is the FCP multiplication of $X$.    Notice that the $\Sigma_j$-equivariance follows from the commutativity of the FCP $X$.  The $\sL$-space structure on $\ob \cJ = \ob \cI(U)$ is the case of $X(V) = *$, and so the projection 
\[
t \colon \sX(\cJ) = \sX(\cI(U)) \arr \ob \cJ = \ob \cI(U)
\]
is a map of $\sL$-spaces.  

We now define the action of $\sL$ on
\[
\mor \cI(U) = \coprod_{0 \leq m \leq n} \cI_c(\bR^n, U) \times_{O(n)} \cI(\bR^m, \bR^n) \times_{O(m)} \cI_c(\bR^m, U).
\]
On each summand of the coproduct this is given by:
\begin{gather*}
\xymatrix{ \cI_c(U^j, U) \times
\displaystyle\prod_{i = 1}^{j} \cI_c(\bR^{n_i}, U) \times_{O(n_i)} \cI(\bR^{m_i}, \bR^{n_i}) \times_{O(m_i)} \cI_c(\bR^{m_i}, U)  \ar[d] \\
 \cI_c(\bR^{n}, U) \times_{O(n)} \cI(\bR^{m}, \bR^{n}) \times_{O(m)} \cI_c(\bR^{m}, U) }
 \\
(\gamma ; [f_1, \phi_1, g_1], \dotsc, [f_j, \phi_j, g_j]) \longmapsto [\gamma \circ \bigoplus_i f_i, \bigoplus_i \phi_i, \gamma \circ \bigoplus_i g_i ].
\end{gather*}
If $f_i$ and $g_i$ represent subspaces $V_i \subset U$ and $W_i \subset U$, then the action of $\gamma \in \sL(j)$ takes the $j$-tuple of isometries $(\phi_1 \colon V_1 \arr W_1, \dotsc, \phi_j \colon V_j \arr W_j)$ to the isometry:
\[
\gamma \circ \bigl( \bigoplus_i \phi_i \bigr) \circ \gamma^{-1} \colon \gamma \bigl( \bigoplus_i V_i \bigr) \arr \gamma \bigl( \bigoplus_i W_i \bigr).
\]
It is clear that this action stabilizes the subspace $\mor \cJ \subset \mor \cI(U)$ of inclusions, so $\mor \cJ$ is an $\sL$-space as well.  For both $\cJ$ and $\cI(U)$, it is immediate that the source, target and composition maps preserve the $\sL$-action.  This finishes the proof of (i).  It is also straightforward to check that the left module structure map $\lambda$ for $\sX(\cJ)$ and $\sX(\cI(U))$ is a map of $\sL$-spaces, so we have proved (iv) as well.   

To get the action of $\bbL$, specialize all of the $\sL$-actions to the case $j = 1$.  This does not depend on the FCP structure of $X$, so (ii) follows.
\end{proof}

\end{document}